\documentclass[a4paper,11pt,reqno]{article}

\usepackage[T1]{fontenc}
\usepackage[utf8]{inputenc}
\usepackage[a4paper, scale=0.8, centering]{geometry}
\usepackage{hyperref}
\usepackage{amsfonts,amssymb,amsmath,amsthm}
\usepackage{mathtools,mathrsfs,mathabx,bbm,nicefrac, bbm}

\usepackage{color, xcolor}

\usepackage{subcaption}
\usepackage{graphicx}
\usepackage[all]{xy}

\usepackage{enumerate}

\newcommand{\bbE}{{\ensuremath{\mathbb E}} }

\newcommand{\bbN}{{\ensuremath{\mathbb N}} }

\newcommand{\bbP}{{\ensuremath{\mathbb P}} }

\newcommand{\bbR}{{\ensuremath{\mathbb R}} }

\newcommand{\bbT}{{\ensuremath{\mathbb T}} }

\newcommand{\bbZ}{{\ensuremath{\mathbb Z}} }

\newcommand{\cA}{{\ensuremath{\mathcal A}} }

\newcommand{\cF}{{\ensuremath{\mathcal F}} }
\newcommand{\cG}{{\ensuremath{\mathcal G}} }
\newcommand{\cH}{{\ensuremath{\mathcal H}} }

\newcommand{\cJ}{{\ensuremath{\mathcal J}} }
\newcommand{\cK}{{\ensuremath{\mathcal K}} }
\newcommand{\cL}{{\ensuremath{\mathcal L}} }

\newcommand{\cV}{{\ensuremath{\mathcal V}} }

\newcommand{\ga}{\alpha}

\newcommand{\gu}{\upsilon}

\renewcommand{\epsilon}{\varepsilon}
\renewcommand{\theta}{\vartheta}
\renewcommand{\phi}{\varphi}

\newcommand{\sumtwo}[2]{\sum_{\substack{#1 \\ #2}}}

\renewcommand{\tilde}{\widetilde}

\newcommand{\ind}{\mathbbm{1}}
\newcommand{\dd}{{\ensuremath{\mathrm d}} }

\renewcommand{\1}[1]{\mathbbm{1}_{\{ #1 \}}}

\newcommand{\leftmean}[1]{\overset{\xleftarrow{\phantom{#1}}}{#1}}
\newcommand{\rightmean}[1]{\overset{\xrightarrow{\phantom{#1}}}{#1}}

\newcommand{\Tn}{\bbT^d_N}

\newtheorem{assumption}{Assumption}
\newtheorem{theorem}{Theorem}[section]
\newtheorem{definition}{Definition}[section]
\newtheorem{lemma}[theorem]{Lemma}
\newtheorem{corollary}[theorem]{Corollary}
\newtheorem{proposition}[theorem]{Proposition}
\newtheorem{remark}{Remark}[section]

\numberwithin{equation}{section}

\title{Hydrodynamic limit of the Schelling model with spontaneous Glauber and Kawasaki dynamics}

\author{Florent BARRET
\thanks{MODAL'X, UMR 9023, UPL, Univ. Paris Nanterre, F92000 Nanterre France. \\Email:
\emph{florent.barret@parisnanterre.fr} and \emph{niccolo.torri@parisnanterre.fr}}
\and
Niccol\`o TORRI
\footnotemark[1]
}

\date{}

\begin{document}

\maketitle

\begin{abstract}
In the present article we consider the Schelling model, an agent-based model describing a segregation dynamics when we have a cohabitation of two social groups. As for several social models, the behaviour of the Schelling model was analyzed along several directions, notably by exploiting theoretical physics tools and computer simulations. 
This approach led to conjecture a phase diagram in which either different social groups were segregated in two large clusters or they were mixed. 
In this article, we describe and analyze a perturbation of the Schelling model as a particle systems model by adding a Glauber and Kawasaki dynamics to the original Schelling dynamics. 
 As far as the authors know, this is the first rigorous mathematical analysis of the perturbed Schelling model.
We prove the existence of an hydrodynamic limit described by a reaction-diffusion equation with a discontinuous non-linear reaction term. The existence and uniqueness of the solution is non trivial and the analysis of the limit PDE is interesting in its own.
Based on our results, we conjecture, as in other variations of this model, the existence of a phase diagram in which we have a mixed, a segregated and a metastable segregation phase.
 We also describe how this phase transition can be viewed as a transition between a relevant and irrelevant disorder regime in the model.
\\[0.2cm]
	\\[0.05cm]
\end{abstract}

\section{Introduction}
Schelling's model of segregation was introduced by Thomas Schelling in 1971 \cite{S71, S0678}.
The original model is defined on a square grid of $N^2$ sites (or, more generally, on a regular graph with $N$ sites) where agents (individuals) belonging to two groups are disposed. 
Each agent located at a given site of the grid compares its group with the group of its neighbors. More precisely, we fix a tolerance threshold $T\in[0,1]$. We call $r_x$ the fraction of neighbors belonging to the agent's group at site $x$ and we say that the agent is satisfied if $r_x\ge T$. If the agent is unsatisfied, then he \emph{moves} on a site that makes him satisfied. What we mean with ``moving'' depends on the precise dynamics defined on the grid. 
If some sites are assumed to be empty, then an unsatisfied agent moves on the nearest empty site which makes him satisfied. 
If all the sites are occupied, then either we swap the position between two unsatisfied agents if the swapping makes both of them satisfied (the so called \textit{Kawasaki-Schelling dynamics}), or we change (flip) the group of an unsatisfied agent if this operation makes him satisfied (\textit{Glauber-Schelling dynamics}). In the first case we say that the system is closed, while in the second case the system is open, since this operation can be seen as a swapping with the outside. 

Several variations of this model exists, and these variations depend on several parameters  \cite{S71}:
 \begin{enumerate}[(1)]
	\item the neighborhood (its size, its geometry),
	\item the initial distribution of the agents,
	\item the choice of the satisfaction condition (e.g. the value of the tolerance parameter, or one could introduce a different tolerance for each group),
	\item the local dynamics between agents.
\end{enumerate}
In the original model, some sites are assumed to be empty. Several variants of Schelling's model have been considered in the recent literature in order to study the behaviour of the model when the fundamental parameters are modified. We refer to \cite{AL22PhD} for a complete overview on the subject.
 Among the different variations, let us mention that there can be more than two groups of agents \cite{HS20}, or/and that the Schelling dynamics can be perturbed: each site has a positive probability to switch regardless of its satisfaction  (spontaneous Glauber and/or Kawasaki dynamics) \cite{BMR14, MSS07, SS07}. 

The unperturbed model has attracted the interest of the mathematical community over the last 15 years \cite{BEL15, BEL16, BEL18, BIKK12, DH21, IKLZ17, HS20}.

The main concern is the behaviour of the model for large times: does the model reach a stationary state~? a stationary distribution~? If so what are the features of this equilibrium~?

A common result of the considered variants is the existence of three stationary states separated by two critical thresholds $T_f$ and $T_c$ 
towards which the system evolves, suggesting an \textit{universal} behaviour of the model.
 If $T_f<T<T_c$ the different groups merge into clusters and we observe the appearance of at least two macroscopic clusters (segregation), while if $T<T_f$ or $T>T_c$ 
we do not observe the appearance of two macroscopic clusters, see \cite{AL22PhD}.


In the present paper, we approach the model from a physical point of view, by interpreting the agent dynamics as a particle systems in interaction.
This approach was adopted by the physical community to study this model, see for instance \cite{CFL09, GVN09, OF18, OF18-1}. 
In particular we consider the setting where the points of the grid (a discrete torus $\Tn:=\left(\mathbb{Z}/N\mathbb{Z}\right)^d$, $d\ge 1$) are fully occupied and unsatisfied agents flip their state if it makes them satisfied (Glauber-Schelling dynamics). The size of the neighborhood taken into account to compute the fraction $r_x$, grows at most logarithmically with $N$.  Let us stress that the Glauber-Schelling dynamics was consider in the physics literature from a computer simulation viewpoint with a fixed size of the neighborhood, see e.g. \cite{OF18-1}.

Moreover we introduce random perturbations, either by flipping a state of an agent at rate $\beta$ (spontaneous Glauber dynamics) or exchanging the position of two agents at rate $\alpha N^2$ (accelerated spontaneous Kawasaki dynamics).

To summarize, we assume the following features:
\begin{enumerate}[(1)]
	\item the neighborhood size  used to compute the fraction $r_x$, is going to infinity with $N^d$, the number of sites,
	\item the initial distribution of the agents is fixed (deterministic) and converges as $N$ goes to infinity,
	\item we fix the tolerance parameter $T\in[0,1]$,
	\item we introduce two random perturbations of the  Glauber-Schelling mechanism: regardless of their satisfaction, a site can change type (spontaneous Glauber dynamics), and a site can swap type with a closest-neighbor (spontaneous accelerated Kawasaki dynamics).
\end{enumerate}

 From a statistical physics perspective, the main question concerns the impact of the random perturbations on the system behaviour: is there a phase transition (in the parameter $\beta$ tuning the spontaneous Glauber dynamics) between a phase where the disorder supersedes the behaviour of the model and a phase where the mechanism of the unperturbed model drives the behaviour of the system ?

Our main result (Theorem \ref{thm:convergence}) proves, by rescaling the space as $\frac1N$, an hydrodynamic limit. The limit is described by a reaction-diffusion equation and we give a complete description of the limit PDE that we get. The assumption about random perturbations (4) is important to ensure the existence of a diffusive term and that all the configurations are accessible, which is fundamental in the theory of the hydrodynamic limit, \cite{KL99}.

In the case where the size of the neighborhood stays finite in the limit, we obtain a classical reaction-diffusion equation.  This is the case where the size of the interaction term stays finite and thus microscopic. However, when the size of the neighborhood goes to infinity we get a non-linearity (the reaction term) which is discontinuous at two points. In this case, the interaction of the  Glauber-Schelling dynamics takes into account more and more agents but in the limit, the reaction term is still purely local but  discontinuous. 
In this ``mesoscopic'' limit, the existence and uniqueness of the solution of the reaction-diffusion equation with discontinuities is one of  the major points of the paper. Moreover, it is not a mere technical problem since the limiting equation does not have a unique solution for some class of initial condition, and some values of $\beta$, the parameter tuning the spontaneous Glauber dynamics. 

Finally, we conjecture the existence of a rich phase diagram in which, beyond a  disordered (and mixed) phase, where the spontaneous Glauber dynamics dominates and an ordered (and segregated) phase, where the Glauber-Schelling dynamics dominates, there is a transition in between. In this phase, we expect the system to show a metastable behaviour: the mixed and segregation phases coexist and depending on the parameters, one of them should be the most stable one and dominates the long-time behaviour of the system. Critical points depend on the parameters $\beta$ and $T$ but not on $\alpha$, see Figure \ref{fig1}. A rigorous proof of the phase diagram will require a delicate analysis of the local dynamics that goes beyond the techniques used in the present paper. We reserve this for future work.

Let us stress that in the disordered phase, based on Remark \ref{rem:gradflow}, we expect to have a mixed configuration since we conjecture that a typical configuration looks like a Bernoulli distribution of parameter $p=\frac12$ on each site, while in the ordered phase the parameter is $p=p_{T, \beta}\neq \frac12$. This means that a very large part of the configuration is either $0$ or $1$ and the appearance of clusters is possible. 
With our method, we are not able to predict the precise geometry of the clusters, even if we expect segregation in this phase, see Section \ref{sec:conj}.

In \cite{HS20}, the authors prove also a convergence of a discrete model of Schelling dynamics to the solution of a reaction equation (without diffusion) baptized a continuous Schelling dynamics. We point out that the model is quite different since, in their work, the authors consider a macroscopic neighborhood (which gives at the limit, an integro-differential equation),  do not assume any spontanous random perturbation (either Glauber or Kawasaki) and consider a model with $M\geqslant 2$ groups. Also, the authors consider a fixed tolerance parameter of $T=\frac12$ and assume that the initial configuration is given by random independent uniform variables. The proof of the convergence is based on a coupling between the discrete and the continuous Schelling dynamics.
\medskip

Our method of proof is based on the technique of the relative entropy method in the framework developed by Jara and Mezenes in \cite{JM18, JM20} and also used by Funaki and Tsunoda in \cite{FT19} for a finite number of particle in the interaction. However, in our setting, we need to improve their bounds to cover the case where the number of particles in the interaction is going to infinity. More precisely, in order to use the relative entropy method, a central step is the control of $\partial_t \cH_N(t)$, the derivative in time of the relative entropy between the law of the process $\mu_N(t)$ and a discrete measure which approximates the density solution of the reaction-diffusion equation, see Proposition \ref{pro:YauInequality} and Equation \eqref{eq:goalJ_t}. Since the number of particles in the interaction term grows with $N$, the size of the system, we need to retrace the bounds obtained in \cite{JM18, JM20} and  \cite{FT19} by taking into account the size of the interactions. This is done in Theorem \ref{prop:UB_V^+}. With our bound \eqref{eq:goalJ_t}, we get that as soon as the diameter of the interaction grows at most as $\delta(\log(N))^{1/d}$ (see Assumption \ref{ass1}), the relative entropy is $O(N^{d-\epsilon})$ for some $\epsilon >0$ (see Equation \eqref{eq:control-entropy}) which entails that the empirical measure is close in probability to a deterministic discrete process defined by Equation \eqref{eq:discretPDE}.

To complete the proof of the main result (Theorem \ref{thm:convergence}), we also prove that this deterministic discrete process, defined by Equation \eqref{eq:discretPDE}, converges to a solution of a limiting reaction-diffusion equation (Equation \eqref{eq:hydrlimit}). This is done in two steps: we first prove that the limiting PDE has a solution (in Proposition \ref{prop:existence_sol_PDE}), and for some class of initial conditions, this solution is locally unique (in Proposition \ref{prop:v_t-v_0}). The existence result is done via an approximating sequence of smooth non-linearities which are natural in our framework (defined by Equation \eqref{eq:limitePDE-Kter}). 
Note that we do not use the deterministic process defined by Equation \eqref{eq:discretPDE} which is discrete in space.
The second step is therefore to prove that the deterministic process, defined by Equation \eqref{eq:discretPDE}, has accumulation points in a uniform norm on compact set which are all solutions of Equation \eqref{eq:hydrlimit} (this is Theorem \ref{thm:discrconv}). If we have uniqueness for solutions of  Equation \eqref{eq:hydrlimit}, we have the main result.

 We establish local uniqueness for  \eqref{eq:hydrlimit} only for a class of initial conditions (in Proposition \ref{prop:v_t-v_0}), we use and adapt arguments of Gianni \cite{gianni} and Deguchi \cite{deguchi} (which proves existence and uniqueness with only one discontinuity). Note that for some initial conditions, \eqref{eq:hydrlimit} does not have a unique solution, see Remark \ref{rem:non-unique} for a simple concrete example. It would be therefore quite interesting to understand if for such initial conditions, the empirical measure process converges in some sense. We also reserve this for future work.

Note also that, still in \cite{HS20}, the continuous Schelling dynamics does not have a unique solution for all initial conditions. However, starting from a random Gaussian field, the authors prove the solution exists and is a.s. unique.

A detailed plan of the method of proof and the article is given at the end of Section \ref{sec:results} containing the main result and assumptions. In the following Section, we define the model.

\subsection*{Acknowledgments}
 The authors thank \textit{Antoine Lucquiaud} and \textit{Alessandra Faggionato} for the fruitful discussion on the model and on the article.  The authors also thank \textit{Oriane Blondel} for pointing out reference \cite{HS20}.
This research has been conducted within the FP2M federation (CNRS FR 2036) and as part of the project Labex MME-DII (ANR11-LBX-0023-01).

The authors thank the referees for their helpful suggestions.

\section{The model}

\subsection{Configurations}
For $N\in \mathbb N=\{1,2,3, \ldots\}$ we let $\Tn=(\mathbb{Z}/N\mathbb{Z})^d$ be the discrete torus and let $\Omega_N=\{0,1\}^{\Tn}$ be the space of all possible configurations. 
We call $\eta\in\Omega_N$ a configuration and $i\in\Tn$ a site. 
Let $\mathcal{V}_N\subset\Tn\setminus\{0\}$ be a 
subset of the discrete torus with a diameter which can grow with $N$ (see Assumption \ref{ass1} for the precise hypothesis on the geometry of $\cV_N$). We say that two sites $i,j\in \Tn$ are neighbors if $i-j\in \mathcal{V}_N$. We denote $K_N=|\mathcal V_N|$ its cardinality. For a configuration $\eta\in\Omega_N$ and a site $i\in \Tn$
we let
\begin{equation}
r_i(\eta)=\frac1{K_N}\sum_{j\in\mathcal{V}_N}\1{\eta_i=\eta_{i+j}}
\qquad
\text{and}\qquad
\label{def:meanfield}
\rho_i(\eta)=\frac1{K_N}\sum_{j\in\mathcal{V}_N}\eta_{i+j}.
\end{equation}
The quantity $\rho_i(\eta)$ is the mean field of $\eta$ on the neighborhood $\mathcal{V}_N+i$. Let us observe that $\rho_i(\eta)$ is independent of $\eta_i$. We note that 
\begin{equation}\label{eq:rel_r_rho}
r_i(\eta)=\rho_i(\eta)\1{\eta_i=1}+(1-\rho_i(\eta))\1{\eta_i=0}.
\end{equation}

For a given configuration, we now introduce the definition of {\it stable, unstable} and {\it potentially stable site}. 
\begin{definition}\label{def:1}
For a given site $i$, let us denote $\eta^i$ the configuration where we change $\eta_i$ to $1-\eta_i$.

Let $T\in[0,1]$. If $r_i(\eta)<T$, the site $i$ is said \textit{unstable} for $\eta$, otherwise if $r_i(\eta)\geqslant T$ the site is said \textit{stable} for $\eta$.
An unstable site $i$ for $\eta$ which is stable for $\eta^i$ is said \textit{potentially stable}.
\end{definition}

Note that $r_i(\eta^i)=1-r_i(\eta)$. Thus
\begin{enumerate}
\item a site $i$ is potentially stable if and only if $r_i(\eta)<T$ and $r_i(\eta)\le 1-T$. In particular if $T\leqslant\frac12$ an unstable site for $\eta$ is automatically potentially stable.
\item if $T>\frac12$ and $1-T< r_i(\eta) <T$, we have $r_i(\eta^i)<T$ and the site $i$ is unstable for $\eta$ and $\eta^i$.
\end{enumerate}


 Let us stress that $T$ small (that is, close to $0$) entails that the system configuration is easily close to stability, while $T$ large (that is, close to $1$) entails that it is more difficult for configurations to be stable, the constraints are not easy to satisfy. We refer to Section \ref{sec:conj} for a discussion on the dynamics.

\subsection{Infinitesimal generator, construction of the process}

Fix $\alpha >0$ and $\beta>0$. Let us consider the following dynamics: starting from a configuration $\eta$
\begin{enumerate}
\item if a site $i$ is potentially stable, we flip the value at $i$ with rate $1$ (Glauber-Schelling dynamics),
\item two nearest-neighbors $i$ and $j$ exchange their values with rate $\alpha N^2$ (accelerated spontaneous Kawasaki dynamics),
\item a site $i$ can change its value at rate $\beta>0$ (spontaneous Glauber dynamics).
\end{enumerate}

This dynamics defines an infinitesimal generator $\cL_N$ defined for $F$ a function on $\Omega_N$ by
\begin{align}\label{eq:defLtorus}
\cL_N F(\eta)=&\sum_{i\in \bbT_N^d }(F(\eta^i)-F(\eta))(\1{r_i(\eta)<T,r_i(\eta)\leqslant 1-T}+\beta)+\alpha N^2\sumtwo{i,j\in\bbT_N^d}{|i-j|=1,}(F(\eta^{ij})-F(\eta)),
\end{align}
where $\eta^{ij}$ is the configuration where the values at site $i$ and $j$ have been exchanged.

The following proposition states that the process is well defined, since the state space is finite.
\begin{proposition}\label{prop:feller}
Given an initial configuration $\eta_0$, $\cL_N$ is the infinitesimal generator of a Feller process, denoted $(\eta^N(t))_{t\geq 0}$.
\end{proposition}
 We let $\mu^N_t$ be the distribution of $\eta^N(t)$.

%

\begin{remark}
 In this article we focus on the compact setting (torus) because a non compact framework, as $\bbR^d$, presents technical problems for the convergence of the process, nevertheless the discrete model can be well defined on $\bbZ^d$ (cf. \eqref{eq:defLtorus} and Proposition \ref{prop:feller})  and we conjecture that our main results (cf. Theorem \ref{thm:convergence}) hold in this setting.
\end{remark}

\section{Main results}\label{sec:results}
We let $u^N_0(i):=\bbE_{\mu_N}\big[\eta_i^N(0)\big]$, $i\in \Tn$ be the initial distribution of our process, that is, $\eta_i^N(0)$ is distributed as a Bernoulli of parameter $u_0^N(i)$.

For a vector $v$, $|v|$ denotes its euclidean norm and $|v|_\infty$ its uniform norm.

\begin{assumption}Assumptions on $\cV_N$.
\label{ass1}
\begin{enumerate}
\item Let $\ell_\cV$ be the diameter of $\cV$. Then, $\ell_\cV\le \delta \big( \log N \big)^{\frac1d}$ for some $\delta>0$.
\item If $d=1$, suppose that $\cV_N\subset \mathbb{Z}\setminus \mathbb{N}$.
\end{enumerate}
\end{assumption}
\begin{assumption}Assumptions on $u_0^N$.
\label{ass2}
\begin{enumerate} 
\item There exists $\epsilon>0$ such that $\epsilon \le u_0^N(i)\le 1-\epsilon$ uniformly on $i\in \Tn$ and $N\in \bbN$.
\item  There exists $C_0>0$ independent of $N$ such that  $|\nabla u_0^N(i)|_\infty\le \frac{C_0}{N}$, where $\nabla u_0^N(i)=(u_0^N(i+e_k)-u_0^N(i))_{k=1}^d,$ with $e_k\in \bbZ^d$ the unit vector of direction $k$.
\item 
Let $\gu_{u_0^N}^N=\bigotimes_{i\in \Tn}\mathrm{B}\big(u_0^N(i)\big)$ be the law of a sequence of independent Bernoulli of parameter $u_0^N(i)$.
Suppose that
$\cH(\mu_0^N \, |\,  \gu _{u^N(0)}^N)=O(N^{d-\epsilon_0})$ for some $\epsilon_0>0$ small, where $\cH(\mu \, |\,  \gu)$ is the relative entropy of $\mu \ll\gu$,
\begin{equation}\label{de:relEntropy}
\cH(\mu \, |\,  \gu):=\int\frac{\dd \mu}{\dd \gu}  \log \bigg(\frac{\dd \mu }{\dd \gu}  \bigg)\dd \gu
\end{equation}
\item Let $\tilde u_0^N(x)$ be the linear interpolation on $\bbT^d=(\bbR/\bbZ)^d$, the $d$ dimensional torus, of $u_0^N(i)$ such that $\tilde u_0^N(i/N)=u_0^N(i)$. Then, there exists $u_0\in \mathcal C(\bbT^d)$ such that $\tilde u_0^N$ converges uniformly to $u_0$ in $\mathcal C(\bbT^d)$.
\end{enumerate}
\end{assumption}
\begin{remark}
The assumption \ref{ass1}(2.) is only technical and it could be removed by considering the dimension $d=1$ separately from the rest of the dimensions, cf. Remark \ref{rem:remove-hyp-chiante}.
\end{remark}

\begin{remark}
Let us observe that Assumption \ref{ass2}(3.) is stronger than the typical assumption used in the relative entropy method (see Section \ref{sec:entropymethod}), that is $\cH(\mu_0^N \, |\,  \gu _{u^N(0)}^N)=o(N)$, cf. \cite{Y91, KL99}, which is the first step to prove the convergence in Theorem \ref{thm:convergence}. This is due to the fact that we need a stronger control on the error term in order to balance the fact that the sequence $\cV_N$ grows with $N$, see \eqref{eq:control-entropy}.
\end{remark}

Define, for $t\geqslant 0$,
\begin{equation}\label{def:empmeas}
\pi^N_t=\pi^N_t(\eta, \dd v)=\frac1{N^d}\sum_{i\in \Tn}\,\eta_i(t)\delta_{\frac iN}(\dd v)
\end{equation}
 the empirical measure associated to the Markov process $\eta$ where the space is rescaled by $\frac1N$. $\pi^N_t$ is a positive measure on $\bbT^d$.

We now state our main result, which concerns the convergence in probability of the empirical measure.
\begin{theorem}\label{thm:convergence}
Under Assumptions \ref{ass1} and \ref{ass2}, if $u_0$ (the limit of the initial condition, according to Assumption \ref{ass2}-4) is such that
\begin{enumerate}
\item $u_0\in \mathcal C^1(\bbT^d)$, with $\nabla u_0$ Lipschitz,
\item $\nabla u_0(x)\neq 0$ for $x\in \bbT^d$ such that $u_0(x)=\min(T,1-T)$ or $u_0(x)=1-\min(T,1-T)$,
\end{enumerate}
then the reaction-diffusion equation 
\begin{equation}\label{eq:hydrlimit}
\begin{cases}
\partial_tu(t,x)=2\alpha\Delta u(t,x)+\beta(1-2u(t,x))+g_{\infty}(u(t,x)),\\
u(0,x)=u_0(x),
\end{cases}
\end{equation} 
with $g_\infty$ defined by $g_{\infty}(p):=(1-p)\1{1-p<\min(T,1-T)}-p\1{p\leqslant \min(T,1-T)}$,
has an unique solution $u=u(t,x)$ with $(t,x)\in [0,\tau]\times \bbT^d$ for some $\tau>0$, where $\bbT^d$ is the $d$-dimension torus.
Moreover, for every test function $\phi:\bbT^d\to \bbR$ and for every $\epsilon>0$, 
\begin{equation}\label{eq:conv_emp_meas}
\lim_{N\to +\infty}\mu^{N}\bigg(\,\Big|\langle \pi^{N}, \phi \rangle-\langle u, \phi \rangle\Big|>\epsilon\bigg)=0\,,\qquad \forall \, t\in [0,\tau]\, , 
\end{equation}
where $\langle \pi^N, \phi \rangle$ and $\langle u, \phi \rangle$ denote the integral of $\phi$ with respect to the measure $\pi^N$ or $u(x)\dd x$ respectively.
\end{theorem}
\begin{remark}
If the solution of \eqref{eq:hydrlimit} is not unique, which is not a technical difficulty but a real possibility for some initial conditions (see Remark \ref{rem:non-unique} for a concrete example), then 
any accumulation point of the sequence of empirical measure is a solution of  \eqref{eq:hydrlimit}, see Theorem \ref{thm:discrconv}.
\end{remark}

\subsection{Organisation of the paper}
To prove Theorem \ref{thm:convergence} we first prove that the empirical measure is close to a discrete measure $u^N$ which is a solution of a discrete analogous of \eqref{eq:hydrlimit}, this is Theorem \ref{thm:intermediateGoal}. Its proof is based on an entropy method approach in which the relative entropy between $\mu^N$ and $\gu_{u^N}^N$, see Theorem \ref{thm:entropy}. Even if this technique is quite standard in the particle systems theory, some new technical estimations arising from the geometry of the system are needed, this is Theorem \ref{prop:UB_V^+}.
In Section \ref{sec:estimatesolu^N} we discuss some central technical estimations about $u^N$, in order to describe the behaviour of the discrete model.
 Then, in Section \ref{sec:reaction-diff} we discuss the existence and uniqueness of \eqref{eq:hydrlimit} and in Section \ref{sec:conv_discrPDE} we show the convergence of the $u^N$ toward the density $u$ by completing the proof of Theorem \ref{thm:convergence}. We stress that the proof of the existence and uniqueness is not standard and the analysis of this PDE is interesting in its own.
 
 In Section \ref{sec:conj} we state our conjecture on the phase diagram of the model.

\subsection{Conjecture on the phase diagram}\label{sec:conj}
In this Section, we discuss the phase diagram that describes the mixed and segregated phases. We start by setting the Equation \eqref{eq:hydrlimit} in a more convenient form.
Set $p_0(T):=\min(T,1-T)\in [0, \frac12]$. For $p\in[0,1]$ we define 
\begin{align}
\gamma_{\infty, \beta}(p)&=\left\{
\begin{aligned}
&\beta\left(p-\frac12\right)^2+\frac12\left(p^2-p_0(T)^2\right) 
&\text{ for $0\leqslant p<p_0(T)
$},\\
& \beta\left(p-\frac12\right)^2  &\text{ for $p_0(T)\leqslant p<1-p_0(T)
$},\\
&\beta\left(p-\frac12\right)^2+\frac12\left((1-p)^2-p_0(T)^2\right) &\text{ for $1-p_0(T)\leqslant p\le 1$}.
\end{aligned}\right.
\end{align}
Our conjecture is based on the analysis of $\gamma_{\infty, \beta}$ and it is represented in Figure \ref{fig1}. 
We observe that $\gamma_{\infty,\beta}$ is continuous and satisfies $\gamma_{\infty,\beta}(p)=\gamma_{\infty,\beta}(1-p)$. For $p~\neq~p_0(T),1-p_0(T)$, we have that
$\gamma'_{\infty,\beta}(p)=-\beta(1-2p)-g_{\infty}(p)$ and \eqref{eq:hydrlimit} can be written as
\begin{equation}
\partial_tu(t,x)=2\alpha\Delta u(t,x) -\gamma'_{\infty,\beta}(u(t,x)).
\end{equation}
In such a way $\gamma_{\infty, \beta}$ can be viewed as a potential function of the system, its analysis provides the stable and metastable equilibrium points of the system.
Therefore, to discuss the phase transition we can look at the structure of $\gamma_{\infty,\beta}(p)$, see Figure \ref{fig2}.
The function $p\mapsto \beta \left(p-\frac12\right)^2+\frac12\left(p^2-p_0(T)^2\right)$ has a unique minimum at $p=p^\ell:=\frac{\beta}{1+2\beta}$.
Therefore, if $0\le p^\ell<p_0(T)<\frac12 $, 
the function $\gamma_{\infty,\beta}$ has three regular minima: $p^c:=\frac12$, $p^\ell$ and $p^r:=1-p^\ell$. 
Note that 
\[
\gamma_{\infty, \beta}(p^c)=0
\quad \text{ and } \quad 
\gamma_{\infty, \beta}(p^r)=\gamma_{\infty,\beta}(p^\ell)=\frac{\beta}{4(1+2\beta)}-\frac{p_0(T)^2}2 \, .
\]
Then we get that if $p_0(T)<p^m:=\sqrt{\frac{\beta}{2(1+2\beta)}}$, we have $\gamma_{\infty, \beta}(p^c)<\gamma_{\infty, \beta}(p^\ell)$ and if $p^m<p_0(T)$, we have that $\gamma_{\infty, \beta}(p^c)>\gamma_{\infty}(p^\ell)$. 
If $p^\ell>p_0(T)$, $p^c$ is the only minimum.
\smallskip

The two thresholds for $p_0(T)$ are then $p^\ell$ and $p^m$, see Figure \ref{fig1}. Since $p^\ell<p^m$, we have the following picture: as $T$ is close to $0$ and below $p^\ell$, we have a unique minimum of $\gamma_{\infty, \beta}$, so that typical configurations are close to $p=1/2$ which is of lowest energy of $\gamma_{\infty,\beta}$. It means that, at equilibrium, we expect a configuration balanced between $0$ and $1$ and we do not have segregation.
Then, as $T$ goes above the threshold $p^\ell$ but stays below $p^m$, other minima at $p=p^\ell$ and $p=p^r$ appear, and these two configurations are metastable since their energy is higher, so we can have segregation for a small proportion of the time. The next threshold is $p^m$, at which the two metastable configurations become stable and $p=1/2$ is the metastable one so that we expect stable segregation. 
For $T$ above $\frac12$ the picture is symmetric.

\begin{figure}[t]
\centering
\begin{subfigure}{0.4\textwidth}
    \includegraphics[width=\textwidth]{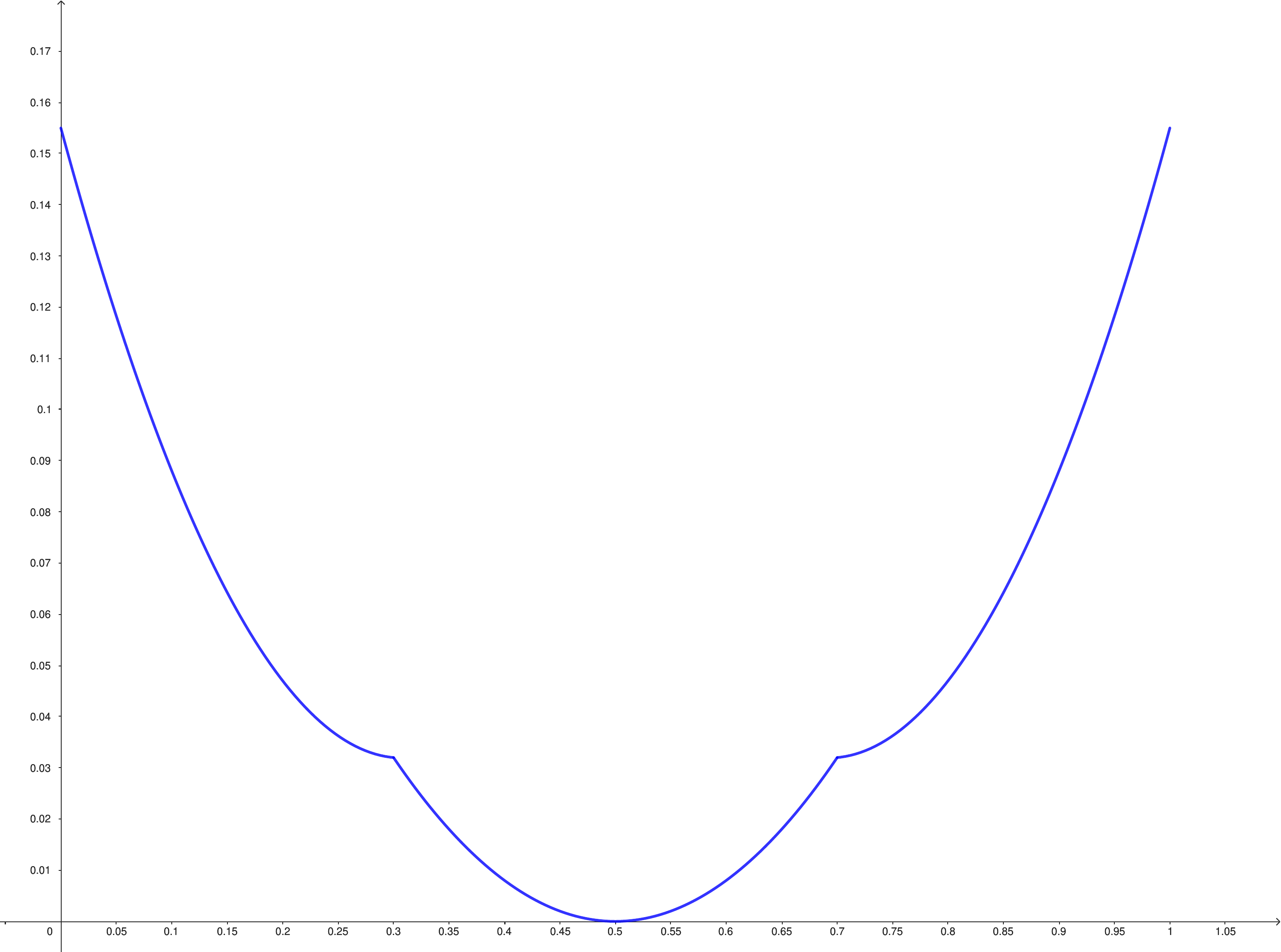}
    \caption{$T=0.3$ and $\beta=0.8$}
    \label{fig:1}
\end{subfigure}
    \hfill
\begin{subfigure}{0.4\textwidth}
    \includegraphics[width=\textwidth]{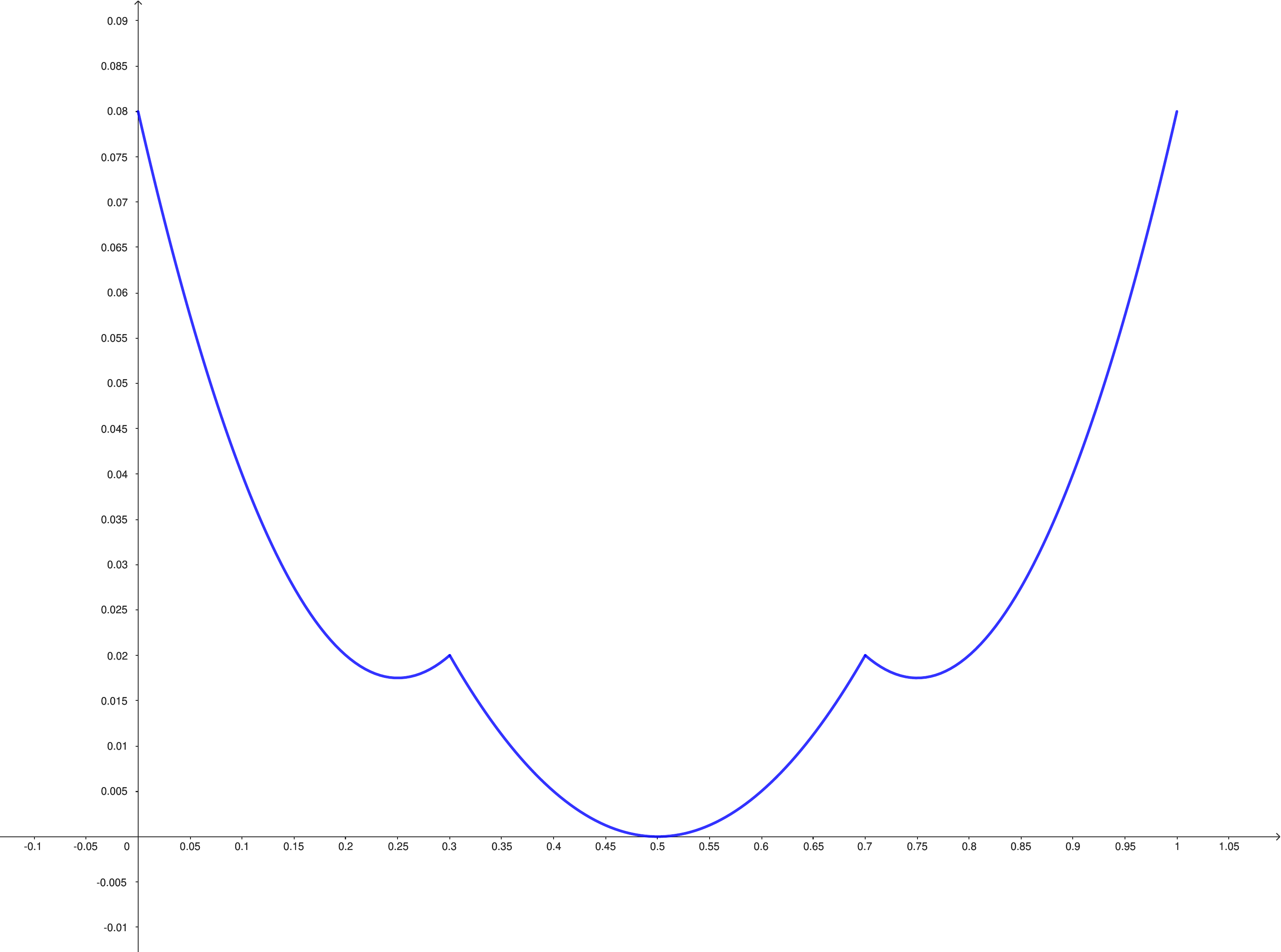}
    \caption{$T=0.3$ and $\beta=0.5$}
    \label{fig:2}
\end{subfigure}

\begin{subfigure}{0.4\textwidth}
    \includegraphics[width=\textwidth]{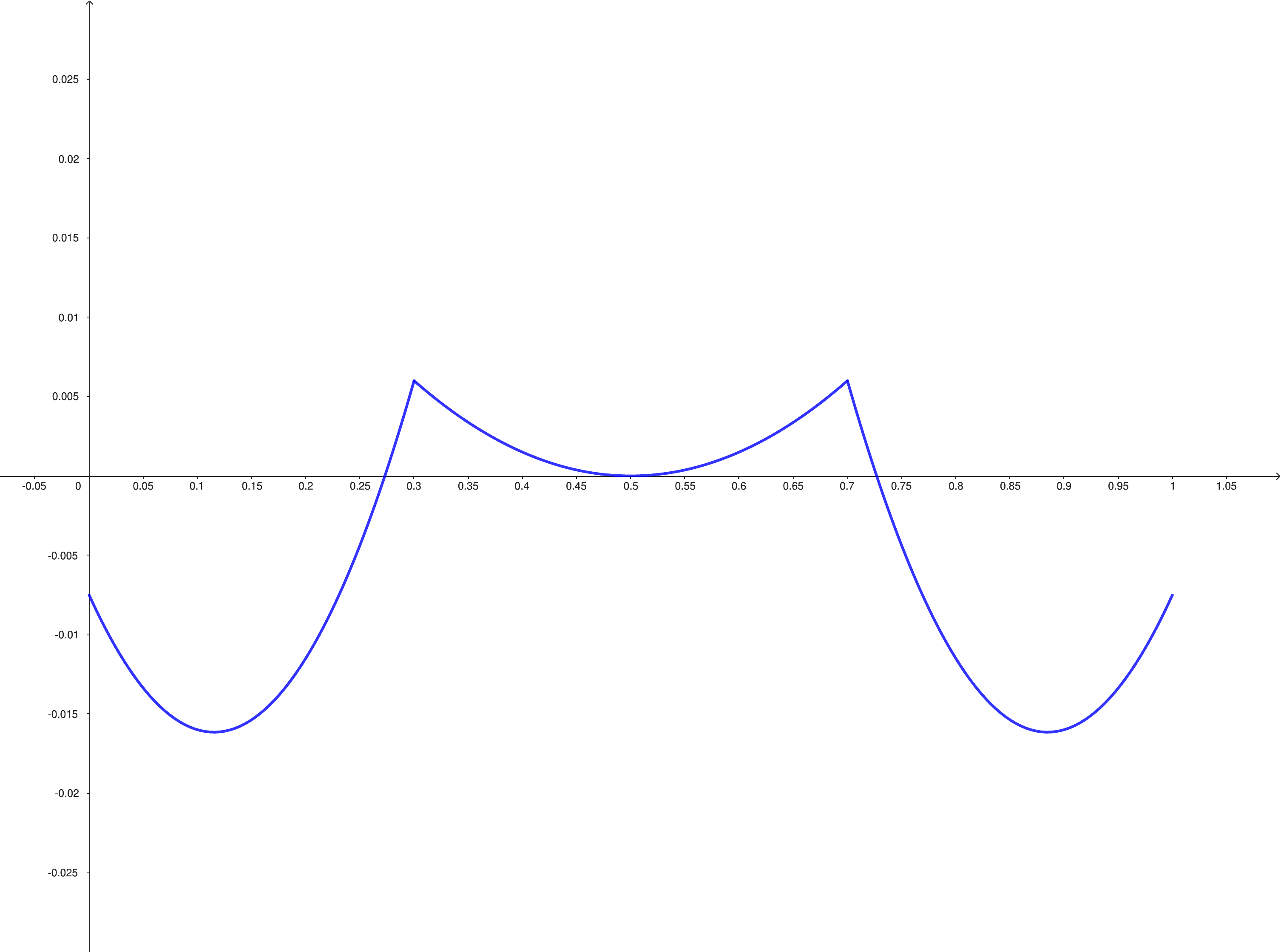}
    \caption{$T=0.3$ and $\beta=0.15$}
    \label{fig:3}
\end{subfigure}
\hfill
\begin{subfigure}{0.4\textwidth}
    \includegraphics[width=\textwidth]{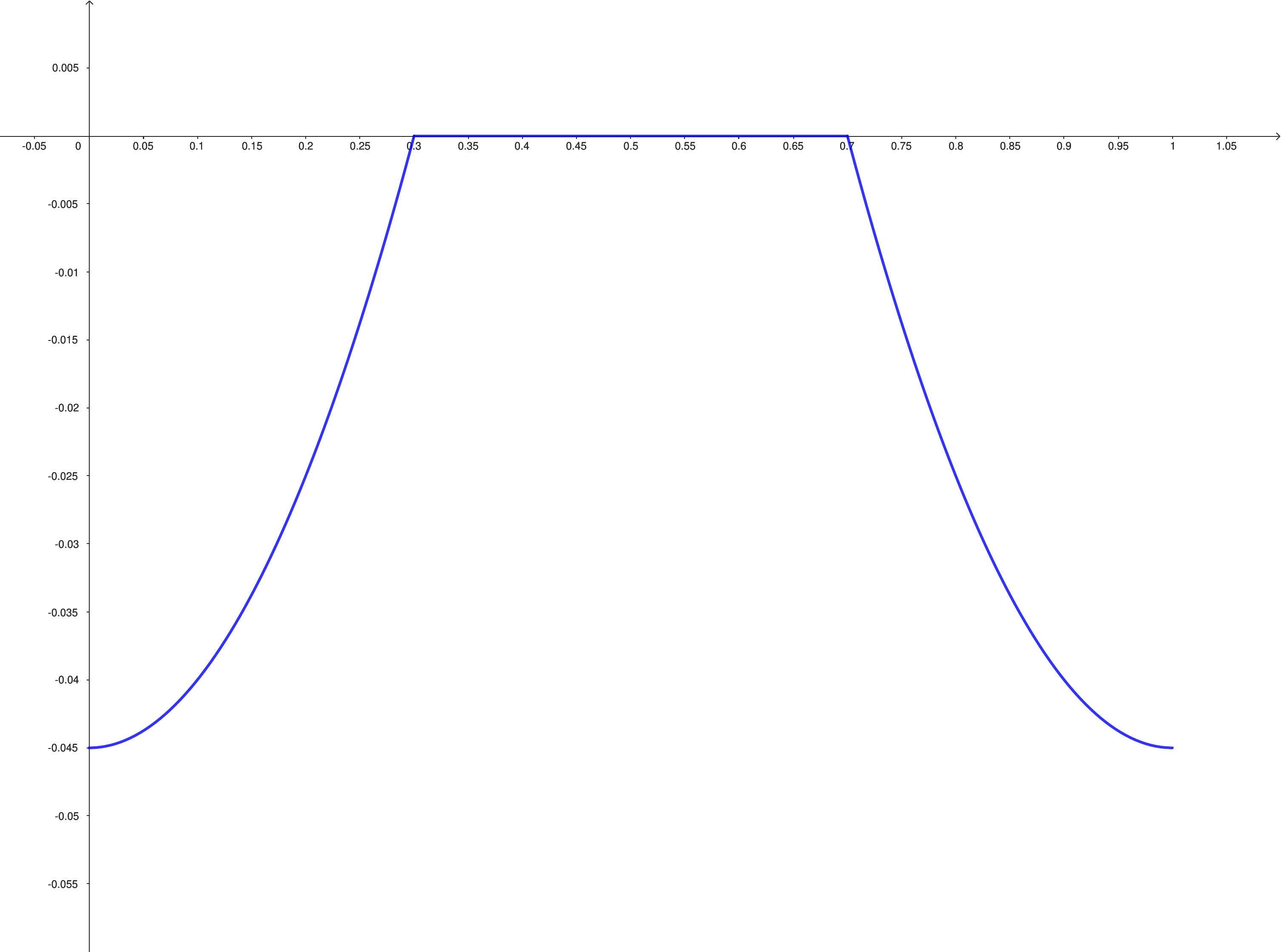}
    \caption{$T=0.3$ and $\beta=0$}
    \label{fig:4}
\end{subfigure}

\caption{Representation of $\gamma_{\infty, \beta}(p)$, $p\in[0,1]$, for different values of $\beta$ and a fixed $T=0.3$. 
The discontinuities of $\gamma_{\infty, \beta}^\prime(p)$ are situated at $p_0(T)$ and $1-p_0(T)$.
In \eqref{fig:1} we have that $p_0(T)=0.3<p^\ell\approx 0.3076$ and the unique stable equilibrium point is $p^c=\frac12$. In \eqref{fig:2} we have that $p^\ell=0.25<p_0(T)=0.3<p^m\approx 0.3535$, so that $p^c=\frac12$ is stable, while $p^\ell$ and $p^r$ are metastable equilibrium.
In \eqref{fig:3} we have that $p_0(T)=0.3>p^m\approx 0.2401$ and $p^\ell$ and $p^r$ become stable while $p^c$ is metastable. Finally in \eqref{fig:4} we illustrated the limit case with $\beta=0$, and the two stable equilibriums are $p^\ell=0$ and $p^r=1$. In this case, the local minimum $p^c$ degenerates into the segment $[p_0(T), 1-p_0(T)]$.}
\label{fig2}
\end{figure}

\begin{figure}[t]
\centering
\includegraphics[scale=1.6]{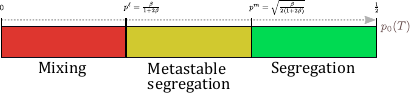}
\caption{Representation of the different phases of the system as function of the parameter $p_0(T)\in [0, \frac12]$. When $T$ is close to $0$ or $1$ ($p_0(T)\in (0, p^\ell)$) we do not have segregation (red parts) and typical configurations are provided by a mixing of $0$ and $1$. If $T$ is close to $1/2$ ($p_0(T)\in (p^m,\frac12)$) we have segregation (green parts): a very large part of the configuration are composed of $0$ or $1$. We have also intermediate values of $T$ ($p_0(T)\in (p^\ell,p^m)$) for which the segregation is metastable (yellow parts).}
\label{fig1}
\end{figure}

Let us discuss heuristically why if $p_0(T)<p^\ell$, we expect that the random perturbation dominates and that we have a mixed phase.
Let us suppose for a while that $\beta=0$, so that we do not consider the spontaneous Glauber dynamics and the change of the status of a site is only due to the Glauber-Schelling dynamics. 
According with the literature (see e.g., \cite{AL22PhD} and the reference therein), small and large values of
$T$ produce analogous behaviour of the system, but for different reasons. More precisely, given a configuration $\eta$, according with Definition \ref{def:1}, a site $i$ is potentially stable if and only if $r_i(\eta)< T$ and $r_i(T)\le 1-T$. 
Therefore, if $T$ is small, an unstable site is automatically potentially stable so that unstable sites become stable because of the Glauber-Schelling dynamics. Moreover, if $T$ is small, the system is already close to a stable configuration since the constraints are easily satisfied. 
On the other hand, if $T$ is close to $1$, then $1-T$ is small, so that the largest part of the unsatisfied sites do not change their status because this does not make them satisfied and the Glauber-Schelling dynamics is swiftly blocked. 
In both these case we do not expect segregation. 
If we introduce the spontaneous Glauber perturbation in the system and $\beta$ is large compared to $p_0(T)$, we expect that the spontaneous flips drive the behaviour of the system since the noise in each neighborhood is non-negligible, forcing the Glauber-Schelling mechanism to a continuous flipping of the sites. We then expect that if $p_0(T)<p^\ell$ the typical configuration is a mixing between the two groups without macroscopic clusters.

On the other hand, if $T$ is close to $1/2$, the Glauber-Schelling dynamics can give rise to a segregation process in which typical configurations are composed by a large part of $0$ or $1$ and we observe the appearance of macroscopic clusters which stabilizes the configurations. Therefore, if the formation of the clusters is sufficiently fast compared to the perturbation given by the spontaneous Glauber dynamics, the dynamics of the system should be close to the one without perturbation. We then expect the formation of cluster and so segregation.

Finally, we would discuss the role of $\alpha$ in the system. We conjecture that the spontaneous Kawasaki dynamics acts on the system diffusely, by steering the fluctuations which are responsible to the convergence of the system around stable configurations given by the lowest energy of $\gamma_{\infty, \beta}$. 
In Figure \ref{fig2} we observe that in the metastable phase, independently from the initial configuration, after a short time the fluctuations pushes the system to stabilize around a typical configuration close to $p=\frac12$, the unique stable minimum of $\gamma_{\infty, \beta}$. In the extreme case of $\beta=0$, we expect that the fluctuations pushed the system to one of the two stables configurations and one of the two groups disappears.

\begin{remark}\label{rem:gradflow}
The hydrodynamic limit \eqref{eq:hydrlimit} has, at least, two different formulations as a gradient flow: 
\begin{enumerate}
\item in the classical $L^2(\bbT^d)$ setting, with the potential $\mathcal{F}$ defined for $u:\bbT^d\to[0,1]$
\begin{align}
\mathcal{F}(u)=\int \alpha\|\nabla u\|^2+\gamma_{\infty,\beta}(u),
\end{align}
Equation \eqref{eq:limitePDE} can be written as $\partial_t u=-\delta\mathcal{F}(u)$ where $\delta\mathcal{F}$ denotes the Fréchet derivative of $\mathcal{F}$.

\item in a Wasserstein-like setting defined in \cite{LM13} with the entropy potential $\mathcal{H}$, for $u:\bbT^d\to]0,1]$, and $\xi:\bbT^d\to\bbR$ (seen as an element in the tangent bundle)
\begin{align}
\mathcal{H}(u)&=\frac12\int (2u\log(2u)-2u+1)\geqslant 0,&
\mathcal{K}(u)\xi&=-2\alpha\nabla \cdot (u\nabla \xi)+\frac{\gamma'_{\infty,\beta}(u)}{\log(2u)}\xi,
\end{align}
Equation \eqref{eq:hydrlimit} can be written as \[\partial_t u=-\mathcal{K}(u)(\delta\mathcal{H}).\] 
Since $\delta\mathcal{H}=\log(2u)$ and $\mathcal{K}(u)(\delta\mathcal{H})
=-2\alpha \Delta u+\gamma'_{\beta,\infty}(u).$
\end{enumerate}

Both formulations could be useful to establish a rigorous proof of the phase diagram given in Figure \ref{fig1}. In particular, along a gradient flow the potential is non-increasing, thus for all $t>0$, along a solution $u$ we have $\mathcal{F}(u(t))\le \mathcal{F}(u(t=0))$ and if $u$ converges to a stationary solution $v$, it must be a stationary point of $\cF$ (i.e. $\delta \cF(v)=0$).

For the second formulation, note that 
\begin{equation}
\frac{\gamma'_{\infty,\beta}(u)}{\log(2u)}>0, u\in]0,1]
\Longleftrightarrow
p_0(T)<p^\ell=\frac{\beta}{1+2\beta}
\Longleftrightarrow
\beta>\frac{p_0(T)}{1-2p_0(T)}.
\end{equation}
Thus, for $p_0(T)<p^\ell$, we are in the mixing phase of the diagram and $\mathcal{K}(u)$ is positive definite in the sense that 
\begin{align}
\int \xi \mathcal{K}(u)\xi=\int 2\alpha u|\nabla\xi|^2+\frac{\gamma'_{\infty,\beta}(u)}{\log(2u)}\xi^2\geqslant 0.
\end{align}
Then, one can prove that along a solution $u$, we get that:
\begin{align}
\frac{\dd}{\dd t}\mathcal{H}(u)=\int \partial _t u \delta\mathcal{H}(u)\leqslant-\int \gamma'_{\infty,\beta}(u)\log(2u)\leqslant -c\mathcal{H}(u)
\end{align}
where $c=\beta- \frac{p_0(T)}{1-2p_0(T)}>0$, and we get that $\mathcal{H}(u(t))\leqslant \mathcal{H}(u(0))e^{-ct}$. Thus $\mathcal {H}(u(t))$ goes to $0$ as $t\to\infty$, this entails that $u$ converges to the only stationary point of $\mathcal{H}$ which is the constant $\frac12$. Heuristically, it suggests that an exponential relaxation is taking place in the mixing phase of Figure \ref{fig1}.

\end{remark}

\section{Relative entropy method}\label{sec:entropymethod}
Using the {\it relative entropy} method, in this section we prove that the empirical measure $\pi^N(t)$ is close to $u^N(t)=(u^N(t,i))_{i\in \Tn}$ the solution of a suitable discrete PDE.

In order to state the main result of this section, we observe that the generator $\cL_N$ in \eqref{eq:defLtorus} can be written as $\cL_N=\cG_{N}+2\alpha N^2 \cK_{N}$ where $\cG_N$ is the generator which describes the Glauber-Schelling and spontaneous Glauber dynamics and $\cK_N$ is the generator which describes the spontaneous Kawasaki dynamics, that is, 
\begin{align}\label{def:glauberGen}
\cG_NF(\eta)=&\sum_{i\in \bbT_N^d }\Big(\1{r_i(\eta)<T,r_i(\eta)\leqslant 1-T}+\beta\Big)(F(\eta^i)-F(\eta)), \\\label{def:kawasakiGen}
\cK_N F(\eta)=& \frac 12 \sumtwo{i,j\in\bbT_N^d}{|i-j|=1}(F(\eta^{ij})-F(\eta)).
\end{align}
Let us stress that $\cG_N$ can be written as follows
\begin{equation}
\cG_N F(\eta)=\sum_{i\in \Tn} (c_i(\eta)+\beta)\big(F(\eta^i)-F(\eta)\big),
\end{equation}
where $c_i(\eta)$ is a local function which describes the dynamics (Glauber-Schelling dynamics). 
To be more precise, we write $c_i(\eta)=c_0(\tau_i \eta)$ where $c_0$ is the flipping rate of a particle at the origin, that is,
\begin{equation}
c_0(\eta):=\1{r_0(\eta)<T,r_0(\eta)\leqslant 1-T}
\end{equation}
and $(\tau_i\eta)_j=\eta_{i+j}$, likewise for a function $u=(u_i)_{i\in \Tn}$, $\tau_i$ acts on $u$, that is $(\tau_iu)_j=u_{i+j}$.
Note that $c_0$ is a random variable which take the value $0$ or $1$.

We let 
\begin{equation}
\kappa_{N}=\kappa_N(T):= \min\Big\{\lfloor K_N T \rfloor-1; \lfloor K_N(1-T)\rfloor \Big\}.
\end{equation}
We observe that $\lim\limits_{n\to +\infty} \frac{\kappa_N}{K_N}=\min(T,1-T)$.

Since $c_0(\eta)=\1{r_0(\eta)\leqslant \frac{\kappa_{N}}{K_N}}
=\1{r_0(\eta)\leqslant\frac{\kappa_{N}}{K_N},\eta_0=0}+\1{r_0(\eta)\leqslant \frac{\kappa_{N}}{K_N},\eta_0=1}$, we define
\begin{equation}\label{eq:c0+-}
c_0^+(\eta):=\1{\rho_0(\eta)\geqslant 1-\frac{\kappa_N}{K_N}} \qquad \text{and}\qquad
c_0^-(\eta):=\1{\rho_0(\eta)\leqslant \frac{\kappa_N}{K_N}},
\end{equation}
so that $c_0(\eta) =c_0^{+}(\eta)(1-\eta_0)+c_0^{-}(\eta)\eta_0,$
by \eqref{eq:rel_r_rho}. The functions $c^+$ and $c^-$ can be viewed as the rate of creation and annihilation of a particle at $i=0$ respectively.

\smallskip

For any function $u=(u_i)_{i\in \Tn}$ we define (recall that $\mathrm{B}\big(u\big)$ is the law of a Bernouilli of parameter $u\in[0,1]$)
\begin{equation}\label{eq:bernouilli}
\gu_{u}(\dd\eta)=\gu_{u}^N(\dd\eta):=\bigotimes_{i\in \Tn}\mathrm{B}\big(u_i\big)(\dd \eta_i)
\end{equation}
and we let $c_0^+(u)$ and $c_0^-(u)$ be the expectation of $c_0^+(\eta)$ and $c_0^-(\eta)$ under $\gu_u$, that is,
\begin{align}\label{eq:c^pm u}
c_0^+(u):=\mathbb{P}_{\gu_u}\Bigg(\rho_0(\eta)\geqslant 1-\frac{\kappa_N}{K_N}\Bigg)\qquad \text{and}\qquad
c_0^-(u):=\mathbb{P}_{\gu_u}\Bigg(\rho_0(\eta)\leqslant \frac{\kappa_N}{K_N}\Bigg).
\end{align}
We finally define
\begin{equation}\label{def:G}
G(u):=c_0^+(u) (1-u_0)-c_0^-(u) u_0\quad \text{and}\quad G(i, u):=G(\tau_i u).
\end{equation}
Let $u^N(t)=(u^N(t,i))_{i\in \Tn}$ be the solution of
\begin{equation}\label{eq:discretPDE}
\begin{cases}\partial_tu^N(t,i)=2\alpha N^2 \Delta u^N(t,i)+\beta(1-2u^N(t,i))+G(i,u^N)\\
u^N(t,i)=u^N_0(i),\end{cases}
\end{equation}
where $\Delta u^N(t,i)$ is the discrete Laplacian on the torus. Note that \eqref{eq:discretPDE} can be interpreted as a discretized version of \eqref{eq:hydrlimit}.

\begin{remark}
Note that \eqref{eq:discretPDE} is a first order ordinary differential equation in $\bbR^{\Tn}$. Therefore, we have a solution, locally in time, starting from every initial condition.
\end{remark}

To prove Theorem \ref{thm:convergence} we first define a discrete approximation of $u$ and we show an equivalent of Theorem \ref{thm:convergence} for $u^N$ defined in \eqref{eq:discretPDE}. More precisely, we consider $u^N$ as a measure on $\bbT^d$, that is, 
 \begin{equation}
 u^N(t,\dd v):=\frac{1}{N^d}\sum_{i\in \Tn}u^N(t,i)\delta_{\frac{i}{N}}(\dd v)\,,
 \end{equation}
The main result of this section is the following theorem.
 \begin{theorem}\label{thm:intermediateGoal}
Under Assumptions \ref{ass1} and \ref{ass2}, for every test function $\phi:\bbT^d\to \bbR$ and for every $\delta>0$ there exists $\tau>0$ such that
\[
\lim_{N\to +\infty}\mu^N\bigg(\,\Big|\langle \pi^N, \phi \rangle-\langle u^N, \phi \rangle\Big|>\delta\bigg)=0\,,\qquad \forall\, t\in [0,\tau]\,.
\]
\end{theorem}

To prove Theorem \ref{thm:intermediateGoal}, the main ingredient is that the relative entropy of $\gu^N_{u^N(t)}$ (cf. \eqref{eq:bernouilli}) with respect to $\mu_t^N$ stays small in time, if it is small at $t=0$.  
 \begin{theorem}\label{thm:entropy}
 Under Assumptions \ref{ass1} and \ref{ass2}, with $\delta>0$ sufficiently small, we have that for some $\epsilon >0$ small
\begin{equation}
\cH_N(t):=\cH(\mu_t^N \, |\,  \gu^N_{u^N(t)})=O(N^{d-\epsilon}), \quad \forall\, t\in [0,\tau]
\end{equation}
with $\tau=\tau(\delta)>0$.
 \end{theorem}
 To prove Theorem \ref{thm:entropy} it is enough to show that
\begin{equation}\label{eq:goalJ_t}
\partial_t \cH_N(t)\le C \ell_\cV^{d} \Big(\cH_N(t)+ O(N^{d-a})\Big)
\end{equation}
with $a\in (0,1)$ if $d\ge 2$ and $a\in (0, \frac12)$ if $d=1$, indeed in this case Gronwall's inequality gives 
\begin{align}\label{eq:control-entropy}
\cH_N(t) \le \Big(\cH_N(0)+t O(N^{d-a})\Big)e^{C \ell_\cV^{d} t}.
\end{align}
Since $e^{C \ell_\cV^{d} t}\le N^{C\delta t}$ by Assumption \ref{ass1}, the proof of Theorem \ref{thm:entropy} is complete. We prove \eqref{eq:goalJ_t} in Section \ref{sec:proofThmEntropy}

 \smallskip
 
Let us observe that Theorem \ref{thm:entropy} implies Theorem \ref{thm:intermediateGoal}. Indeed, we recall the entropy inequality stated for a set $\cA$ and two measures $\gu \ll \mu$, cf. A1.8.2 of \cite{KL99} or Section 2.2 of \cite{FT19},
\begin{equation}\label{eq:entropyineqforB}
\mu(\cA)\le \frac{\log 2 + \cH(\mu \, |\,  \gu)}{\log\big(1+\frac{1}{\gu(\cA)}\big)}.
\end{equation}
For a given test function $\phi$ and $\delta >0$, we let
\begin{equation}
\cA_{N,t,\phi}^\delta:=\big\{\eta\in \Omega_N \colon \big|\langle \pi^N, \phi \rangle-\langle u^N, \phi \rangle\big|>\delta\big\},
\end{equation}
so that the proof follows by Theorem \ref{thm:entropy} and \eqref{eq:entropyineqforB} if
\begin{equation}\label{eq:guA}
\gu_{u^N(t)}^N(\cA_{N,t,\phi}^\delta)\le e^{-C_\delta N^d}.
\end{equation}
Since $u^N \in (0,1)$ (cf. Proposition \ref{prop:u^Nin01}), the proof of \eqref{eq:guA} is model independent and follows line by line the proof of Proposition 2.2 in \cite{FT19}, we omit the details.

\begin{remark}\label{rem:1.5}
Let us note that $c_0$ can be expressed as a polynomial on the variables $\eta_j$'s, as in relation (1.5) of \cite{FT19}. This remark will be useful in the sequel of the paper.
For this purpose, let $A\subset\cV_N$, we denote:
\begin{align}\label{c_A^+}
c_A^+(\eta)=\prod_{j\in A}(1-\eta_j)\prod_{j\in\bar{A}\cap\cV_N}\eta_j\quad \text{and}\quad 
c_A^-(\eta)=\prod_{j\in A}\eta_j\prod_{j\in\bar{A}\cap\cV_N}(1-\eta_j)=c_A^+(1-\eta)
\end{align}
where $1-\eta$ is the configuration with $(1-\eta)_i=1-\eta_i$ since $r_0(\eta)=r_0(1-\eta)$.
Note that 
\begin{equation}
c_A^+(\eta)=\left\{
\begin{aligned}
1&\text{ if $\eta_j=0$ for $j\in A$ and $\eta_j=1$ for $j\in \bar{A}\cap\cV_N$}\\
0&\text{ otherwise}.
\end{aligned}\right.
\end{equation}
By an abuse of notation, for a function $u=(u_j)_{j\in \Tn}$ we let
 let also $c_A^+(u)=\prod_{j\in A}(1-u_j)\prod_{j\in\bar{A}\cap\cV_N}u_j$ and accordingly for $c_A^-(u)$.
By \eqref{eq:c0+-} we have that
\begin{equation}
c_0^+(\eta)
 =\sum_{k=0}^{\kappa_N}\1{r_0(\eta)=\frac{k}{K_N},\eta_0=0}
 =(1-\eta_0)\sumtwo{A\subset \cV_N}{|A|\leqslant \kappa_N}c_A^+(\eta).
\end{equation}
Accordingly, we have 
\begin{equation}
c_0^-(\eta)
 =\sum_{k=0}^{\kappa_N}\1{r_0(\eta)=\frac{k}{K_N},\eta_0=1}
 =\eta_0\sumtwo{A\subset \cV_N}{|A|\leqslant \kappa_N}c_A^-(\eta).
\end{equation}
\end{remark}

In the rest of this section we prove Theorem \ref{thm:entropy}. The strategy that we use follows the one used to prove the analogous result in \cite{FT19} and  \cite{JM20}, but some extra-technicality is required due to the geometry of our problem.
 
\subsection{Proof of Theorem \ref{thm:entropy}}
\label{sec:proofThmEntropy}
%


To make the notation lighter, for $t\ge 0$, $j\in \Tn$, $p \in (0,1)$ and $u^N$ the solution of \eqref{eq:discretPDE}, we let
\begin{align}\label{eq:notations}
u_j(t):=u^N(t,j), \qquad \chi(p):=p(1-p), \qquad
\omega_j:=\frac{\eta_j-u_j}{\chi(u_j)}.
\end{align}
and, more generally, whenever the context is clear we omit the superscript $N$, so that $\gu_{u^N(t)}^N$ and $\mu^N_t$ will be  denoted simply by $\gu_{u(t)}$ and $\mu_t$ respectively. Therefore, throughout this section $u=u^N$, the solution of \eqref{eq:discretPDE}.

To compare $\mu$ and $\gu_u$ we introduce
\begin{equation}
\theta_{\ga}(\dd\eta):=\bigotimes_{i\in \Tn}\mathrm{B}\big(\alpha\big)
\end{equation} 
be a sequence of independent Bernoulli of parameter $\alpha\in (0,1)$ defined on the space of configurations.
We define
\begin{equation}
f_t := \frac{\dd \mu_t}{\dd \gu_{u(t)}} 
\qquad\text{and}\qquad
\psi_t :=\frac{\dd \gu_{u(t)}}{\dd \theta_{\ga}} .
\end{equation}
%
We have all the ingredients to state Yau's inequality in our context. The proof is quite standard (cf. proof of Lemma A.1 in \cite{JM20}), so that it is omitted.

\begin{proposition}\label{pro:YauInequality}
For any $t\ge 0$ we have that 
\begin{equation}\label{eq:courent1}
\partial_t \cH_N(t) \le
-\int\,  \Gamma_N\Big(\sqrt{f_t(\eta)}\Big)\, \gu_{u(t)}(\dd \eta)+
\int\,  f_t(\eta)\, \Big[\cL_N^{*,\gu_{u(t)}} \mathbf 1 -\partial_t \log\psi_t\Big]\, \gu_{u(t)}(\dd \eta),
\end{equation}
where $\cL_N^{*,\gu_{u(t)}}$ is the adjoint of $\cL_N$ with respect to the measure $\gu_{u(t)}$ and $\Gamma_N(h)(\eta)=\cL_Nh^2(\eta)-2 h(\eta)\cL_Nh(\eta)$ is the carré du champ operator.
\end{proposition}

We define the {\it current} $\cJ_t=\cJ_t^N(\eta)$ as 
\begin{equation}\label{eq:defcurrent}
\cJ_t:=\cL_N^{*,\gu_{u(t)}} \mathbf 1 -\partial_t \log\psi_t.
\end{equation}
Our main goal is to estimate the current $\cJ_t$ to control the right hand side of \eqref{eq:courent1} and get \eqref{eq:goalJ_t}. 

\subsubsection{The current $\cJ_t^N$}
To control $\cL_N^{*,\gu_{u(t)}}  \mathbf 1 $ we have to compute the adjoint of $\cK_N$ and of $\cG_N$, cf. \eqref{def:glauberGen} and \eqref{def:kawasakiGen}. We follow the computations done in \cite{FT19}.
By Lemma 2.4 of \cite{FT19}, we get
\begin{equation}\label{eq:adjK}
\cK_N^{*,\gu_{u(t)}} \mathbf 1 
=-\frac12  \sumtwo{i,j\in\Tn,}{|i-j|=1}(u_i-u_j)^2\omega_j\omega_i +\sum_{i\in \Tn} (\Delta u)_i \omega_i
\end{equation}
where $(\Delta u)_i=\sum_{j\in \Tn, |i-j|=1} (u_j-u_i)$ is the discrete Laplacian.
Since $c_0$ satisfies the condition (1.5) of \cite{FT19} (see Remark \ref{rem:1.5}), by Lemma 2.5 of \cite{FT19} we get
\begin{equation}\label{eq:adjG}
\begin{split}
\cG_N^{*,\gu_{u(t)}} \mathbf 1 
&=\sum_{i\in \Tn} \left(c_i^+(\eta) (1-u_i)-c_i^-(\eta) u_i+\beta (1-2u_i)\right)\omega_i\\
&=\sum_{i\in \Tn} \left((c_i^+(\eta)-c_i^+(u)) (1-u_i)-(c_i^-(\eta)-c_i^-(u)) u_i\right)\omega_i\\
&\hspace{2.5cm}+\sum_{i\in \Tn} \left(c_i^+(u) (1-u_i)-c_i^-(u) u_i+\beta (1-2u_i)\right)\omega_i.
\end{split}
\end{equation}
In the second equality we centered the variables $c_i^+(\eta)$ $c_i^-(\eta)$ since under $\gu_u$, $c_i^+(\eta)$ and  $c_i^-(\eta)$ are Bernouilli random variables with expectation $c_i^+(u)$ and $c_i^-(u)$ respectively.
Finally, by Lemma 2.6 of \cite{FT19} we get
\begin{equation}\label{eq:delta_psi}
\partial_t \log\psi_t(\eta)=\sum_{i\in \Tn}(\partial_tu_i)\omega_i.
\end{equation}
Summarizing, we obtain the following result.
\begin{proposition}
The current $\cJ_t(\eta)$ satisfies 
\begin{equation}
\begin{split}
\cJ_t(\eta)
=\sum_{i\in \Tn}(-\partial_tu_i +  2\alpha N^2 (\Delta u)_i +\beta(1-2u_i)+  &  G(i,u))\omega_i\\
& -V(u, \eta)+V^+(u,\eta)-V^-(u,\eta),
\end{split}
\end{equation}
where $G$ was defined in \eqref{def:G} and
\begin{align}
\label{eq:defV^+}
V^+(u,\eta)&=\sum_{i\in \Tn} \left(c_i^+(\eta)-c_i^+(u) \right)(1-u_i)\omega_i,\\
\label{eq:defV^-}
V^-(u,\eta)&=\sum_{i\in \Tn} \left(c_i^-(\eta)-c_i^-(u) \right)u_i\omega_i,\\
\label{eq:defV}
V(u, \eta)&=-\alpha N^2 \sumtwo{i,j\in\Tn,}{|i-j|=1}(u_i-u_j)^2\omega_j\omega_i.
\end{align}
In particular, if $u$ satisfies \eqref{eq:discretPDE}
the current reduces to the second line.
\end{proposition}
In the rest of the section we provide estimates of $V^+$, $V^-$ and $V$.

\subsubsection{\bf Estimates of $V^+$ and $V^-$}
Let us denote $\{e_1,e_2,\dots,e_d\}$ the canonical basis of $\mathbb{Z}^d$.
For $\phi:\Tn\to\bbR$ and $i\in\Tn$ and $k\in\{1,\dots, d\}$, let $\nabla_k \phi(i)=\phi(i+e_{k})-\phi(i)$. We denote $\|\nabla \phi\|_{\infty}=\max_{i,k}|\nabla_k \phi(i)|$.


We note that
\begin{equation}\label{eq:V-} 
V^-(u,\eta)=-V^+(1-u,1-\eta),
\end{equation}
 so that the bound for $V^+$ can be transferred to $V^-$, see Remark \ref{rem:transfV+toV-}. In the following we get an upper-bound for $V^+$. Denote $\theta_i^+=c_i^+(\eta)-c_i^+(u)$ and $\omega^+_i=(1-u_i)\omega_i=\frac{\eta_i-u_i}{u_i}$.

Then
\begin{align}
V^+(u,\eta)&=\sum_{i\in \Tn} \theta^+_i\omega^+_i
\end{align}
To bound $V^+(u,\eta)$ we follow the method used by Jara and Menezes \cite{JM20} and by Funaki and Tsuneda \cite{FT19}.
For this purpose let us observe that the carré du champ referred to the generator $\cL_N$, namely, $\Gamma_N(h)=\cL_Nh^2-2 h\cL_Nh$ can be decomposed as 
\begin{equation}\label{eq:decCdC}
\Gamma_N(h)=\Gamma_N^\cG(h)+2\alpha N^2\Gamma_N^\cK(h),
\end{equation} 
where $\Gamma_N^\cG(h)$ and $\Gamma_N^\cK(h)$ are the carré du champ related to the generator $\cG_N$ and $\cK_N$ respectively (cf. \eqref{def:glauberGen} and \eqref{def:kawasakiGen}).
In particular,
\begin{equation}\label{def:cdcK}
\Gamma_N^\cK(h)=\frac12\sumtwo{i,j\in \Tn \colon }{|i-j|=1}(h(\eta^{i,j})-h(\eta))^2\,.
\end{equation}
In the next result we provide the control that we need for $V^{+}$.
\begin{theorem}\label{prop:UB_V^+}
Under Assumption \ref{ass1}, we have that for any $u:\Tn\to [0,1]$ such that
\begin{enumerate}
\item $\epsilon:=1-|||u|||_{\infty}>0$ with $|||u|||_{\infty}:=\min\Big\{\|u\|_\infty, \|1-u\|_\infty\Big\}$, 
\item $\|\nabla u(i)\|_\infty\le \frac{C_0}{N}$ with $C_0$ independent of $N$,
\end{enumerate}
and for any density $f$ with respect to $\gu_u$ we have that
\begin{equation}\label{eq:VB}
\int V^{+}(u,\eta)f(\eta)\gu_u(\dd\eta)
\leqslant \delta N^2 \int \Gamma^\cK_N\big(\sqrt f\big)(\eta) \gu_u(\dd \eta)+ C_1 \ell_\cV^{d}
 \cH(f\gu_u \, |\,  \gu _u)  +C_2 N^{d-a},
\end{equation}
for any $\delta>0$ and $C_1=3^dd (d+1)\frac6{\epsilon^2}\left(\frac1{\delta}+\frac{C_0}{\epsilon^2}\right)$, $C_2= \frac{6C_0}{\epsilon^4}$ for $N>\frac{2C_0}{\epsilon^2}$. 
\end{theorem}

\begin{remark}\label{rem:transfV+toV-}
Note that if the estimate \eqref{eq:VB} holds, then it holds for $V^-$. Indeed, using \eqref{eq:V-}, the fact that $|||u|||_{\infty}=|||1-u|||_{\infty}=1-\epsilon$ and that, under $\gu_{1-u}(\dd\eta)$, $1-\eta$ has for law $\gu_u$ we get
\begin{align*}
\int V^-(u,\eta)f(\eta)&\gu_u(\dd\eta)
=-\int V^{+}(1-u,1-\eta)f(\eta)\gu_u(\dd\eta)\\
&=\int V^{+}(1-u,\eta)f(1-\eta)\gu_{1-u}(\dd\eta)\\
&\leqslant  \delta N^2 \int \Gamma^\cK_N\big(\sqrt f\big)(1-\eta) \gu_{1-u}(\dd \eta)+ C_1 \ell_\cV^{d}
 \cH(f(1-\cdot)\gu_{1-u} \, |\,  \gu _{1-u})  +C_2 N^{d-a}\\
&= \delta N^2 \int \Gamma^\cK_N\big(\sqrt f\big)(\eta) \gu_{u}(\dd \eta)+ C_1 \ell_\cV^{d}
 \cH(f\gu_{u} \, |\,  \gu _{u})  +C_2 N^{d-a}.
 \end{align*}
 \end{remark}

\begin{proof}[Proof of Theorem \ref{prop:UB_V^+}]
The proof will proceed in several steps.

\smallskip

Note that, according to Hoeffeding Inequality (Lemma \ref{lem:hoeff}), under the probability measure $\gu_u$, $\theta^+_i$ is sub-Gaussian with variance parameter $\frac14$ and $\omega^+_i$ is sub-Gaussian with variance parameter $\frac1{4u_i^2}$.

As in Jara and Menezes \cite{JM20}, we proceed to use an averaged version of $V^+$. Let $\ell>0$ and consider $\Lambda_\ell=\{0\dots,\ell-1\}^d$ the cube of size $\ell$ starting at $0$. Let $p_\ell(i)=\ell^{-d}\1{i\in\Lambda_\ell}$ and $\hat p_\ell(i)=\ell^{-d}\1{i\in-\Lambda_\ell}=p_\ell(-i)$. Then for $\phi$ defined on $\Tn$, we set
\begin{align}
\leftmean{\phi}(i)&:=p_\ell*\phi(i)=\sum_{j+k=i}p_\ell(j)\phi(k)=\sum_{j}p_\ell(j)\phi(i-j)=\ell^{-d}\sum_{j\in\Lambda_\ell}\phi(i-j)\\
\rightmean{\phi}(i)&:=\hat p_\ell*\phi(i)=\sum_{j+k=i}p_\ell(-j)\phi(k)=\sum_{j}p_\ell(-j)\phi(i-j)=\ell^{-d}\sum_{j\in\Lambda_\ell}\phi(i+j).
\end{align}

Let us denote $q_\ell=p_\ell*p_\ell$. We have
$q_\ell(i)=\ell^{-2d}|\Lambda_\ell\cap(i-\Lambda_\ell)|.$
Thus, $0\leqslant q_\ell(i)\leqslant \ell^{-d}$ and $q_\ell(i)=0$ if and only if $i\notin \Lambda_{2\ell-2}$. 
All sum being finite, we get for $\phi$ and $\psi$ defined on $\Tn$:
\begin{align*}
\sum_{i,j} \phi(i) \psi(i+j)q_\ell(j)
&=\sum_{i,j,k} \phi(i) \psi(i+j)p_\ell(k)p_\ell(j-k)
=\sum_{i,k}\phi(i)p_\ell(k)\sum_j \psi(i+j)p_\ell(j-k)\\
&=\sum_{i,k}\phi(i)p_\ell(k)(\hat p_\ell*\psi) (k+i)
=\sum_j\sum_i \phi(i)p_\ell(j-i)(\hat p_\ell*\psi) (j)\\
&=\sum_j (p_\ell*\phi)(j)(\hat p_\ell*\psi) (j)=\sum_j \leftmean{\phi}(j)\rightmean{\psi} (j).
\end{align*}

Thus let $V^{+,\ell}:=\sum_{i\in\Tn}\theta^+_i\omega^{+,\ell}_i$ with $\omega^{+,\ell}_i=\sum_j\omega^{+}_{i+j}q_\ell(j)$.  
We have that
\begin{equation}
V^{+,\ell}
=\sum_{i,j\in\Tn}\theta^+_i\omega^{+}_{i+j}q_\ell(j)
=\sum_{i\in\Tn}\leftmean{\theta^+}_i\rightmean{\omega^{+}}_i.
\end{equation}

Then Theorem \ref{prop:UB_V^+} is proved with the following two estimates.
\begin{lemma}\label{lem:avg}
Suppose that $|||u|||_{\infty}<1$ and let $\epsilon=1-|||u|||_{\infty}$. 
Then for any $\ell>\ell_\cV$
\begin{equation}\label{eq:VBell}\int V^{+,\ell}f\gu_u(\dd\eta)
\leqslant C\epsilon^{-1}\ell_\cV^{d/2}
\left( \cH(f\gu_u \, |\,  \gu _u) + \frac{N^d}{\ell^d}\right).
\end{equation}
The constant $C$ depends only on the dimension, namely it can be taken as $C=3^d 2^{d/2+2}$.
\end{lemma}

\begin{lemma}\label{lem:diff}
Suppose $\epsilon=1-|||u|||_{\infty}>0$ and $\|\nabla u(i)\|_\infty\le \frac{C_0}{N}$ for some $C_0>0$ independent of $N$, where $\nabla u(i)=(\nabla_k u(i))_{k=1}^d$. 
Then, for any $\ell>\ell_\cV$, with $\ell=N^{\kappa}$ for some $\kappa>0$, and $\delta >0$
\begin{multline}\label{eq:VB-VBell}
\int (V^+-V^{+,\ell})  f\gu_u(\dd\eta)  \leqslant \delta N^2 \int \Gamma_N^\cK(\sqrt f)\, \gu_u(\dd \eta) \\ 
 +C_1\ell_\cV^d \frac{g_d(\ell)\ell^d}{N} \Bigg(\Big(2+\frac{\ell_\cV}{\ell}\Big)^d 
\cH\big(f\gu_u \, \big| \, \gu_u\big)+ \frac{N^d}{\ell^d}\Bigg) +C_2 N^{d-1}
\end{multline}
where $C_2=\frac{3C_0}{\epsilon^4}\left(1+\frac{2C_0}{N\epsilon^2}\right)$ and $C_1=3^d d(d+1)\left(\frac2{\delta \epsilon^2}\left(1+\frac{C_0}{N\epsilon^2}\right)+C_2\right)$ depends only on the dimension and $g_d(\ell)$ is defined in \eqref{g_d}.
\end{lemma}

Indeed, we choose $\ell$ such that $\frac{\ell^d g_d(\ell)}{N}\le C_0$: more precisely, for any $\delta_0\in (0,1)$ we take
\begin{align}\label{eq:choix_ell}
\ell= N^{\frac12(1-\delta_0)}\quad \text{for}\quad d =1, \qquad \text{and} \qquad
\ell=N^{\frac{1}{d}(1-\delta_0)} \quad \text{for}\quad d\ge 2.
\end{align}
Since $\ell_\cV^d$ has a $\log$-growth (cf. Assumption \ref{ass1}), this choice of $\ell$ together with Lemmas \ref{lem:avg} and \ref{lem:diff} concludes the proof of Theorem \ref{prop:UB_V^+}. 
\end{proof}
We now prove Lemma \ref{lem:avg} and Lemma \ref{lem:diff}. 

\vspace{-0.5cm}
\begin{proof}[Proof of Lemma \ref{lem:avg}] 
We start by recalling that $V^{+,\ell}=\sum_{i\in\Tn}\leftmean{\theta^+}_i\rightmean{\omega^{+}}_i$.
Note that, under $\gu_u$, using Lemma~\ref{lem:SGsum}, $\rightmean{\omega^{+}}_i$ is a sub-Gaussian variable with variance parameter $\sum_{i\in\Lambda_\ell}\frac1{4u_i^2\ell^{2d}}\leqslant \frac1{4\epsilon^2\ell^d}.$

 Since $c_i^+(\eta)$ is a function of $(\eta_{i+j})_{j\in\cV_N}$,  then for $i$ and $j$ such that $|i-j|_{\infty}>\ell_\cV$, $\theta^+_{i}$ and $\theta^+_{j}$ are independent, and sub-Gaussian with variance parameter $\frac14$. 
%
%
Using Lemma \ref{lem:SGsum}, $\leftmean{\theta^+}_i$ is a sub-Gaussian variable with variance parameter 
$\frac{\ell_\cV^d}{4\ell^{d}}\left(1+\frac{\ell_\cV}{\ell}\right)^d\leqslant 2^{d-2}\frac{\ell_\cV^d}{\ell^{d}}.$

We note that all the sites involved in the averages $\leftmean{\theta^+}_i\rightmean{\omega^{+}}_i$ are in $i+Q_{\ell_\cV/2+\ell}$, where $Q_m=\{-m, \ldots, m\}^d$ is the $d$-dimensional cube centered at $0$.
In particular, for $i$ and $j$ such that $|i-j|_{\infty}>\ell_\cV+2\ell$, the corresponding averages are independent (under $\gu_u$).
Then, we can take a partition of $\Tn$ into independent sites by letting $i=j+(\ell_\cV+2\ell)k$ where $j\in \Lambda_{\ell_\cV+2\ell}$ and $k\in \Lambda_{\lceil N/(\ell_\cV+2\ell)\rceil}$.
The entropy inequality (cf. (B.3) in \cite{JM20}) gives
\begin{align*}
\int V^{+,\ell}f\gu_u(\dd\eta)
&=\sum_{j\in \Lambda_{\ell_\cV+2\ell}}\int\sum_{k\in \Lambda_{\lceil N/(\ell_\cV+2\ell)\rceil}}\leftmean{\theta^+}_{j+(\ell_\cV+2\ell)k}\rightmean{\omega^+}_{j+(\ell_\cV+2\ell)k} f\gu_u(\dd\eta)\\
&\leqslant 
\sum_{j\in \Lambda_{\ell_\cV+2\ell}}\frac1{\gamma}\Bigg(\cH(f\gu_u \, |\,  \gu _u)\\
&\qquad\quad+
\sum_{k\in \Lambda_{\lceil N/(\ell_\cV+2\ell)\rceil}}\log\int\exp\Bigg\{\gamma \leftmean{\theta^+}_{j+(\ell_\cV+2\ell)k}\rightmean{\omega^+}_{j+(\ell_\cV+2\ell)k}\Bigg\}\gu_u(\dd \eta)\Bigg)\\
&\leqslant
\frac{(\ell_\cV+2\ell)^d}{\gamma}\cH(f\gu_u \, |\,  \gu _u) + \frac1\gamma\sum_{i\in \Tn}\log\int\exp\Big\{\gamma \leftmean{\theta^+}_{i}\rightmean{\omega^+}_{i}\Big\}\gu_u(\dd \eta)
\end{align*}

According to Lemma \ref{lem:SGprod}, for $\gamma^{-1}=2^{d/2+2}\ell^{-d}\epsilon^{-1}\ell_\cV^{d/2}$, we have that
\[\log\displaystyle\int\exp\Big\{\gamma  \leftmean{\theta^+}_{i}\rightmean{\omega^+}_{i}\Big\}\gu_u(\dd \eta)\leqslant\log 3.\]
Thus we obtain
\begin{equation*}
\int V^{+,\ell}f\gu_u(\dd\eta)
\leqslant 2^{d/2+2}\epsilon^{-1}\ell_\cV^{d/2}
\left(\left(2+\frac{\ell_\cV}{\ell}\right)^d \cH(f\gu_u \, |\,  \gu _u) + \frac{N^d}{\ell^d}\log 3\right).
\end{equation*}
\end{proof}

\begin{proof}[Proof of Lemma \ref{lem:diff}] 
We first prove the Lemma by also assuming that $\cV_N\subset \mathbb{Z}^d\setminus \mathbb{N}^d$. In Remark \ref{rem:remove-hyp-chiante} at the end of the proof we show how to remove this assumption in dimension $d>1$. In dimension $d=1$ the assumption  $\cV_N\subset \mathbb{Z}\setminus \mathbb{N}$ is necessary (cf. Assumption \ref{ass1}).

We then proceed as in Jara and Menezes \cite{JM20}. 
We use the fact that 
\begin{align}
\notag
V^+-V^{+,\ell}
&=\sum_{i\in\Tn}\theta^+_i(\omega^+_i-\omega^{+,\ell}_i)\\
\label{eq:V^+-V^+l}
&=\sum_{i,j\in\Tn}\theta^+_i\omega^{+}_{i+j}(\1{0}(j)-q_\ell(j)).
\end{align}

We now use Lemma 3.2 in \cite{JM20} stating that there exists a function $\Phi_\ell : \mathbb{Z}^d\times\mathbb{Z}^d\to\bbR$ which is a flow connecting the distribution $\1{0}$ to $q_\ell$, i.e
\begin{itemize}
\item $\Phi_\ell(i,j)=-\Phi_\ell(j,i)$
\item $\sum_{j\colon |i-j|=1}\Phi_\ell(i,j)=\1{0}(i)-q_\ell(i)$
\item $\Phi_\ell(i,j)=0$ for $i,j\notin \Lambda_{2\ell-1}$
\item there is a constant $C=C(d)$ independent of $\ell$ such that 
\begin{align}\label{eq:bondPhi}
\sum_{|i-j|=1}|\Phi_\ell(i,j)|^2<C g_d(\ell)
\qquad \text{and}\qquad 
\sum_{|i-j|=1}|\Phi_\ell(i,j)|<C \ell
\end{align}
where 
\begin{equation}\label{g_d}
g_d(\ell)=\left\{\begin{aligned}
\ell & \text{ for $d=1$}\\
\log(\ell )&\text{ for $d=2$}\\
1&\text{ for $d\geqslant 3$}
\end{aligned}\right.
\end{equation}
\end{itemize}
Using the flow $\Phi_\ell$ we therefore have 
\begin{align}
V^+-V^{+,\ell}
&=\sum_{i,j\in\Tn}\theta^+_i\omega^{+}_{i+j}\sum_{k\colon|j-k|=1}\Phi_\ell(j,k)\notag \\
&=\sum_{i,j\in\Tn}\theta^+_i\omega^{+}_{i+j}\sum_{k=1}^d\Big\{\Phi_\ell(j,j+e_k)-\Phi_\ell(j-e_k,j)\Big\}\notag\\
&=\sum_{k=1}^d\sum_{i,j\in\Tn}\theta^+_i\Phi_\ell(j,j+e_k)(\omega^{+}_{i+j}-\omega^{+}_{i+j+e_k})\notag\\
&=\sum_{k=1}^d\sum_{i,j\in\Tn}\theta^+_{i-j}\Phi_\ell(j,j+e_k)(\omega^{+}_{i}-\omega^{+}_{i+e_k})\notag\\
&=\sum_{k=1}^d\sum_{i\in\Tn}h^k_{i}(\omega^{+}_{i}-\omega^{+}_{i+e_k})\label{eq:V+-V+l}
\end{align}
where 
\begin{equation}\label{eq:def-h}
h_i^k=h_{i}^{\ell,k}:=\sum_{j\in\Tn}\theta^+_{i-j}\Phi_\ell(j,j+e_k)=\sum_{j\in\Lambda_{2\ell-1}}\theta^+_{i-j}\Phi_\ell(j,j+e_k).
\end{equation}
To complete the result we need to estimate \eqref{eq:V+-V+l}.
%
%
We apply Lemma 3.5 of \cite{FT19}, whose assumptions are satisfied in our case. Note that the hypothesis of Lemma 3.5 are satisfied in our case: $u_-=\epsilon$ and $u_+=1-\epsilon$ and $h_i^k(\eta^{i,i+e_k})=h_i^k(\eta)$ for any configuration $\eta$ inasmuch $h_i^k$ is only a function of the sites $\eta_{i-j+s}$, for $s\in\cV_N$, $j\in \Lambda_{2\ell-1}$, so that it does not depend on $\eta_i$ and $\eta_{i+e_k}$ since $\cV_N \subset \bbZ^d \setminus \bbN^d$. Moreover, we observe that Lemmas 3.4 and 3.5 of \cite{FT19} apply to our case replacing $\chi(u_i)$ by $u_i$.
Therefore, for any $\alpha'>0$ we have that
\begin{align}\notag
\int h^k_{i}(\omega^{+}_{i}-\omega^{+}_{i+e_k}) f \gu_u(\dd\eta) \le 
 \frac{\alpha'}{2} &\int \big(\sqrt{f(\eta^{i,i+e_k})}-\sqrt{f(\eta})\big)^2 \gu_u(\dd \eta)\\&+\frac{C_{1,\epsilon}}{\alpha'}\int (h_i^k)^2 f \gu_u(\dd \eta) + R_{i}^k.
 \label{eq:bound_h}
\end{align}
with $C_{1,\epsilon}=\frac2{\epsilon^2}\left(1+\frac{C_0}{N\epsilon^2}\right)$ and
where the rest term $R_i^k$ is controlled as
\begin{align}\label{eq:Rik}
R_i^k\le C_{2,\epsilon} \big| \,\nabla_k u_i\, \big| \int \big| h_i^k(\eta)\big| f \gu_u(\dd\eta)\,,
\end{align}
with, for $\epsilon<\frac12$, $C_{2,\epsilon} =\frac3{\epsilon^4}\left(1+\frac{2C_0}{N\epsilon^2}\right)$.
%
We now take $\alpha' = \delta N^2$, with $\delta>0$. We get (cf. \eqref{def:cdcK})
\begin{align}
\int (V^+-V^{+,\ell})f&\gu_u(\dd \eta)  =\sum_{k=1}^d\sum_{i\in\Tn}\int h^k_{i}(\omega^{+}_{i}-\omega^{+}_{i+e_k})f\gu_u(\dd \eta) \notag \\ 
& \le  \delta N^2 \int  
\Gamma_N^\cK(\sqrt f)\, \gu_u(\dd \eta)+\frac{C_{1,\epsilon}}{\delta N^2} \sum_{k=1}^d\sum_{i\in\Tn}\int (h_i^k)^2 f \gu_u(\dd \eta) + 
\sum_{k=1}^d\sum_{i\in\Tn}R_{i}^k. 
 \label{eq:bound_h_2}
\end{align}
The first term is the same of \eqref{eq:VB-VBell}, then to conclude the proof we have to upper bound the second and the third term of \eqref{eq:bound_h_2}.

Let us start from the third one. Using that $\big| \,\nabla_k u_i\, \big|\le \frac{C_0}{N}$ and $|h_i^k|\le 1+(h_i^k)^2$ in \eqref{eq:Rik} we get 
\begin{align*}
R_i^k\le \frac{C_0C_{2,\epsilon}}{N} \int \big(1+(h_i^k)^2\big) f \gu_u(\dd\eta)\,.
\end{align*}
 So that, the last term of \eqref{eq:bound_h_2} is bounded by \[C_0C_{2,\epsilon}\left( N^{d-1} + \frac{1}{N} \sum_{k=1}^d\sum_{i\in\Tn} \int (h_i^k)^2f \gu_u(\dd\eta)\right).\] Let us note that this last term dominates the second term of \eqref{eq:bound_h_2}. To conclude the proof we need to upper-bound 
\[
 \sum_{k=1}^d\sum_{i\in\Tn} \int (h_i^k)^2f \gu_u(\dd\eta).
\]
As we noted above, $h_i^k$ depends only on the sites $\eta_{i-j+s}$, for $s\in\cV_N$ and $j\in \Lambda_{2\ell-1}$, so that the random variables $h_{i}^k$ and $h_{i^\prime}^k$ are independent if $|i-i^\prime|_\infty>(2\ell+\ell_\cV)$. We then decompose $\Tn$ as a disjoint union of cubes of size $2\ell+\ell_\cV$, that is, 
we write $i=j+(\ell_\cV+2\ell)z$ where $j\in \Lambda_{\ell_\cV+2\ell}$ and $z\in\Lambda_{\lceil N/(\ell_\cV+2\ell)\rceil}$, 
\begin{align*}
\sum_{i\in\Tn} \int (h_i^k)^2f \gu_u(\dd\eta)=
 \sum_{j\in \Lambda_{2\ell+\ell_\cV}}\sum_{z\in\Lambda_{\lceil N/(\ell_\cV+2\ell)\rceil}} \int (h_{j+z}^k)^2f \gu_u(\dd\eta).
\end{align*}
We then apply the entropy inequality and we get 
\begin{equation}
\begin{split}\label{eq:entropy_eq}
\sum_{i\in\Tn} \int (h_i^k)^2f \gu_u(\dd\eta)
& \le \frac{1}{\gamma} \sum_{j\in \Lambda_{2\ell+\ell_\cV}}\Bigg(
\cH\big(f\gu_u \, \big| \, \gu_u\big)\\
&\qquad+ \log\int\exp\Bigg\{\gamma \sum_{z\in\Lambda_{\lceil N/(\ell_\cV+2\ell)\rceil}} (h_{j+z}^k)^2
\Bigg\}\, \gu_u(\dd \eta)
\Bigg)\\
&= \frac{(2\ell+\ell_\cV)^d}{\gamma} 
\cH\big(f\gu_u \, \big| \, \gu_u\big)+  
\frac{1}{\gamma}\sum_{i\in \Tn} \log\int\exp\Big\{\gamma (h_{i}^k)^2
\Big\}\, \gu_u(\dd \eta)
\end{split}
\end{equation}
 To conclude we use a concentration inequality. 
 We have that $h_{i}^k$ is sub-Gaussian random variable, let $\sigma^2$ be its variance parameter. By Proposition \ref{prop:subgauss} we have that for any $\gamma\le \frac{1}{4\sigma^2}$, 
\[
 \int\exp\Big\{\gamma (h_{i}^k)^2
\Big\}\, \gu_u(\dd \eta) \le \log 3\, .
\]
 Moreover, to get an upper bound on the variance parameter, we use the same decomposition of the sum \eqref{eq:def-h} into subsets which are independent (since $\theta_j^+$ is a function of $\eta_{i+j}$, for $i\in\cV_N$ and is sub-Gaussian with variance parameter $\frac14$). This is done in Lemma F.12 in \cite{JM20}. 
 We then have that $\sigma^2\le C_d \ell_\cV^d g_d(\ell)$, where $C_d$
 is a constant which depends only on the dimension.
 By getting $\gamma$ as large as possible, namely $\gamma^{-1}=(d+1) \ell_\cV^d \, g_d(\ell)$ we obtain that  
 \begin{align}\label{eq:bound_h^2}
 \sum_{i\in\Tn} \int (h_i^k)^2f \gu_u(\dd\eta)
 \le (d+1)\ell_\cV^d \, g_d(\ell)\ell^d\Bigg(\Big(2+\frac{\ell_\cV}{\ell}\Big)^d 
\cH\big(f\gu_u \, \big| \, \gu_u\big)+ \frac{N^d}{\ell^d}\log 3\Bigg).
 \end{align}
 
 To conclude the proof we have to remove the assumption $\cV_N\subset \bbZ^d\setminus\bbN^d$. We show that in Remark \ref{rem:remove-hyp-chiante}.
 \end{proof}

\begin{remark}\label{rem:remove-hyp-chiante}
We now show how it is possible to remove the assumption $\cV_N\subset \mathbb{Z}^d\setminus \mathbb{N}^d$ if $d>1$.

By \eqref{eq:V+-V+l} we recall that
$V^+-V^{+,\ell}=\sum_{k=1}^d\sum_{i\in\Tn}h^k_{i}(\omega^{+}_{i}-\omega^{+}_{i+e_k})$
where $h_i^k$ is defined in \eqref{eq:def-h}.
 We also observe that, by assumption, $\ell >\ell_\cV$. 
Since $h_i^k$ is a function of the sites $\eta_{i-j+s}$, for $s\in\cV_N$, $j\in \Lambda_{2\ell-1}$, whenever $|j|>\ell_\cV$, then $\theta^+_{i-j}$ does not depend on $\eta_i$ and $\eta_{i+e_k}$.
Therefore we split $h_i^k$ into the sum of two functions $h_i^\prime$ and $h_i^\second$. The first one is independent of $\eta_i$ and $\eta_{i+e_k}$, while the second one depends on these sites, 
\begin{align}
h_i^\prime:=\sumtwo{j\in\Lambda_{2\ell-1}}{|j|>2\ell_\cV}\theta^+_{i-j}\Phi_\ell(j,j+e_k), 
\qquad \text{and}\qquad 
h_i^\second:=\sumtwo{j\in\Lambda_{2\ell-1}}{|j|\le 2\ell_\cV}\theta^+_{i-j}\Phi_\ell(j,j+e_k).
\end{align}
In such a way 
\begin{align}\label{eq:splitV+-V+l}
V^+-V^{+,\ell}&=(V^+-V^{+,\ell})^\prime + (V^+-V^{+,\ell})^\second,
\end{align}
where
\begin{align}
(V^+-V^{+,\ell})^\prime :=\sum_{k=1}^d\sum_{i\in\Tn}h^\prime_{i}(\omega^{+}_{i}-\omega^{+}_{i+e_k}) 
\quad \text{and}\quad 
(V^+-V^{+,\ell})^\second :=\sum_{k=1}^d\sum_{i\in\Tn}h^\second_{i}(\omega^{+}_{i}-\omega^{+}_{i+e_k})
\end{align}
We can apply the method used above to $(V^+-V^{+,\ell})^\prime$ obtaining that \eqref{eq:bound_h^2} holds.
We control $(V^+-V^{+,\ell})^\second $ by showing that $\int (V^+-V^{+,\ell})^\second f\gu_u(\dd \eta) =O(N^{d-\epsilon})$, for some $\epsilon>0$.
For this purpose we apply Cauchy-Swartz inequality to the measure $f \gu_u(\dd \eta)$, which gives
\begin{align}\label{eq:h_second_bound}
\int h^\second_{i}(\omega^{+}_{i}-\omega^{+}_{i+e_k}) f \gu_u(\dd \eta) & \le 
\Big(\int (h^\second_{i})^2 f \gu_u(\dd \eta) \Big)^\frac12  \Big(\int(\omega^{+}_{i}-\omega^{+}_{i+e_k})^2 f \gu_u(\dd \eta)\Big)^\frac12   
\end{align}
Therefore
\begin{align}
\notag
\int (V^+-V^{+,\ell})^\second& f\gu_u(\dd \eta)  =\sum_{k=1}^d\sum_{i\in \Tn}\int h^\second_{i}(\omega^{+}_{i}-\omega^{+}_{i+e_k}) f \gu_u(\dd \eta) \\
\notag
& \le 
\sum_{k=1}^d\sum_{i\in \Tn}  \Big(\int (h^\second_{i})^2 f \gu_u(\dd \eta) \Big)^\frac12  \Big(\int(\omega^{+}_{i}-\omega^{+}_{i+e_k})^2 f \gu_u(\dd \eta)\Big)^\frac12
\\
\label{eq:h''cauchy}
& \le 
\left(\sum_{k=1}^d\sum_{i\in \Tn}  \Big(\int (h^\second_{i})^2 f \gu_u(\dd \eta) \Big)
\sum_{k=1}^d\sum_{i\in \Tn} \Big(\int(\omega^{+}_{i}-\omega^{+}_{i+e_k})^2 f \gu_u(\dd \eta)\Big) \right)^\frac12,
\end{align}
where in the last inequality we used again Cauchy-Swartz inequality.

We observe that $(\omega^{+}_{i}-\omega^{+}_{i+e_k})=
(\frac{\eta_{i}}{u_i}-\frac{\eta_{i+e_k}}{u_{i+e_k}})=\eta_i(\frac{1}{u_i}-\frac{1}{u_{i+e_k}})+ \frac{1}{u_{i+e_k}}(\eta_{i}-\eta_{i+e_k}),$ so that 
\begin{align*}
\int(\omega^{+}_{i}-\omega^{+}_{i+e_k})^2 f \gu_u(\dd \eta) &\le 
2 \int\eta_i\Big(\frac{1}{u_i}-\frac{1}{u_{i+e_k}}\Big)^2 f \gu_u(\dd \eta)+2 \int\frac{1}{u_{i+e_k}^2}(\eta_{i}-\eta_{i+e_k})^2 f \gu_u(\dd \eta)\\
&\le C_\epsilon|\nabla u_k|_\infty^2+C_\epsilon \int(\eta_{i}-\eta_{i+e_k})^2 f \gu_u(\dd \eta)\le C^\prime_\epsilon,
\end{align*}
uniformly on $i$ and $N$. So that, 
\begin{equation}
\sum_{k=1}^d\sum_{i\in \Tn} \Big(\int(\omega^{+}_{i}-\omega^{+}_{i+e_k})^2 f \gu_u(\dd \eta)\Big) \le C^\prime_\epsilon N^d.
\end{equation}
Therefore, by \eqref{eq:h''cauchy}, to conclude the proof it is enough to show that 
\begin{equation}
\sum_{k=1}^d\sum_{i\in \Tn}  \Big(\int (h^\second_{i})^2 f \gu_u(\dd \eta) \Big)=O(N^{d-\epsilon}),
\end{equation}
for some $\epsilon>0$ small.
For this purpose we have to look more carefully at the function $\Phi_\ell(j, j+e_k)$ which defines the flow. 
We recall (cf. Appendix G of \cite{JM20}) that in the construction of the flow connecting $\1{0}$ and $p_\ell*p_\ell$,
we first define a flow $\Psi_\ell$ connecting $\1{0}$ and $p_\ell$ supported in $\Lambda_\ell$ which satisfies \eqref{eq:bondPhi} and then we define
\begin{align*}
\Phi_\ell(j,j+e_k):=
\sum_{i\in \Tn}\Psi_\ell(i,i+e_k)p_\ell(j-i) ,
\end{align*}
therefore
\begin{align*}
\sumtwo{j\in \Lambda_{2\ell-1}}{|j|\le 2\ell_\cV}\big|\Phi_\ell(j,j+e_k)\big|&\le 
\sumtwo{j\in \Lambda_{2\ell-1}}{|j|\le 2\ell_\cV}\sum_{i\in \Tn}\big|\Psi_\ell(i,i+e_k)\big|p_\ell(j-i)\\
&\le \sum_{i\in \Tn}\big|\Psi_\ell(i,i+e_k)\big| \sumtwo{j\in \Lambda_{2\ell-1}}{|j|\le 2\ell_\cV} p_\ell(j-i)
\le C \ell \frac{\ell_\cV^d}{\ell^d},
\end{align*}
where we used that $\sumtwo{j\in \Lambda_{2\ell-1}}{|j|\le 2\ell_\cV} p_\ell(j-i)\le \frac{\ell_\cV^d}{\ell^d}$ uniformly on $i$ and \eqref{eq:bondPhi} applied to $\Psi_\ell$.
We deduce that, since $|\theta^+_{i-j}|\le 2$,
\begin{align*}
\sum_{k=1}^d\sum_{i\in \Tn}  \Big(\int (h^\second_{i})^2 f \gu_u(\dd \eta) \Big)\le 
C N^d \frac{\ell_\cV^{2d}}{\ell^{2(d-1)}}.
\end{align*}
Since $\ell=N^{\kappa}$ for some $\kappa>0$ and $\ell_\cV^d$ grows logarithmically, the result follows (in dimension $d\ge 2$).

\begin{remark}
 Let us observe that in dimension $d=1$ we obtain that the last term is at least $O(N^{d})$, so our approach does not allow to remove this assumption.
\end{remark}
\end{remark}

\subsubsection{\bf Estimate of $V$} 
We show that  under the hypothesis of Theorem \ref{prop:UB_V^+}, $V(u,\eta)$ satisfies \eqref{eq:VB}.
For $i\in \Tn$ and $k\in \{1, \ldots, d\}$ we let 
\begin{align*}
\tilde \omega_i^{k}:=-\alpha N^2(u_i-u_{i+e_k})^2\omega_i.
\end{align*}
In such a way we get 
\begin{align*}
V(u,\eta)&=-\alpha N^2 \sumtwo{i, j \in \Tn}{|i-j|=1}(u_i-u_j)^2\omega_i\omega_j\\
&=-\alpha N^2 \sum_{k=1}^d\sum_{i\in \Tn}\Bigg\{(u_i-u_{i+e_k})^2\omega_i\omega_{i+e_k}
+ (u_i-u_{i-e_k})^2\omega_i\omega_{i-e_k}\Bigg\}\\
&= \sum_{k=1}^d\sum_{i\in \Tn}\Bigg\{\tilde \omega_i^{k} \omega_{i+e_k} 
+ \tilde \omega_{i-e_k}^{k} \omega_{i} \Bigg\}=2 \sum_{k=1}^d\sum_{i\in \Tn} \tilde \omega_{i-e_k}^{k} \omega_{i} .
\end{align*}
For $k\in \{1, \ldots, d\}$ we let $\tilde V^k:=\sum_{i\in \Tn} \tilde \omega_{i-e_k}^{k} \omega_{i}$ and $\tilde V^{k, \ell}:=   \sum_{i\in \Tn}  \leftmean{ \tilde \omega_{i-e_k}^{k} }\rightmean{\omega_i}$. 
We have that $\tilde V^{k, \ell}$ and $\tilde V^k-\tilde V^{k, \ell}$ satisfy \eqref{eq:VBell} and \eqref{eq:VB-VBell} respectively, which implies that $V(u,\eta)$ satisfies \eqref{eq:VB}. The proof follows the same ideas of $V^+$ and $V^-$ and it is actually simpler since we do not have to deal with the diameter of $\cV_N$. We omit the details.

\subsubsection{\bf Conclusion: Gronwall's inequality \eqref{eq:goalJ_t}}
We observe that since $\Gamma_N^\cK(\sqrt f)\le \Gamma_N(\sqrt f)$, cf. \eqref{eq:decCdC} and that the carré du champ operator is non-negative, \eqref{eq:goalJ_t} is a consequence of Proposition \ref{pro:YauInequality} and Theorem \ref{prop:UB_V^+} with $u(t)=u^N(t)$, the solutions of \eqref{eq:discretPDE}, $f=\frac{\dd \mu_t}{\dd \gu_{u(t)}}$ if we show that, like the initial profile, $u^N$ satisfies Assumption \ref{ass2} for any $t\in [0,\tau]$. This is one of the goal of the Section \ref{sec:estimatesolu^N}, see Propositions \ref{prop:u^Nin01} and \ref{prop:gradient_est}.

\section{Estimates on the solutions $(u^N)$ of \eqref{eq:discretPDE}}
\label{sec:estimatesolu^N}




We say that $u$ is supersolution of \eqref{eq:discretPDE} if $\partial_tu_i\geqslant2\alpha N^2 (\Delta u)_i+\beta(1-2u_i)+G(i,u)$ and that it is a subsolution if $\partial_tu_i\leqslant2\alpha N^2 (\Delta u)_i+\beta(1-2u_i)+G(i,u)$.

Note that any solution is both a super- and a subsolution. We have a comparaison lemma between super and subsolutions.
\begin{proposition}\label{prop:subsuper_sol_G}
Let $u$ be a supersolution, and $v$ be a subsolution such that $u(0,i)\geqslant v(0,i)$ for all $i\in\Tn$. Then, for all $t\geqslant 0$ and all $i\in\Tn$, $u(t,i)\geqslant v(t,i)$.
\end{proposition}

\begin{proof}
Since $u$ and $v$ have derivative in time they are continuous, consider $i$ and $t$ such that $u(t,i)=v(t,i)$ and $u(t,j)\geqslant v(t,j)$. Then, we have $G(i,u(t))\geqslant G(i,v(t))$ by Proposition \ref{prop:G1}, and thus
\begin{align} 
\partial_t(u(i,t)-v(i,t))
&=2\alpha N^2\sum_{j,|i-j|=1}(u(j,t)-v(j,t))+G(i,u(t))-G(i,v(t))
\geqslant 0.
\end{align}
This proves that $u$ stays above $v$ at all time.
\end{proof}

From the two previous propositions we conclude:
\begin{proposition}\label{prop:u^Nin01}
Let $\delta=\frac{\beta}{1+4\beta}$ and suppose $0<\delta<\frac1{4e}$. Let also $4\delta\leqslant T\leqslant 1-4\delta$ and $0\leqslant\epsilon<2\delta$. For all $K_N>\frac{|\log(\epsilon/2)|}{\epsilon^2}$, and all $N$, consider the solution $(u^N(t,i))_{i\in \Tn}$ of Equation \eqref{eq:discretPDE} starting from $u_0\in[\epsilon,1-\epsilon]$, then we have that $u^N(t,i)\in[\epsilon,1-\epsilon]$ for all $i$ and $t\geq 0$.
\end{proposition}
Note that the proposition holds for $\epsilon=0$.

\begin{proof}
Let $p\in[0,1]$, and set $u(i)=p$, for all $i$. Using that $g_{K_N}(p)=G(i,p)$, we have that
\begin{itemize}
\item if $g_{K_N}(p)+\beta(1-2p)\geqslant 0$, then $u$ is a subsolution;
\item if $g_{K_N}(p)+\beta(1-2p)\leqslant 0$, then $u$ is a supersolution.
\end{itemize}

Then, by using the result of Proposition \ref{prop:g_K1} on the analysis of $g_K(p)$ close to $p=0$ and $p=1$, it is easy to check that $u(i)=\epsilon$ is a subsolution and $u(i)=1-\epsilon$ is supersolution, for $\epsilon$ satisfying the hypothesis of the proposition.

In particular, for $p=1$, $g_{K_N}(1)=0$ , we have a supersolution. For $p=0$, $g_{K_N}(0)=0$, we have a subsolution.
\end{proof}
%

In the next result, we show that if $u$ solves \eqref{eq:discretPDE} and $\|\nabla u(0,i)\|_\infty\le \frac{C}{N}$, then $\|\nabla u(t,i)\|_\infty\le \frac{C\sqrt t}{N}$ for any $t>0$.

\begin{proposition}\label{prop:gradient_est}
Let $u$ be a solution of \eqref{eq:discretPDE} with $u(0)=u_0$ such that, there exists $C_0>0$ for which $|\nabla u_0|_{\infty}\leqslant \frac{C_0}{N}$. Then, there exists $C>0$ such that  $\|\nabla u(t,i)\|_\infty\le \frac{C_0+C\sqrt t}{N}$ for all $t\geqslant 0$.
\end{proposition}
To prove Proposition \ref{prop:gradient_est} we follow \cite{FT19} and the reference therein, in particular \cite{DD05}. Using the notation of \cite{DD05} we let $p(t,x,z)$ be the heat kernel of discrete Laplacian 
\begin{equation}\label{def:discrLaplacian}
\Delta u(t,i)=\sum\limits_{j\in \mathbb Z^d} a(i, i+j)\big[u(t, i+j)-u(t,i)\big],
\end{equation} with $a(i,j)=\ind_{\{j-i \in \Gamma\}}$ and $\Gamma=\{\pm e_i, \, i=1, \dots, d\}$, that is, 
\[
\begin{cases}
p(\cdot,i_0,i)=\delta_{i_0}(\cdot), \\
\partial_t p(t,i_0,i) = \Delta p (t, i_0,i).
\end{cases}
\]
In this case (comments below  (1.2) of \cite{DD05} for the definition and discussion of $a^*$ and $p^*$) we have that $a^*(t,i,j)=a(t,i,j)$ so that $p^*(t,i,j)=p(t,i,j)$. Let $\nabla_{k} u(t,i)=u(t, i+e_k)-u(t, i)$, then since $p$ is a uniform transitions function, there exist $c,C>0$ independent of $t, k$ such that (cf. (1.3) of \cite{DD05})
\begin{equation}\label{eq:nablap_bound_base}
\big| \nabla_{k} p(t,0,i)\big| \le C\frac{p(ct, 0, i)}{\sqrt{1\vee t}}.
\end{equation}
We refer to \cite{DD05}, Section 4 and the reference therein for a proof. We stress that the delicate point is to extend the classical theory of E. De Giorgi, J. Nash and J. Moser to discrete operators. The authors follow mainly \cite{Del99}, but similar results can be also found in \cite{GOS01} (Appendix B) and \cite{SZ97}.

\begin{proof}[Proof of Proposition \ref{prop:gradient_est}]
Using Duhamel's formula (i.e., variation of constant) we get that 
\begin{equation}
u(t,i)=\sum_{j\in\Tn}u(0,j)p_N(t,i,j) +\int_0^t \dd s \sum_{j\in \Tn}(\beta(1-2u_j)+G_j(u))p_N(t-s,i,j),
\end{equation}
where $p_N(t,i,j)=\sum_{z\in N\mathbb Z^d}p(2\alpha N^2 t,i,j+z)$ is the heat kernel of discrete Laplacian on the torus speeds up by a factor $2\alpha N^2$ and $p(t,i,j)$ is the heat kernel introduced above. We observe that  \eqref{eq:nablap_bound_base} gives \begin{equation}\label{eq:nablap_bound}
\big| \nabla_{k} p_N(t,0,i)\big| \le \frac{C}{N}\frac{p_N(ct, 0, i)}{\sqrt{t}}.
\end{equation}
For the first term, using that $p_N(t,i,j)=p_N(t,i+z,j+z)$ for any $i,j,z\in \Tn$ and the assumption on $\nabla u_0$ (cf. Assumptions \ref{ass2}), an integration by parts gives
\begin{align*}
  \bigg| \nabla_k \bigg\{ \sum_{j\in\Tn}u(0,j)p_N(t,i,j)\bigg\} \bigg| 
&=\bigg|  \sum_{j\in\Tn}u(0,j)\big(p_N(t,i+e_k,j) -p_N(t,i,j)\big)\bigg| \\
&=\bigg|  \sum_{j\in\Tn}u(0,j)p_N(t,i,j-e_k) - \sum_{j\in\Tn}u(0,j) p_N(t,i,j)\bigg| \\
&=\bigg|  \sum_{j\in\Tn}\Big(u(0,j+e_k)-u(0,j)\Big)p_N(t,i,j)\bigg|
\le \frac{C}{N}.
\end{align*}

By using that $(\beta(1-2u_j)+G_j(u))$ is bounded and \eqref{eq:nablap_bound} we get that for any $k=1, \dots, d$,
\begin{align*}
& \bigg| \nabla_k \bigg\{ \int_0^t \dd s \sum_{j\in \Tn}(\beta(1-2u_j)+G_j(u))p_N(t-s,i,j) \bigg\} \bigg|\le \frac{C}{N}\int_0^t \frac{1}{{\sqrt{t-s}}}\dd s = \frac{C}{N}\sqrt t \,.
\end{align*} 
\end{proof}

\section{Existence and uniqueness of reaction-diffusion PDE}\label{sec:reaction-diff}
In Section \ref{sec:conv_discrPDE}, we prove that in the limit $N\to\infty$, the solution of the discretized Equation \ref{eq:discretPDE}, with $u^N(t=0)=u^N_0\in[0,1]^{\Tn}$ satisfying Assumption \ref{ass2},
converges to a solution of the scalar nonlinear reaction diffusion equation 
\begin{itemize}
\item when $K_N\to K$ as $N\to\infty$:
\begin{equation}\label{eq:limitePDE-K}
\partial_tu(t,x)=2\alpha\Delta u(t,x)+\beta(1-2u(t,x))+g_{K}(u(t,x)) ,
\end{equation}
\item when $K_N\to \infty$ as $N\to\infty$:
\begin{equation}\label{eq:limitePDE}
\partial_tu(t,x)=2\alpha\Delta u(t,x)+\beta(1-2u(t,x))+g_{\infty}(u(t,x)).
\end{equation}
\end{itemize}

The main difference between the two equations is that in the first case ($K<+\infty$), $g_K$ is a $C^1$ function on $[0,1]$ (thus Lipschitz) so the reaction diffusion equation \eqref{eq:limitePDE-K} is very classical, whereas for the second case ($K=+\infty$), since $g_\infty$ is  not even continuous we need to consider \eqref{eq:limitePDE}  as a subdifferential inclusion.

\smallskip
The main results of this section are 
\begin{enumerate}[(i)]
\item Proposition \ref{prop:existence_sol_PDE} which proves existence of a solution, in a suitable sense, for the equation \eqref{eq:limitePDE},
\item Proposition \ref{prop:v_t-v_0} which proves local uniqueness of the solution, in a suitable sense, for the equation \eqref{eq:limitePDE}, starting from a suitable class of initial conditions,
\item and Theorem \ref{thm:discrconv} which proves that all accumulation points of $(u^N)_N$ is a solution, in a suitable sense, for the equation \eqref{eq:limitePDE}.
\end{enumerate}
%


In the rest of this section we change our notations and define $v=2u-1$. We center the solution around the constant steady state $u=\frac12$. It simplifies the presentation and proofs of our results. The original form of our equations can be retrieved by letting $u=\frac12(v+1)$. In such a way \eqref{eq:limitePDE-K} takes the form
\begin{equation}\label{eq:limitePDE-Kbis}
\partial_tv(t,x)-2\alpha\Delta v(t,x)+2\beta v-2g_K\left(\frac12(v+1)\right)=0.
\end{equation}

\subsection{Solution of \eqref{eq:limitePDE-K}}
Let us denote $s(t,x,y)$ the semigroup of the operator $\frac12\Delta$ on $\bbT^d$, that is,
\begin{equation}
s(t,x,y)=\frac{1}{(2\pi t)^{d/2}}\sum_{k\in\mathbb{Z}^d}\exp\left(-\frac{\|x-y-k\|^2}{2 t}\right).
\end{equation}
Denote also $s_0(t,x,y)$ the semigroup of the operator $\frac12\Delta$ on $\mathbb{R}^d$,
\begin{equation}
s_0(t,x,y)=\frac{1}{(2\pi t)^{d/2}}\exp\left(-\frac{\|x-y\|^2}{2 t}\right).
\end{equation}
Note that $\xi\mapsto s_0(t,0,\xi)$ is the density of $d$ independent normal random variables with variance $t$.

Let us consider $(S^{\lambda, \gamma}_t)$ the semigroup on $L^{1}(\bbT^d)$ defined by, for $f\in L^{1}(\bbT^d)$, $\lambda\geqslant 0$ and $\gamma >0$
\begin{equation}\label{lapl_semigroup}
S^{\lambda, \gamma}_tf(x)
=\int_{\bbT^d}e^{-\lambda t}s(\gamma t,x,y)f(y)\dd y
=\int_{\bbR^d}e^{-\lambda t}s_0(\gamma t,x,y)\tilde f(y)\dd y,
\end{equation}
where  for a measurable function $f$ on $\mathbb{T}^d$, we denoted $\tilde f$ its extension on $\mathbb R^d$ defined by $\tilde f(x)=f(x-\lfloor x\rfloor)$.

Another way to define $S^{\lambda, \gamma}_t$ is to use the Brownian motion: denote by $X$ a Brownian motion on $\mathbb{R}^d$ starting from $x$ on some probability space $(\Omega,\cF, \mathbb{P}_x)$, indeed we have $S^{\lambda, \gamma}_tf(x)=e^{-\lambda t}\mathbb{E}_x(\tilde f(X_{\gamma t}))$, and for all $\lambda \geqslant 0$, and $\gamma >0$, $S^{\lambda, \gamma}$ is a $C_0$-contraction semigroup on $L^p(\bbT^d)$ for $p\in[1,+\infty]$.

As we will look at \eqref{eq:limitePDE-K} in its mild form, the following result is crucial to study the regularity of the solution.

\begin{proposition}\label{prop:mildprop}
For $v_0\in  L^{\infty}(\bbT^d)$ and $g\in L^{\infty}([0,\tau]\times \mathbb{T}^d)$, define
\begin{equation}\label{eq:mild}
v(t,x):=S^{\lambda, \gamma}_t v_0(x)+\int_0^tS^{\lambda, \gamma}_{t-s}(g(s,\cdot))(x)\dd s.
\end{equation} 
then $v\in C([0,\tau], \bbT^d)$.
We have the following estimates, for all $(t,x)\in\mathbb{R}^+\times\bbT^d$,
\begin{align}\label{eq:v_bonded}
|v(t,x)|&\leqslant e^{-\lambda t} \|v_0\|_\infty+\frac1{\lambda}(1-e^{-\lambda t})\|g\|_\infty\leqslant  \|v_0\|_\infty+\frac1{\lambda}\|g\|_\infty 
\end{align}
and for all $\tau>0$, there exists a constant $C$ depending only on $\tau,\gamma, \lambda$ and $d$, such that for all $(t,x), (s,y)\in[1/\tau,\tau]\times \mathbb{T}^d$ with $s<t$
\begin{align}\label{eq:modofcont}
|v(t,x)-v(s,y)|&\leqslant C((t-s)|\log(t-s)|+\|x-y\|)(\|g\|_\infty + \|v_0\|_\infty).
\end{align}
\end{proposition}

\begin{remark}
$v$ is called a mild solution of the equation $\partial_t v-\frac{\gamma}{2} \Delta v+\lambda v=g$ with initial value $v_0$.
The fact that a mild solution is a classical solution if $g$ is sufficiently regular is a result from Pazy (\cite{pazy}, Corollary 4.2.5). 
\end{remark}
 Estimates \eqref{eq:v_bonded} and \eqref{eq:modofcont} are quite standard but we include the proof for the sake of completeness.
\begin{proof}[Proof of Proposition \ref{prop:mildprop}]
The fact that  $v\in C([0,\tau],L^{\infty}( \bbT^d))$ is a consequence of the fact that $S^{\lambda,\gamma}$ is a $C_0$ contraction semigroups on $L^{\infty}$.
For the first estimate \eqref{eq:v_bonded}, we have that:
\begin{align*}
|v(t,x)|
&\leqslant |S^{\lambda, \gamma}_t v_0(x)|+\left|\int_0^tS^{\lambda, \gamma}_{t-u}(g(u,\cdot))(x)\dd u\right|\\
&\leqslant e^{-\lambda t} \|v_0\|_{\infty}\int_{\bbR^d}s_0(\gamma t,x,z)\dd z +\|g\|_{\infty}\int_0^s e^{-\lambda (t-u)}\int_{\bbR^d}s_0(\gamma (t-u),x,z)\dd z\dd u\\
&=e^{-\lambda t} \|v_0\|_\infty+\frac1{\lambda}(1-e^{-\lambda t})\|g\|_\infty.
\end{align*}

For the estimate \eqref{eq:modofcont}, we start by letting $s<t$, we have 
\begin{align*}
|v(t,x)-v(s,y)|
&\leqslant |S^{\lambda, \gamma}_t v_0(x)- S^{\lambda, \gamma}_s v_0(y)|\\
&\qquad+\left|\int_0^tS^{\lambda, \gamma}_{t-u}(g(u,\cdot))(x)\dd u-\int_0^sS^{\lambda, \gamma}_{s-u}(g(u,\cdot))(y)\dd u\right|\\
&\leqslant I_1 \|v_0\|_{\infty} +(I_2+I_3)\|g\|_{\infty}
\end{align*}
where 
\begin{align*}
I_1&=\int_{\bbR^d} \left|e^{-\lambda t}s_0(\gamma t,x,z)-e^{-\lambda s}s_0(\gamma s,y,z)\right|\dd z,\\
I_2&=\int_0^s\int_{\bbR^d} \left|e^{-\lambda (t-u)}s_0(\gamma (t-u),x,z)-e^{-\lambda (s-u)}s_0(\gamma (s-u),y,z)\right|\dd z\dd u,\\
I_3&=\int_s^t\int_{\bbR^d} e^{-\lambda (t-u)}s_0(\gamma (t-u),x,z)\dd z\dd u.
\end{align*}
We have that $I_3\leqslant t-s$.
For $I_1$ and $I_2$, we use the fact that, for $i=1\dots d$
\begin{align*}
\partial_t s_0(t,0,\xi)&=\frac12\Delta_\xi s_0(t,0,\xi)=\frac1{2}\left(\frac{\|\xi\|^2}{t^2}-\frac{d}{t}\right)s_0(t,0,\xi)\\
\partial_{\xi_i}s_0(t,0,\xi)&=-\frac{\xi_i}{t}s_0(t,0,\xi).
\end{align*}
First for $I_1$, denote $c_1(r)=r t+(1-r)s$ and $c_2(r)=r (z-x)+(1-r) (z-y))=z-(r x+ (1-r) y)$ for $r\in[0,1]$, 
We have that 
\begin{align*}
|(\partial_t s_0)(\gamma c_1(r),0,c_2(r))|&\leqslant \frac1{2}\left(\frac{\| c_2(r)\|^2}{\gamma^2 c_1(r)^2}+\frac{d}{\gamma c_1(r)}\right)s_0(\gamma c_1(r),0,c_2(r))\\
|(\partial_{\xi_i} s_0)(\gamma c_1(r),0,c_2(r))|&\leqslant \frac{|c_2(r)_i|}{\gamma c_1(r)}s_0(\gamma c_1(r),0,c_2(r))
\end{align*}
Therefore
\begin{align*}
I_1&=\int_{\bbR^d} \left|e^{-\lambda t}s_0(\gamma t,0,z-x)-e^{-\lambda s}s_0(\gamma s,0,z-y)\right|\dd z\\
&\leqslant \int_{\bbR^d} \int_0^1 |c'_1(r)\partial_t(e^{-\lambda t}s_0 (\gamma t,0,\xi))_{|t=c_1(r),\xi=c_2(r)}|\dd r\dd z\\
&\quad+\int_{\bbR^d} \int_0^1 \sum_{i=1}^d|c'_2(r)_i\partial_{\xi_i}(e^{-\lambda t}s_0 ( \gamma t,0,\xi))_{|t=c_1(r),\xi=c_2(r)}|\dd r\dd z\\
&\leqslant I_{1,1}+I_{1,2},
\end{align*}
where we have, by applying Fubini and a change of variable 
\begin{align*}
I_{1,1}
&= (t-s)\int_{\bbR^d} \int_0^1 \left[\lambda e^{-\lambda c_1(r)}+\frac{\gamma}{2}\left(\frac{\| c_2(r)\|^2}{c_1(r)^2}+\frac{d}{c_1(r)}\right)\right]s_0 ( \gamma c_1(r),0,c_2(r))\dd r\dd z\\
&\leqslant (t-s)\int_0^1\left[\lambda e^{-\lambda c_1(r)}+\frac{\gamma}{2}\left(\frac{d \gamma c_1(r)}{\gamma ^2c_1(r)^2}+\frac{d}{\gamma c_1(r)}\right)\right]\dd r\\
&\leqslant (t-s)\int_0^1\left[\lambda e^{-\lambda c_1(r)}+\frac{ d}{c_1(r)}\right]\dd r
=\int_{s}^{t} \left(\lambda e^{-\lambda u}+\frac{ d}{u}\right)\dd u\\
&=e^{-\lambda s}(1-e^{-\lambda (t-s)})+d\log\left(\frac{t}s\right)
\leqslant \left(\lambda+\frac{d}{s}\right)(t-s).
\end{align*}
For the term $I_{1,2}$, we get, using Cauchy-Schwarz, Fubini and a change of variable
\begin{align*}
I_{1,2}
&=\int_{\bbR^d} \int_0^1 \sum_{i=1}^d|y_i-x_i|e^{-\lambda c_1(r)} \frac{|c_2(r)_i|}{ \gamma c_1(r)}s_0(\gamma c_1(r),0,c_2(r))\dd r\dd z\\
&\leqslant \|x-y\| \int_0^1\frac{ e^{-\lambda c_1(r)}}{ \gamma c_1(r)}\int_{\bbR^d}  \|c_2(r)\|s_0(\gamma c_1(r),0,c_2(r))\dd z\dd r
= \|x-y\| \int_0^1 \frac{ e^{-\lambda c_1(r)}}{ \sqrt{\gamma c_1(r)}}\dd r C_1\\
&\leqslant C_1\|x-y\|\int_s^t \frac{e^{-\lambda u}}{\sqrt{\gamma u}}\frac{\dd u}{t-s}
\leqslant \frac{C_1}{\sqrt{\gamma s}} \|x-y\|
\end{align*}
where $C_1$ is the expectation of the quadratic norm of $X=(X_1,X_2,\dots X_d)$ of $d$ independent standard normal variables: $C_1=\mathbb{E}(\|X\|)\leqslant \mathbb{E}(\|X\|^2)^{1/2}=\sqrt{d}$ (we also have $C_1=\sqrt{2}\frac{\Gamma((d+1)/2)}{\Gamma(d/2)}\sim \sqrt{d}$).
We get that, for $s\geqslant \frac1T$,
\begin{align*}
I_1
&\leqslant \left(\lambda+\frac{d}{s}\right)(t-s)+\sqrt{\frac{d}{\gamma s}} \|x-y\|
\leqslant  \left(\lambda+d \tau\right)(t-s)+\sqrt{\frac{d \tau}{\gamma }} \|x-y\|.
\end{align*}

For $I_2$, we make the same computations with $c_1(r)=r (t-u)+(1-r)(s-u)=r t+(1-r)s-u$ where $u\in[0,s]$, we have, since $c_1'(r)=t-s$,
\begin{align*}
I_2
&\leqslant \int_0^s\int_{s-u}^{t-u} \left(\lambda e^{-\lambda v}+\frac{ d}{v}\right)\dd v\dd u+C_1\|x-y\|\int_0^s\int_{s-u}^{t-u} \frac{e^{-\lambda v}}{\sqrt{\gamma v}}\frac{\dd v}{t-s}\dd u
\end{align*}
For the first integral, we have
\begin{align*}
 \int_0^s\int_{s-u}^{t-u} \left(\lambda e^{-\lambda v}+\frac{ d}{v}\right)&\dd v\dd u
=(1-e^{-\lambda (t-s)})\int_0^s e^{-\lambda (s-u)}\dd u+d\int_0^s\log\left(\frac{t-u}{s-u}\right)\dd u\\
&=\frac1{\lambda}(1-e^{-\lambda(t-s)})(1-e^{-\lambda s})\\&\qquad+d\left[t\log(t)-s\log(s)-(t-s)\log(t-s)\right]\\
&\leqslant \lambda s (t-s) +d\left[s(\log(t)-\log(s))+(t-s)(\log(t)-\log(t-s))\right]\\
&\leqslant \lambda s (t-s) +d(t-s)+d(t-s)(|\log(t)|+|\log(t-s))|)\\
&\leqslant (\lambda s+d+|\log(t)|) (t-s) +d(t-s)|\log(t-s)|.
\end{align*}
For the second integral, we get
\begin{align*}
 \int_0^s\int_{s-u}^{t-u} \frac{e^{-\lambda v}}{\sqrt{\gamma v}}\dd v\dd u
&\leqslant (t-s)\int_0^s\frac1{\sqrt{\gamma (s-u)}}\dd u
=(t-s)\sqrt{\frac{2s}{\gamma}}.
\end{align*}
Then we obtain, for $1/\tau\leqslant s<t\leqslant \tau$
\begin{align*}
I_2
&\leqslant (\lambda s+d+|\log(t)|) (t-s) +d(t-s)|\log(t-s)|+C_1\|x-y\|\sqrt{\frac{2s}{\gamma}}\\
&\leqslant (\lambda \tau+d+|\log(\tau)|) (t-s) +d(t-s)|\log(t-s)|+\sqrt{\frac{2d\tau}{\gamma}}\|x-y\|.
\end{align*}
At last, we get the following estimate
\begin{align*}
&|v(t,x)-v(s,y)|
\leqslant \left [\left(\lambda+d \tau\right)(t-s)+\sqrt{\frac{d \tau}{\gamma }} \|x-y\|\right]\|v_0\|_{\infty}\\
&\qquad+\left[ (\lambda\tau+|\log(\tau)|+d+1) (t-s) +d(t-s)|\log(t-s)|+\sqrt{\frac{2d\tau}{\gamma}}\|x-y\|\right]\|g\|_{\infty}.
\end{align*}
\end{proof}

We modify a little our equation \eqref{eq:limitePDE-K}, both in order to obtain a sharper estimate on the uniform norm of the solution and to get a coherent notation with the solution of the limit equation when $K\to+\infty$.

We define for $q\in[-1,1]$,
\begin{align}
r_{K}(q)
&:=-\int_{\frac12}^{\frac12(q+1)}4g_K(s)ds.
\end{align}
We have that
\begin{align}
r'_{K}(q)=-2g_K\left(\frac12(q+1)\right)=-(1-q)\mathbb{P}_{\frac{1-q}2}\big[X<\kappa(K,T)\big] +(1+q)\mathbb{P}_{\frac{1+q}2}\big[X\leqslant \kappa(K,T)\big].
\end{align}

So we let $h_K$ be
\begin{equation}
h_K(q)=\left\{
\begin{aligned}
-1&\text{ for $q<-1$}\\
-r'_K(q)+q&\text{ for $q\in[-1,1]$}\\
1&\text{ for $q>1$}
\end{aligned}\right.
\end{equation}
Since $r'_K(1)=0$ and $r'_K(-1)=0$, $h_K$ is continuous on $\mathbb{R}$.
We have that, for $q\in[-1,1]$,
\begin{align*}
h_K(q)=-r'_{K}(q)+q
=(1-q)\mathbb{P}_{\frac{1-q}2}\big[X<\kappa(K,T)\big] -(1+q)\mathbb{P}_{\frac{1+q}2}\big[X\leqslant \kappa(K,T)\big]+q.
\end{align*}

We now solve the following equation
 \begin{equation}\label{eq:limitePDE-Kter}
\partial_tv(t,x)-2\alpha\Delta v(t,x)+(2\beta+1) v=h_{K}(v(t,x)).
\end{equation}
Note that, this equation and \eqref{eq:limitePDE-Kbis} are exactly the same with the term $v$ added on both sides if $\|v\|_{\infty}\leqslant 1$, since for $q\notin[-1,1]$, $h_K(q)\neq q$. Thus, a solution $v$ of \eqref{eq:limitePDE-Kbis} with $\|v\|_{\infty}\leqslant 1$ will also be a solution of  \eqref{eq:limitePDE-Kter} and reciprocally.

\begin{proposition}\label{prop:solK}
For $v_0\in  L^{\infty}(\bbT^d)$ with $\|v_0\|_{\infty}\leqslant 1$,  there exists a unique solution $(v(t,x), t\ge 0, x\in\bbT^d)$ to the problem
\begin{itemize}
\item $v$ is continuous from $\mathbb{R}_+^*$ to $L^{\infty}(\bbT^d)$
\item $v$ satisfies, for all $t>0$ and $x\in\bbT^d$ 
\begin{equation}\label{eq:mildK}
v(t,x)=S^{2\beta+1,4\alpha}_t v_0(x)+\int_0^t S^{2\beta+1,4\alpha}_{t-s}[h_{K}(v(s,\cdot))](x)\dd s.
\end{equation}
\end{itemize}
We say that $v$ is a mild solution to \eqref{eq:limitePDE-Kter}. We have also that $\|v\|_{\infty}\leqslant 1$ and $v$ satisfies
\begin{equation}\label{eq:mildK2}
v(t,x)=S^{2\beta,4\alpha}_t v_0(x)-\int_0^t S^{2\beta,4\alpha}_{t-s}[r'_{K}(v(s,\cdot))](x)\dd s
\end{equation}
and thus is a mild solution of \eqref{eq:limitePDE-Kbis}.
\end{proposition}

\begin{proof} The first part of the proposition comes from a fixed point argument (see also Pazy \cite{pazy}, Theorem 6.1.2) applied to the following functional. Let $\tau>0$ and define the functional $F:C(]0,\tau],L^{\infty}(\bbT^d))\to  C(]0,\tau],L^{\infty}(\bbT^d))$ defined by
\begin{equation}
F(v)(t,x):=S^{2\beta+1,4\alpha}_t v_0(x)+\int_0^t S^{2\beta+1,4\alpha}_{t-s}[h_{K}(v(s,\cdot))](x)\dd s.
\end{equation}
We equip $C(]0,\tau],L^{\infty}(\bbT^d))$ with the uniform topology on all compact subset. We can apply the Banach fixed point Theorem to $F$ (see the proof of Pazy \cite{pazy}, Theorem 6.1.2.  Moreover, the mapping $v_0\mapsto v$ is Lipschitz continuous from $L^{\infty}$ to $C(]0,\tau],L^{\infty}(\bbT^d))$.
An application of Proposition \ref{prop:mildprop} proves that $\|v\|_\infty\leqslant 1$.

Since $h_K$ is differentiable, then, if $v_0\in C^2(\bbT^d)$, $v_0$ is in the domain of $\Delta$ and thus $v$ is  classical solution of  \eqref{eq:limitePDE-Kter} (Theorem 6.1.5 \cite{pazy}). Thus $v$ is a classical solution of  \eqref{eq:limitePDE-Kbis} since $\|v\|_{\infty}\leqslant 1$ and a mild solution of \eqref{eq:limitePDE-Kbis}. Now consider an approximating sequence $(v_{0,n})$ in $C^2(\bbT^d)$ of $v_0\in L^{\infty}$, and $(v_n)$  the sequence of mild solutions with initial value $v_{0,n}$ and $v$ the mild solution of \eqref{eq:limitePDE-Kter} with initial value $v$. Then, since $v_0\mapsto v$ is Lipschitz continuous, by the dominated convergence theorem, we get that, uniformly on $[t_0,\tau]\times\bbT^d$ for all $t_0>0$, the right hand side of
\begin{equation}
v_n(t,x)=S^{2\beta,4\alpha}_t v_{0,n}(x)-\int_0^t S^{2\beta,4\alpha}_{t-s}[r'_{K}(v_n(s,\cdot))](x)\dd s
\end{equation}
converges to $S^{2\beta,4\alpha}_t v_{0}(x)-\int_0^t S^{2\beta,4\alpha}_{t-s}[r'_{K}(v(s,\cdot))](x)\dd s$, whereas the left hand side converges to $v$. So we obtain that $v$ is  a mild solution of  \eqref{eq:limitePDE-Kbis}.
\end{proof}

\subsection{Solution of \eqref{eq:limitePDE}}
\subsubsection{\textbf{Existence of a solution}}
We use the same transform as before, and let $h_{\infty}$ be the pointwise limit of $h_K$ given below in \eqref{eq:hinfty}. The equation \eqref{eq:limitePDE} is now formally
 \begin{equation}\label{eq:limitePDE-2}
\partial_tv(t,x)-2\alpha\Delta v(t,x)+(2\beta+1) v=h_{\infty}(v(t,x)).
\end{equation}

For the limiting equation, we prove first that the family $(v_K)_K$ of solutions associated to $h_K$ with common initial value $v_0\in L^{\infty}(\bbT^d)$ in $C(\mathbb{R}^+_*\times\mathbb{\bbT^d})$ is compact (with uniform norm on all compact subset). Then, by taking the limit, any accumulation point $v_{\infty}$ of the sequence satisfy the mild formulation of the limiting equation relaxed as a subdifferential inclusion.

In order to prove this, we set some notations:
\begin{align}
r_{\infty}(q)
&:=-\int_{\frac12}^{\frac12(q+1)}4g_{\infty}(s)\dd s
\end{align}
$h_{\infty}$, the pointwise limit of $h_K$, is the function on $\mathbb{R}$
\begin{align}\label{eq:hinfty}
h_\infty(q)&=2g_{\infty}\left(\frac{1+q}2\right)+q
=-\mathbbm{1}_{q\leqslant -2\rho}+q\mathbbm{1}_{-2\rho<q\leqslant 2\rho}+\mathbbm{1}_{q>2\rho}.
\end{align}
Then $h_\infty$ is non-decreasing and is the left-derivative of the convex function $H_\infty$
\begin{align}
H_\infty(q)&:=-r_\infty(q)+\frac{q^2}2\\
\notag
&=[-q+2\rho+2\rho^2]\1{q\leqslant -2\rho}+\frac12 q^2\1{-2\rho<q\leqslant 2\rho}+ [q-2\rho+2\rho^2]\1{2\rho<q}
\end{align}

The subdifferential of $H_\infty$ at $q$ is defined as 
\[\partial H_\infty(q)=\{p\in \mathbb{R}, H_\infty(q')-H_\infty(q)\geqslant p(q'-q), \text{ for all $q'\in[-1,1]$}\}.\]
In particular we have,
\begin{align}
\partial H_\infty(q)
&=\left\{\begin{aligned}
&\{-1\}&&\text{ for $q<-2\rho$},\\
&[-1,-2\rho]&&\text{ for $q=-2\rho$},\\
&\{q\}&&\text{ for $-2\rho<q<2\rho$},\\
&[2\rho,1]&&\text{ for $q=2\rho$},\\
&\{1\}&&\text{ for $q>2\rho$}.\\
\end{aligned}\right.
\end{align}
We adopt the following definition for a solution of the equation \eqref{eq:limitePDE-2}:
 \begin{equation}\label{eq:limitePDE-3}
\partial_tv(t,x)-2\alpha\Delta v(t,x)+(2\beta+1) v\in \partial H_\infty(v(t,x)).
\end{equation}

\begin{definition}\label{def:subdiff}
We say that $v$ is a mild solution of Equation \eqref{eq:limitePDE-3} if it satisfies, for some $\tau>0$ and all $t\leqslant \tau$
\begin{equation}
v(t,x)=S^{2\beta+1,4\alpha}_t v_0(x)+\int_0^t S^{2\beta+1,4\alpha}_{t-s}[w(s,\cdot)](x)\dd s.
\end{equation}
where $w\in L^2([0,\tau]\times\bbT^d)$, with $w(t,x)\in\partial H_\infty(v(t,x))$ almost everywhere.
\end{definition}

\begin{proposition}\label{prop:existence_sol_PDE}
For $v_0\in  L^{\infty}(\bbT^d)$, any accumulation point (in $C(\mathbb{R}_+^*\times \bbT^d)$ equipped with uniform norm on each compact set), of the sequence $(v_K)$ of solutions given by Proposition \ref{prop:solK} is a mild solution of \eqref{eq:limitePDE-3}.
\end{proposition}
As a consequence of the proposition, there exists $(v(t,x), t\ge 0, x\in\bbT^d)$ a mild solution of \eqref{eq:limitePDE-3} such that $v$ is continuous from $\mathbb{R}_+^*$ to $L^{\infty}(\bbT^d)$, and $\|v\|_\infty\leqslant 1$.

The existence of a solution for a given initial condition $v_0$ is not difficult and can be proved in different ways. Here we adopt some kind of regularization procedure, since we have a natural family of differentiable functions (namely the $(h_K)$) approximating $h_\infty$ and we use the convergence of the sequence $(v_K)$ in the next section to prove the convergence of  the stochastic process. We also present the proof because we need its arguments in order to prove the Theorem \ref{thm:discrconv}.

In the Remark \ref{rem:monotone}, we present another construction of solution(s) using the monotonicity of $h_\infty$ which is interesting since it also gives an insight on the problem of non-uniqueness.

\begin{proof}
For each $K$, we have a mild solution $v_K$ from Proposition \ref{prop:solK}. From Proposition \ref{prop:mildprop}, we have that each solution is uniformly bounded, uniformly continuous on $[1/\tau, \tau]\times\bbT^d$, and the modulus of continuity only depends on $\tau>1$ (since the others parameters are fixed). 

Therefore, by the Arzela-Ascoli Theorem, the sequence $(v_K)_K$ is compact on the space $C([1/\tau,\tau]\times \bbR^d)$ and we can extract a subsequence converging uniformly in $C([1/\tau,\tau]\times \bbR^d)$, and then by a diagonal argument, a sequence converging to a limit $v_\infty$ in  $C(]0,\infty[\times \bbR^d)$, uniformly on each compact. Note that since $\|v_K\|\leqslant 1$ for all $K$, we also have $\|v_{\infty}\|\leqslant 1$. We show that $v_\infty$ satisfies \eqref{eq:limitePDE-3}.

Let us assume that $\rho>0$.
Denote for any $s>0$, $\cA_\rho(s):=\{y\in \bbR^d \colon |v_\infty(s,y)|= 2\rho\}$.
Note that since $h_\infty$ is uniformly continuous on $[-1,-2\rho[$, $]-2\rho,2\rho[$ and $]2\rho,1]$, we have that for any $y\notin\cA_\rho(s)$, $\lim_{K\to\infty}h_K(v_K(s,y))=h_\infty(v_\infty(s,y))$. Let $y\in\cA_\rho(s)$ and assume $v_\infty(s,y)=2\rho$ without loss of generality, since $v_K(s,y)$ converges to $v_\infty(s,y)$, for all $\epsilon>0$ such that $2\rho-\epsilon >0$, there exists $K_0$ such that for all $K\geqslant K_0$, $2\rho-\epsilon<v_K(s,y)<2\rho+\epsilon$.
Thus, using Lemma \ref{lemma:h_Kbounded}, we have that, for all $K\geqslant K_0$, $2\rho-2\epsilon\leqslant h_K(v_K(s,y))\leqslant 1$.
Then taking the limits in $K$ and $\epsilon\to0$, we get
\begin{align}\label{eq:lim_inf_sup_h}
h_\infty(2\rho)=2\rho\leqslant \liminf_{K\to\infty} h_K(v_K(s,y))\leqslant \limsup _{K\to\infty} h_K(v_K(s,y))\leqslant 1=h_\infty(2\rho^+).
\end{align}

For $\rho=0$, we have the same inequality since then $h_{\infty}(0)=-1$ and $h_\infty(0+)=1$.

Let $w_+(s,y)=\limsup _{K\to\infty} h_K(v_K(s,y))$ and $w_-(s,y)=\liminf_{K\to\infty} h_K(v_K(s,y))$, thus we have that for all $(s,y)\in]0,+\infty[\times\bbT^d$:
\begin{align}
h_\infty(v_\infty(s,y))\leqslant w_-(s,y)\leqslant  w_+(s,y)\leqslant h_\infty(v_\infty(s,y)^+).
\end{align}

Since $h_K(v_K)$ is bounded, by the Banach Alaoglu Theorem, the sequence is weakly compact in $L^2(]0,\tau[\times\bbT^d)$, and we have a subsequence of $(h_K(v_K))_K$ converging weakly to $w\in L_{loc}^2(]0,+\infty[\times\bbT^d)$. Since the density of the semigroup $S^{2\beta+1,4\alpha}$ is in $L^2(]0,\tau]\times \bbR^d)$ for all $T>0$, we have as $K\to+\infty$,
\begin{align}
v_\infty(t,x)&= S^{2\beta+1,4\alpha}_t v_0(x)+\int_0^t S^{2\beta+1,4\alpha}_{t-s}w(s,\cdot)(x)\dd s.
\end{align}
Moreover, $w_-$ and $w_+$ are bounded and therefore in $L^2(]0,\tau]\times\bbT^d)$. Let $\phi\in L^2(]0,\tau]\times \bbT^d)$ and $\phi\geqslant 0$, by the Fatou Lemma we get
\begin{align}\notag
0=\int_{]0,\tau]\times\bbT^d}&\liminf_{K\to+\infty}(h_K(v_K)-w_-)\phi\\&\leqslant \liminf_{K\to+\infty}\int_{]0,\tau]\times\bbT^d}(h_K(v_K)-w_-)\phi=\int_{]0,\tau]\times\bbT^d}(w-w_-)\phi.
\end{align}
We also have
\begin{align}\notag
0=\int_{]0,\tau]\times\bbT^d}&\liminf_{K\to+\infty}(w_+ -h_K(v_K))\phi\\&\leqslant \liminf_{K\to+\infty}\int_{]0,\tau]\times\bbT^d}(w_+ - h_K(v_K))\phi=\int_{]0,\tau]\times\bbT^d}(w_+-w)\phi.
\end{align}
Thus, almost everywhere on $]0,+\infty[\times\bbT^d$, we have that
\begin{align}\label{eq:lim_inf_sup_w}
h_\infty(v_\infty)\leqslant w_-\leqslant w\leqslant  w_+\leqslant h_\infty(v_\infty^+).
\end{align}
Therefore, $w\in\partial H_\infty(v_\infty)$ a.e.
\end{proof}

\subsubsection{\textbf{Uniqueness of solution of \eqref{eq:limitePDE-3}}}

The main problem concerns the uniqueness of a solution.  We prove first that, we do not have uniqueness for a constant initial condition $v_0(x)=2\rho$ when $2\rho<\frac1{1+2\beta}$, so we are in the case of segregation or metastable segragation described by Figure \ref{fig1}.
\begin{remark}\label{rem:non-unique}
We describe three possible solutions starting from the initial condition $v_0(x)=2\rho$ when $2\rho<\frac1{1+2\beta}$.

Note that 
\begin{align}
S^{2\beta+1,4\alpha}_t v_0(x)=\int_{\mathbb{R}^d}e^{-(2\beta +1)t}s_0(t,x,y)2\rho\dd y=2\rho e^{-(2\beta +1)t}.
\end{align}
Suppose that $v$ does not depend on $x$, $v(t,x)=c(t)$ for all $x\in\bbT^d$, we have
\begin{align}
\int_0^t S^{2\beta+1,4\alpha}_{t-s}[h_{\infty}(v(s,\cdot))](x)\dd s
&=\int_0^t h_{\infty}(c(s)) e^{-(2\beta +1)(t-s)}\dd s.
\end{align}

Let us consider the functions 
\begin{align}
v^1(t,x)=c^1(t)&=2\rho e^{-2\beta t}\\\notag
v^2(t,x)=c^2(t)&=2\rho e^{ -(2\beta+1)t}+\frac1{1+2\beta}\left(1-e^{ -(2\beta+1)t}\right)\\&=2\rho+\left(\frac1{1+2\beta}-2\rho\right) \left(1-e^{ -(2\beta+1)t}\right).
\end{align}
Since, $c^1(t)\in[0,2\rho[$, for $t>0$, we have $h_{\infty}(c^1(t))=c^1(t)$ and then 
\begin{align}
\int_0^t h_{\infty}(c^1(s)) e^{-(2\beta +1)(t-s)}\dd s
&=\int_0^t 2\rho e^{-2\beta s}e^{-(2\beta +1)(t-s)}\dd s=2\rho e^{-(2\beta +1)t}(e^{t}-1).
\end{align}
Therefore
\begin{align}\notag
S^{2\beta+1,4\alpha}_t v_0(x)+\int_0^t S^{2\beta+1,4\alpha}_{t-s}[h_{\infty}(v^1(s,\cdot))](x)\dd s
&=2\rho e^{-(2\beta +1)t}+2\rho e^{-(2\beta +1)t}(e^{t}-1)\\&=2\rho e^{-2\beta t}=v^1(t,x).
\end{align}

Since $2\rho<\frac1{1+2\beta}$, $c^2(t)\in]2\rho,1]$, for $t>0$, we have $h_{\infty}(c^2(t))=1$ and then by the same computation,  
\begin{align}
\int_0^t h_{\infty}(c^2(s)) e^{-(2\beta +1)(t-s)}\dd s
&=\frac1{1+2\beta}(1-e^{ -(2\beta+1)t}).
\end{align}
Therefore, we also have
\begin{align}
S^{2\beta+1,4\alpha}_t v_0(x)+\int_0^t S^{2\beta+1,4\alpha}_{t-s}[h_{\infty}(v^2(s,\cdot))](x)\dd s
&=v^2(t,x).
\end{align}
Thus, both $v^1$ and $v^2$ are mild solutions to \eqref{eq:limitePDE-2} and thus to  \eqref{eq:limitePDE-3} with the same initial conditions. Note that at $t=0$, we have $v^1(0,x)=v^2(0,x)=2\rho$ and
\begin{align}
\partial _t v^1(0,x)&=-4\beta \rho=-(2\beta +1)2\rho+2\rho =-(2\beta +1)2\rho+h_\infty(2\rho^-)\\
\partial _t v^2(0,x)&=-2\rho (2\beta+1)+1=-(2\beta +1)2\rho+h_\infty(2\rho^+).
\end{align}
We see that non uniqueness comes from the fact that at $t=0$, where $v(t,x)=2\rho$, we have at least two choices for the derivative due to the fact that $h_\infty$ is not continuous.

Note that if we consider the mild solution to the subdiffrential inclusion \eqref{eq:limitePDE-3}, then we have at least a third solution: $v^3(t,x)=2\rho$. We consider $w(t,x)=2\rho(2\beta+1)$, we have
\begin{align}
S^{2\beta+1,4\alpha}_t v_0(x)+\int_0^t S^{2\beta+1,4\alpha}_{t-s}[w(s,\cdot)](x)\dd s
&=2\rho e^{-(2\beta +1)t}+2\rho(2\beta+1)\int_0^t e^{-(2\beta +1)(t-s)}\dd s\\
&=2\rho.\notag
\end{align}
Since, $2\rho<\frac1{1+2\beta}$, we have that $2\rho<2\rho(2\beta+1)<1$ so $w(t,x)\in\partial H_{\infty}(2\rho)$.
\end{remark}
Therefore, we cannot expect uniqueness for all initial condition, we have to impose some condition on the initial condition if we want a unique solution.

In the literature, we can find different conditions ensuring that the solution of Equation \eqref{eq:limitePDE-3} is unique. Adapting \cite{gianni} and \cite{deguchi}, we prove that the regularity of the initial condition at the levels where the non-linearity $h_\infty$ is not continuous is sufficient.
\begin{definition}
A function $v_0:\bbT^d\to [-1,1]$ in $C^1(\bbT^d)$ is \emph{regular at level $q\in]-1,1[$} if for all $x\in\bbT^d$, such that $v_0(x)=q$, we have $\nabla v_0(x)\neq 0$.
\end{definition}

\begin{proposition}\label{prop:v_t-v_0}
For $v_0\in C^1(\bbT^d)$, such that $\nabla v_0$ is Lipschitz on $\bbT^d$ and regular at levels $2\rho$ and $-2\rho$, the solution $v$ to Equation \eqref{eq:limitePDE-3} is locally unique. Moreover, the Lebesgue measure of the set $\cA_\rho(s):=\{y\in \bbR^d \colon |v(s,y)|=2\rho\}$ is zero.
\end{proposition}
We adapt two arguments by \cite{gianni} and  \cite{deguchi}.

\begin{lemma}\label{lem:estimatesv0}
If $v$ is a mild solution of \eqref{eq:limitePDE-3} with $v_0\in C^1(\bbT^d)$, and such that $\nabla v_0$ is Lipschitz on $\bbT^d$, then, for all $\tau>0$, there exists a constant $C>0$ such that, for all $t\in[0,\tau]$
\begin{align}
\|v(t)-v_0\|_\infty&\leqslant C t^{1/2},&\|\nabla v(t)-\nabla v_0\|_\infty&\leqslant C t^{1/2}.
\end{align}
\end{lemma}

\begin{proof} Since $v$ is  a mild solution of \eqref{eq:limitePDE-3} there exists $w\in L^\infty_{loc}([0+\infty[\times\bbT^d)$ with $\|w\|_\infty\leqslant 1$ since $w\in\partial H_\infty (v)$. Thus we get
\begin{align}
v(t)-v_0=S_t^{2\beta+1,4\alpha}v_0-v_0+\int_0^t S_{t-s}^{2\beta+1,4\alpha}w(s)\dd s.
\end{align}
Then, we have for the last integral
\begin{align*}
\left\|\int_0^t S_{t-s}^{2\beta+1,4\alpha}w(s)\dd s\right\|_\infty\leqslant \int_0^te^{-(2\beta+1)(t-s)}\dd s\leqslant t.
\end{align*}
We have also
\begin{align*}
|S_t^{2\beta+1,4\alpha}v_0(x)-v_0(x)|
&\leqslant  e^{-(2\beta+1)t}\int_{\bbT^d}s(4\alpha t,x,y)|v_0(y)-v_0(x)|\dd y\\
&\leqslant L e^{-(2\beta+1)t}\int_{\bbR^d}s_0(4\alpha t,x,y)\|y-x\|_2\dd y\\
&\leqslant L e^{-(2\beta+1)t}\sqrt{4\alpha t d}
\leqslant L\sqrt{4\alpha d}t^{1/2}
\end{align*}
where $L$ is the Lipschitz constant of $v_0$ and the third inequality comes from the computation of the upper bound of the quadratic norm of $d$ independent random variables with common variance $4\alpha t$.
Therefore we obtain
\begin{align}
\|v(t)-v_0\|_\infty&\leqslant (L\sqrt{4\alpha d}+\sqrt{\tau}) t^{1/2}
\end{align}
For the second bound, we use the fact that $\frac{\dd}{\dd x_i}s_0(t,x,y)=-\frac{\dd}{\dd y_i}s_0(t,x,y)$, thus we have, using an integration by parts
\begin{align}
\partial_{x_i}S_t^{2\beta+1,4\alpha}(v_0)=S_t^{2\beta+1,4\alpha}(\partial_{x_i}v_0).
\end{align}
Then, we have
\begin{align}\notag
\partial_{x_i} v(t)-\partial_{x_i}v_0=S_t^{2\beta+1,4\alpha}(\partial_{x_i}v_0)&-(\partial_{x_i} v_0)\\&+\int_0^t \int_{\mathbb{R}^d}e^{-(2\beta +1)(t-s)}\partial_{x_i}s_0(4\alpha (t-s),x,y)w(s,y)\dd y\dd s.
\end{align}
We can treat both integrals as before, for the last integral:
\begin{align*}
&\left|\int_0^t \int_{\mathbb{R}^d}e^{-(2\beta +1)(t-s)}\partial_{x_i}s_0(4\alpha (t-s),x,y)w(s,y)\dd y\dd s\right|
\\&\qquad\leqslant 
\int_0^t \frac{e^{-(2\beta +1)(t-s)}}{4\alpha(t-s)}\int_{\mathbb{R}^d}|x_i-y_i|s_0(4\alpha (t-s),x,y)\dd y\dd s\\
&\qquad\leqslant
\int_0^t \frac{e^{-(2\beta +1)(t-s)}}{4\alpha(t-s)}\sqrt{2\alpha(t-s)}\dd s\\
&\qquad\leqslant \frac{1}{\sqrt{2\alpha}}\int_0^{\sqrt{t}}  e^{-(2\beta +1)u^2}\dd u
\leqslant \frac{\sqrt{t}}{\sqrt{2\alpha}}.  
\end{align*}
For the first integral, we have the same estimates as before
\begin{align}
|S_t^{2\beta+1,4\alpha}(\partial_{x_i}v_0)(x)-\partial_{x_i}v_0(x)|
\leqslant L'\sqrt{4\alpha d}t^{1/2}
\end{align}
where $L'$ is the maximum of the Lipschitz constants of $(\partial_{x_i}v_0)_i$.
Thus, we get 
\begin{align}
\|\nabla v(t)-\nabla v_0\|_\infty&\leqslant \left(L'\sqrt{4\alpha d}+\frac1{\sqrt{2\alpha}}\right) t^{1/2}
\end{align}
\end{proof}

We now prove Proposition \ref{prop:v_t-v_0}.
\begin{proof} Let us assume that we have two solutions, $v_1$ and $v_2$ and let $e(t)=\|v_1-v_2\|_{L^{\infty}([0,t]\times\mathbb{R}^d)}$. Note that the previous Lemma entails that, for all $\tau>0$, there exists $C$, such that for $t<\tau$, we have $e(t)\leqslant C\sqrt{t}$.

We define $I^+_{s,t}=\{(s,y), s\leqslant t, |v_1(s,y)-2\rho|\leqslant e(t)\}$ and $I^-_{s,t}=\{(s,y), s\leqslant t, |v_1(s,y)+2\rho|\leqslant e(t)\}$.

Since $v_1$ and $v_2$ are solutions of \eqref{eq:limitePDE-3}, there exists $w_1$ and $w_2$ such that $w_1\in\partial H_{\infty}(v_1)$ a.e. and $w_2\in\partial H_{\infty}(v_2)$. We can decompose each $w_i$ as $w_i(t,x)=f_\infty(v_i(t,x))+g_i(t,x)$ where $f_\infty$ is the continuous part of $\partial H_\infty$:
\begin{align}
f_\infty(q)=\left\{
\begin{aligned}
-2\rho&\text{ for $q\in[-1,-2\rho]$}\\
q&\text{ for $q\in[-2\rho,2\rho]$}\\
2\rho&\text{ for $q\in[2\rho,1]$}
\end{aligned}\right.
\end{align}
and $g_i(t,x)=w_i(t,x)-f_\infty(v_i(t,x))$. Note that 
\begin{align}
g_i(t,x)&=-1+2\rho \text{ a.e. on $\{(t,x), v_i(t,x)<-2\rho\}$},\notag\\ 
g_i(t,x)&=0 \text{ a.e. on $\{(t,x), -2\rho <v_i(t,x)<2\rho\}$},\\ 
g_i(t,x)&=1-2\rho \text{ a.e. on $\{(t,x), v_i(t,x)>2\rho\}$},\notag
\end{align}
since $w_i=h_\infty(v_i)$ a.e. on $\{(t,x), |v_i(t,x)|=2\rho\}$.

As a consequence we have that, up to a negligible set, $\{(s,y), s\leqslant t, g_1(s,y)\neq g_2(s,y)\}\subset I^+_{s,t}\cup I^-_{s,t}$ since,  $g_1(s,y)\neq g_2(s,y)$ entails that one of the following inequalities is true $v_2(s,y)<2\rho<v_1(s,y)$ or $v_2(s,y)<-2\rho<v_1(s,y)$ or $v_1(s,y)<2\rho<v_2(s,y)$ or $v_1(s,y)<-2\rho<v_2(s,y)$. For each case, the inclusion is true: for the first one for example, if $s\leqslant t$ and $v_2(s,y)<2\rho<v_1(s,y)$
\begin{align}
e(t)\geqslant v_1(s,y)-v_2(s,y)=v_1(s,y) -2\rho+2\rho -v_2(s,y)\geqslant v_1(s,y)-2\rho\geqslant 0
\end{align}
thus $(s,y)\in I^+_{s,t}$. The same is true for the other cases.

Therefore we obtain the following expression for the difference $v_1-v_2$:
\begin{equation}\label{ints:Ist}
\begin{split}
v_1(t,x)-v_2(t,x)
=&\int_0^t \int_{\mathbb{T}^d} e^{-(2\beta +1)(t-s)}s(4\alpha(t-s),x,y)(f_\infty(v_1(s,y))-f_\infty(v_2(s,y)))\dd y \dd s\\
& +  \int_0^t \int_{I^+_{s,t}\cup I^-_{s,t}} e^{-(2\beta +1)(t-s)}s(4\alpha(t-s),x,y)(g_1(s,y)-g_2(s,y))\dd y \dd s.
\end{split}
\end{equation}
For the first integral in \eqref{ints:Ist} we note that $f_\infty$ is $1$-Lipschitz, thus
\begin{align*}
&\left|\int_0^t \int_{\mathbb{T}^d} e^{-(2\beta +1)(t-s)}s(4\alpha(t-s),x,y)(f_\infty(v_1(s,y))-f_\infty(v_2(s,y)))\dd y \dd s\right|\\
&\leqslant 
\int_0^t \int_{\mathbb{T}^d} e^{-(2\beta +1)(t-s)}s(4\alpha(t-s),x,y) e(t)\dd y\dd s
\leqslant t e(t).
\end{align*}

For the second integral in \eqref{ints:Ist} we note that $|g_i|\leqslant 1-2\rho$, then we first have
\begin{align*}
&\left|\int_0^t \int_{I^+_{s,t}\cup I^-_{s,t}} e^{-(2\beta +1)(t-s)}s(4\alpha(t-s),x,y)(g_1(s,y)-g_2(s,y))\dd y \dd s\right|\\
&\leqslant
2(1-2\rho)\int_0^t \int_{I^+_{s,t}\cup I^-_{s,t}} e^{-(2\beta +1)(t-s)}s(4\alpha(t-s),x,y)\dd y \dd s.
\end{align*}

Let $s\leqslant t$, since $v_0$ is regular on the level set $\{v_0=2\rho\}$ which is compact (since $\bbT^d$ is) and $\nabla v_0$ is a Lipschitz function, we can find $\delta,\eta>0$ such that on $\{v_0=2\rho\}+B_{\delta}(0)$, $|\nabla v_0(x)|>\eta$.
Using the second part of Lemma \ref{lem:estimatesv0}, and since $e(t)\leqslant C\sqrt{t}$, there exists $T>0$ such that for $s\leqslant t\leqslant T$, $I^+_{s,t}\subset \{v_0=2\rho\}+B_{\delta}(0)$ and on $I^+_{s,t}$, $|\nabla v_1(s)|>\eta/2$.

Since $I^+_{s,t}$ is compact and $\nabla v_1(s)\neq0$, by the implicit function theorem, we can find a finite cover by open balls $(B_i)_{1\leqslant i\leqslant N}$ centered on points on $I^+_{s,t}$ such that locally on each ball $B_i$, the level set $\{v_1(s,y)=2\rho\}$ is the graph of a function, e.g $y_1=\phi(y_2,\dots y_d)$. Note that since $\{v_0=2\rho\}$ is compact, $N$ is uniform in $s\leqslant T$, since by the lemma, we can make the cover of open balls on $\{v_0=2\rho\}$ and take their traces on $\{v_1(s,y)=2\rho\}$.
By the mean value theorem on the first coordinate $y_1$ of $v_1(s)$, we have $I^+_{s,t}\cap B_i\subset [-2e(t)/\nu,2e(t)/\nu] \times \Pi_1(B_i)$, where $\Pi_1$ is the projection along the first coordinate.
Thus,
\begin{align*}
\int_0^t &\int_{I^+_{s,t}\cap B_i} e^{-(2\beta +1)(t-s)}s(4\alpha(t-s),x,y)\dd y \dd s\\
&\leqslant 
\int_0^t  e^{-(2\beta +1)(t-s)}\frac{4e(t)}{\nu\sqrt{4\alpha (t-s)}}\int_{ \Pi_1(B_i)}s(4\alpha(t-s),0,(0,y_2,\dots y_d))\dd y_2\cdots \dd y_d \dd s\\
&\leqslant 
\int_0^t  e^{-(2\beta +1)(t-s)}\frac{2e(t)}{\nu\sqrt{\alpha (t-s)}} \dd s
\leqslant \frac{2e(t)t^{1/2}}{\nu\sqrt{\alpha}}.
\end{align*}
Thus 
\begin{align}
\int_0^t \int_{I^+_{s,t}} e^{-(2\beta +1)(t-s)}s(4\alpha(t-s),x,y)\dd y \dd s
&\leqslant \frac{2Ne(t)t^{1/2}}{\nu\sqrt{\alpha}}.
\end{align}

Since the same holds for $I^-_{s,t}$, we obtain, that for some constant $C>0$, and all  $t<\tau$
\begin{align}
|v_1(t,x)-v_2(t,x)|
&\leqslant (t+ Ct^{1/2})e(t)\leqslant (\tau+ C\tau^{1/2})e(\tau).
\end{align}
Then $e(\tau)\leqslant (\tau+ C\tau^{1/2})e(\tau)$, and taking $\tau$ small enough, we obtain $e(\tau)=0$ thus $v_1=v_2$ on $[0,\tau]\times \bbT^d$.
\end{proof}

\begin{remark}\label{rem:monotone}\textbf{Maximal and minimal solutions.} Another approach to existence of a solution to Equation \eqref{eq:limitePDE} is to use a monotone construction of solutions, which arises from a comparison principle close to the one developed in Proposition \ref{prop:subsuper_sol_G}. This was done initially in \cite{Carl89} and also in \cite{deguchi}  Define $\overline{h_\infty}$ the right continuous version of $h_\infty$ (Equation \ref{eq:hinfty}) by
\begin{align}
\overline{h_\infty}(q)&:=
-\mathbbm{1}_{q< -2\rho}+q\mathbbm{1}_{-2\rho\leqslant q< 2\rho}+\mathbbm{1}_{q\geqslant2\rho}.
\end{align}
Note that $h_\infty$ and $\overline{h_\infty}$ are non decreasing (recall that $2\rho\in[0,1]$). 
Recall that $(S_t)$ is the semigroup on $L^1(\mathbb{T}^d)$ associated to $-2\alpha\Delta $, we denote
\begin{align}
\underline{F}(v)(t,x):=e^{-(2\beta+1)t}S_t v_0(x)+\int_0^t e^{-(2\beta+1)(t-s)}S_{t-s}[h_\infty(v(s,\cdot))](x)\dd s\\
\overline{F}(v)(t,x):=e^{-(2\beta+1)t}S_t v_0(x)+\int_0^t e^{-(2\beta+1)(t-s)}S_{t-s}[\overline{h_\infty}(v(s,\cdot))](x)\dd s.
\end{align}
Then, fixed points of the  maps above are mild solutions of these two formulations of our subdifferential inclusion:
\begin{align}
\partial_tv(t,x)-2\alpha\Delta v(t,x)+(2\beta +1) v(t,x)&=h_\infty(v(t,x))\\
\partial_tv(t,x)-2\alpha\Delta v(t,x)+(2\beta +1) v(t,x)&=\overline{h_\infty}(v(t,x))
\end{align}
Since  $h_\infty$ (resp. $\overline{h_\infty}$) is non decreasing and that $h_\infty(p)\leqslant \overline{h_\infty}(p)$ for all $p\in[-1,1]$, we have that, for $u,v$ two functions such that $-1\leqslant v\leqslant u\leqslant 1$,
\begin{enumerate}
\item $\underline{F}(v)(t,x)\leqslant \underline{F}(u)(t,x)$
\item $\overline{F}(v)(t,x)\leqslant \overline{F}(u)(t,x)$
\item $\underline{F}(u)(t,x)\leqslant \overline{F}(u)(t,x)$
\end{enumerate}
We define the sequences $(V^n)_n$ and $(W^n)_n$ of functions on $\mathbb{R}^+\times\mathbb{T}^d$:
$V^0(t,x)=1$, $W^0(t,x)=-1$ and for all $n\geqslant 1$
\begin{align}
V^n(t,x)&=e^{-(2\beta+1)t}S_t v_0(x)+\int_0^t e^{-(2\beta+1)(t-s)}S_{t-s}[\overline{f}(V^{n-1}(s,\cdot))](x)\dd s\\
W^n(t,x)&=e^{-(2\beta+1)t}S_t v_0(x)+\int_0^t e^{-(2\beta+1)(t-s)}S_{t-s}[\underline{f}(W^{n-1}(s,\cdot))](x)\dd s.
\end{align}
Thus, for $-1\leqslant v_0(x)\leqslant 1$, we can prove by induction that the sequences $(V^n)$ and $(W^n)$ satisfy, for all $n$
\begin{equation}
-1\leqslant W^1\leqslant W^2\cdots \leqslant W^n\leqslant V^n\leqslant \cdots \leqslant V^2\leqslant V^1\leqslant 1.
\end{equation}
By  a compactness and monotony argument, one can prove that $(W^n)$ and $(V^n)$ converge to functions $w$ and $w$ which are mild solutions of the subdifferential inclusion. These are the minimal and maximal solutions of the subdifferential inclusion, in the sense that any other solution (Definition \ref{def:subdiff}) must be bounded below by $w$ and above by $v$. Uniqueness follows if one can prove that $w=v$ and is proved usually (e.g.  in \cite{deguchi}) along the lines of Proposition \ref{prop:v_t-v_0}.
\end{remark}

\section{Convergence of the discrete PDE}\label{sec:conv_discrPDE}
%


In analogy with the continuous setting, we define $v^N=2u^N-1$ where $u^N$ is the solution of the discretized Equation \eqref{eq:discretPDE} and $H(i,v^N)=2G(i,\frac{v^N+1}2)+v^N(i)$. 

In such a way \eqref{eq:discretPDE} becomes
\begin{equation}\label{eq:discretPDE-H}
\begin{cases}\partial_tv^N(t,i)=2\alpha N^2 \Delta u^N(t,i)-(2\beta+1)v^N(t,i)+H(i,v^N)\\
v^N(0,i)=2u^N_0(i)-1,\end{cases}
\end{equation}
The main goal of this section is to prove the following result which states the convergence of $v^N$.

\begin{theorem}\label{thm:discrconv}
Let $v^N$ be the solution of \eqref{eq:discretPDE-H}. Then $(v^N)$ is pre-compact for the uniform convergence on each compact sets of $\bbT^d\times ]0, +\infty[$ and any accumulation points  $v_\infty$ is a solution of \eqref{eq:limitePDE-3}.
In particular, whenever the solution of \eqref{eq:discretPDE-H} is a.e. unique, $v_\infty$ is also (the) mild solution of \eqref{eq:limitePDE} and the whole sequence $v^N$ converges to $v_\infty$, uniformly on all compact sets of $\bbT^d\times ]0, +\infty[$.
\end{theorem}

To prove Theorem \ref{thm:discrconv} we need some technical results. 
Let consider the semigroup of the discrete Laplacian $\frac{1}{2}N^2\Delta^N$ on $\Tn$ and $\mathbb Z^d$, denoted by $s^N(t,i,j)$ and $s_0^N(\gamma t,i,j)$ respectively. In particular we have that $s^N(t,i,j)=p_N(t, i,j)$, where $p_{N}(t,i,j)$ is heat kernel of discrete Laplacian on the discrete torus, cf. \eqref{def:discrLaplacian}.

For any $\lambda, \gamma\ge 0$ and $f:\frac1N \Tn \to \mathbb R$, we let $(S^{N, \lambda, \gamma}_t)$ be the semigroup defined by
\begin{equation}\label{eq:defS^N}
S^{N, \lambda, \gamma}_t f(x)=\sum_{y\in \frac1N \Tn}e^{-\lambda t}s^N(\gamma t,Nx,Ny)f(y)
=
\sum_{y\in \frac1N \mathbb Z^d}e^{-\lambda t}s_0^N(\gamma t,Nx,Ny)\tilde f(y)
\end{equation}
where, as in \eqref{lapl_semigroup}, $\tilde f$ is the periodic extension of $f$ on $\frac1N \mathbb Z^d$.

\smallskip

In the remaining part of this article we will consider $S^{N, \lambda, \gamma}_t f(x)$ with $f\in \mathcal C(\bbT^d)$. In that case we mean that the function $f$ is restricted on $\frac1N\Tn \subset \bbT^d$, which is equivalent to consider $f^N(x):=f(\lfloor Nx \rfloor /N )$. We observe that if $f$ is also Lipschitz, then $\|f-f^N\|_{\bbT^d}\le \frac{c}{N^d}$ for some $c>0$ and $\|f^N\|_{\bbT^d}\le \|f\|_{\bbT^d}$. Then, with the same extension to $\bbT^d$ for $s^N$ we can write, for any $x\in \bbT^d$
\begin{equation}\label{eq:defS^N_2}
S^{N, \lambda, \gamma}_t f(x)=\int_{\bbT^d}e^{-\lambda t}N^d s^N(\gamma t,Nx,Ny)f^N(y)\dd y.
\end{equation}

By a slight abuse of notation, we still denote by $v^N$ the linear interpolation on $\bbT^d$ such that $v^N(t,\frac{i}N)=v^N_i(t)$. We also redefine the function $H$ on the torus $\bbT^d$ by the linear interpolation such that $H(\frac{i}N,v^N)=H(i,v^N)$ and we define $H^N$ as $f^N$ in \eqref{eq:defS^N_2}.

\begin{definition}
Let $N\in \mathbb N$ and $v_0^N\in L^\infty(\bbT^d)$. We say that $(v^N(t,x), t\ge 0, x\in\bbT^d)$ is a mild solution of
 \eqref{eq:discretPDE-H} if
\begin{itemize}
\item  for any $N$, $v^N$ is continuous from $\mathbb{R}_+^*$ to $L^{\infty}(\bbT^d)$,
\item for all $t>0$ and $x\in\bbT^d$ 
\begin{equation}\label{def:mild_sol_dis}
v^N(t,x)=S^{N,2\beta+1, 4N^2\alpha}_t (v_0^N)(x)+\int_0^t S^{N,2\beta+1, 4N^2\alpha}_{t-s}\Big[H^N(\cdot,v^N(s,\cdot))\Big](x)\dd s\,.
\end{equation}
\end{itemize}
\end{definition}
Let $u^N$ be the unique solution of \eqref{eq:discretPDE}, so that $v^N=2u^N-1$ satisfies \eqref{eq:discretPDE-H}. Of course, for any $N$ the solution $v^N$ of \eqref{def:mild_sol_dis} exists and it is unique. 

The proof of Theorem \ref{thm:discrconv} is based on the representation of $v^N$ as in \eqref{def:mild_sol_dis}. We define $\widetilde v^N$ as a slight modification of \eqref{def:mild_sol_dis}, that is,
\begin{equation}
\widetilde v^N(t,x):=S^{2\beta+1, 4\alpha}_t (v_0)(x)+\int_0^t S^{2\beta+1, 4\alpha}_{t-s}\Big[H(\cdot,v^N(s,\cdot))\Big](x)\dd s\, .
\end{equation}
\begin{lemma}\label{Lemma_vtildev}
For any $\tau>0$
\begin{equation}\label{eq:dist_b_v}
\lim_{N\to +\infty}\|\widetilde v^N-v^N\|_{[1/\tau,\tau]\times\bbT^d}=0.
\end{equation}
\end{lemma}
\begin{proof} Let $\tau>0$, then
\begin{multline}\label{int:1}
\sup_{t\in [\frac1\tau,\tau], \,x\in \bbT^d} \Big|\widetilde v^N(t,x)-v^N(t,x)\Big |\le 
\sup_{t\in [\frac1\tau,\tau], \,x\in \bbT^d} \Big|S^{N, \lambda, N^2\gamma}_t v_0^N(x)-S^{\lambda, \gamma}_t v_0(x)\Big|\\ +
\sup_{t\in [\frac1\tau,\tau], \,x\in \bbT^d} 
\Big|\int_0^t \Big\{S^{N,2\beta+1, 4N^2\alpha}_{t-s}\Big[H^N(\cdot,v^N(s,\cdot))\Big](x)- S^{2\beta+1, 4\alpha}_{t-s}\Big[H(\cdot,v^N(s,\cdot))\Big](x) \Big\}\dd s\Big|.
\end{multline}
We show that the right hand side of \eqref{int:1} converges to $0$. We detail the convergence of the second term, which is more delicate. The argument can be adapted to the first term by using Assumption \ref{ass2} which ensures that $v_0^N$ converges to $v_0$ in $C(\bbT^d)$.

We fix $\epsilon\in(0, \frac1\tau)$ and we get that 
\begin{multline}\label{int:2}
\Big|\int_0^t S^{N,2\beta+1, 4N^2\alpha}_{t-s}\Big[H^N(\cdot,v^N(s,\cdot))\Big](x)- S^{2\beta+1, 4\alpha}_{t-s}\Big[H(\cdot,v^N(s,\cdot))\Big](x)\dd s\Big|\\ \le
\Big|\int_0^{t-\epsilon} S^{N,2\beta+1, 4N^2\alpha}_{t-s}\Big[H^N(\cdot,v^N(s,\cdot))\Big](x)- S^{2\beta+1, 4\alpha}_{t-s}\Big[H(\cdot,v^N(s,\cdot))\Big](x)\dd s\Big| + C\epsilon \,,
\end{multline}
where we used that Lemma \ref{lemma:bound:H} which implies that $H(\cdot,v^N)$ is bounded by $1$ uniformly on $N$. 
The integral on the right hand side of \eqref{int:2} is bounded from above by 
\begin{align*}
&\begin{aligned}\int_0^{t-\epsilon} e^{-(2\beta+1)(t-s)} \Big|\int_{\bbT^d}(N^d s^N(4N^2\alpha (t-s),&Nx,Ny)\\&-s(4\alpha (t-s),x,y))H^N(y,v^N(s,y))\dd y\Big|\dd s\end{aligned}\\
&+\int_0^{t-\epsilon} e^{-(2\beta+1) (t-s)}\int_{\bbT^d} s(4\alpha (t-s),x,y)\Big|H^N(y,v^N(s,y))-H(y,v^N(s,y))\Big|\,\dd y \dd s.
\end{align*}
Since $\sup_{s\ge 0, \, y\in \bbT^d}\Big|H^N(y,v^N(s,y))-H(y,v^N(s,y))\Big| \le \frac1N$, the second integral is smaller than $\frac{c_{\alpha,\beta}}{N}$. 
For the first integral, we first use that $H(\cdot,v^N)$ is bounded by $1$ uniformly on $N$ and then we operate the change of variable $u=t-s$ which gives that it is bounded from above by 
\begin{align}\label{int:last}
\int_\epsilon^t e^{-(2\beta+1)u} \int_{\bbT^d} \Big|\,N^d s^N(4N^2\alpha u,Nx,Ny)-s(4\alpha u,x,y)\Big| \,\dd y\,\dd u.
\end{align}
We now use the local central limit theorem (cf. Theorem 2.1.1 and (2.5) in \cite{lawler2010}):
let $\rho$ be the Gaussian Kernel and $\rho(u,x,y)=\frac{1}{u^{d/2}}\rho\Big(\frac{x-y}{u^{1/2}}\Big)$, then 
\begin{equation} \label{Gnedenko}
\psi_{\epsilon,\tau,N}:=\sup_{u\in [\epsilon,\tau], \, y\in \frac1N \mathbb Z^d}\Big|N^d p(u N^2, N x, Ny)  - \rho(u,x,y)\Big|\xrightarrow[]{N\to +\infty}0.
\end{equation} 
We also observe that by symmetry the supremium in \eqref{Gnedenko} is independent of $x$.
Moreover, by Proposition 2.4.6 in \cite{lawler2010} we have that there exist $c_1,c_2>0$ independent of $x,y,u$ such that 
\begin{equation} \label{LD_ub}
\Big|N^d p(uN^2, N x, Ny)  - \rho(u,x,y)\Big|\le \frac{c_1}{u^{\frac{d}{2}}}e^{-c_2 \frac{\|x-y\|^2}{u}}.
\end{equation}
In such a way, for any $M>0$ fixed we write $\bbT^d=B_x(M)\cup B_x(M)^c$ and we get that \eqref{int:last} is smaller than 
\begin{equation}
c_d \int_\epsilon^t e^{-(2\beta+1)u} \Big( M^d \psi_{N, \epsilon,\tau}+\frac{c_1}{M^d u^{\frac{d}{2}}}e^{-c_2 \frac{M^{2d}}{u}}\Big)\, \dd s\le 
C_d\Big( M^d \psi_{N, \epsilon,\tau}+\frac{1}{M^d}\Big),
\end{equation}
where $c_d,C_d>0$ are two positive constants that depend only on the dimension $d$.

We conclude that the right hand side of \eqref{int:2} is bounded by 
$\frac{c_{\alpha,\beta}}{N}+cM^d \psi_{N, \epsilon,\tau}+\frac{c}{M^d} +C\epsilon$,
uniformly on $x\in \bbT^d$ and $t\in [\frac{1}{\tau},\tau]$. Therefore, by taking the limit on $N\to +\infty$ and then on $M\to +\infty$ and $\epsilon\to 0$ we conclude the proof.
\end{proof}

\begin{proof}[Proof of Theorem \ref{thm:discrconv}]
 We control $\widetilde v^N$ to get the convergence of $v^N$. We observe that since $H(\cdot,v^N(s,\cdot))$ is uniformly bounded so that by Proposition \ref{prop:mildprop} $\widetilde v^N$ is uniformly bounded in $N$, uniformly continuous on $[1/\tau, \tau]\times\bbT^d$, and the modulus of continuity only depends on $\tau>1$. By the Ascoli-Arzela Theorem, the sequence $(\widetilde v^N)_N$ is pre-compact on $C([1/\tau,\tau]\times \bbR^d)$ and therefore, by Lemma \eqref{Lemma_vtildev}, $(v^N)$ also. By a diagonal argument, we can extract from $(v^N)$ a subsequence converging uniformly to a limit $v_\infty$ in $C(]0,\infty[\times \bbR^d)$, uniformly on each compact. 
%

Using Corollary \ref{cor:conv_rho}, we can adapt the argument used in the proof of Proposition \ref{prop:existence_sol_PDE} (\ref{eq:lim_inf_sup_h}--\ref{eq:lim_inf_sup_w}) to get that each accumulation point $v_\infty$ is a mild solution of \eqref{eq:limitePDE-3}, we omit the details. 
\end{proof}

\subsection{Proof of Theorem \ref{thm:convergence}}

Theorem \ref{thm:convergence} is now a consequence of Theorems \ref{thm:intermediateGoal} and \ref{thm:discrconv}.

\appendix 


\section{Concentration inequalities}
We follow the definitions in Jara and Menezes \cite{JM20} and \cite{JM18} and Boucheron, Lugosi, Massart \cite{BLM16}, Section 2.3. We omit the proofs since there are present in the references.
\begin{definition}[\cite{JM20} and  \cite{BLM16}, Section 2.3] 
Let $X$ be a real random variable. $X$ is said to be sub-Gaussian with variance parameter $\sigma^2$ if, for all $t\in\bbR$
\begin{equation}
\psi_{X}(t):=\log \mathbb{E}(\exp(tX))\leqslant \sigma^2\frac{t^2}2.
\end{equation}
We denote $\cG(\sigma^2)$ the set of real sub-Gaussian random variables with variance parameter $\sigma^2$.
\end{definition}

\begin{proposition}\label{prop:subgauss}[\cite{BLM16} and \cite{JM18}, Proposition F.7]
The following statements are equivalent:
\begin{enumerate}
\item $X\in \cG(\sigma^2)$
\item For any $t>0$, $\mathbb{P}(|X|>t)\leqslant 2\exp(-\frac{t^2}{2\sigma^2})$
\item $\mathbb{E}(\exp(\gamma X^2))\leqslant 3$ for all $0\leqslant \gamma< \frac1{4\sigma^2}$.
\end{enumerate}
\end{proposition}

Let us complete our family of inequalities:
\begin{lemma}\label{lem:SGprod}[\cite{JM18}, Proposition F.8]
Let $X\in \cG(\sigma_1^2)$ and $Y\in\cG(\sigma_2^2)$, then for all $0\leqslant \gamma< \frac1{4\sigma_1\sigma_2}$,
\begin{equation*}
\mathbb{E}(\exp(\gamma XY))\leqslant 3
\end{equation*}
\end{lemma}

\begin{lemma}\label{lem:SGsum}[\cite{JM18}, Proposition F.12]
Let $X_1,\dots, X_n$ be random variables with $X_i\in\cG(\sigma_i^2)$, such that there is a partition of $K$ subsets $P_1,\dots,P_K$ of $\{1, \ldots,n\}$ each containing $L$ variables that $\Sigma_k=\sigma(X_i,i\in P_k)$ are independent $\sigma$-algebra then, for all real $\alpha_1,\dots,\alpha_n$, the random variable $Y=\sum_i\alpha_i X_i$ is sub-Gaussian with variance parameter $L\sum_i\alpha_i^2\sigma_i^2$.
\end{lemma}
Note that if $L=1$, the variables are independent.

\begin{lemma}[Hoeffeding Inequality,  \cite{BLM16}, Section 2.3]\label{lem:hoeff}
Let $X$ be a bounded random variable with $X\in[a,b]$, then $X-\mathbb{E}X\in\cG\left(\frac{(b-a)^2}4\right)$.
\end{lemma}

\section{Controls for the non-linearities}
In this section, we collect some results about the specifics of our model. Recall the notations: for $\eta\in\{0,1\}^{\Tn}$
\begin{align*}
c_0^+(\eta)&:=\1{\rho_0(\eta)\geqslant 1-\frac{\kappa_N}{K_N}},&
c_0^-(\eta)&:=\1{\rho_0(\eta)\leqslant \frac{\kappa_N}{K_N}},
\end{align*}
and $\kappa_{N}=\kappa_N(T):= \min\Big\{\lfloor K_N T \rfloor-1; \lfloor K_N(1-T)\rfloor \Big\}$.
For any function $u=(u_i)_{i\in \Tn}$, we define
$\gu_{u}(\dd\eta)=\gu_{u}^N(\dd\eta):=\bigotimes_{i\in \Tn}\mathrm{B}\big(u_i)$ where $B(u_i)$ denote a Bernoulli distribution with parameter $u_i$.
and we let $c_0^+(u)$ and $c_0^-(u)$ be the expectations of $c_0^+(\eta)$ and $c_0^-(\eta)$ under $\gu_u$.  We set $(\tau_i\eta)_j=\eta_{i+j}$, and likewise $\tau_i$ acts on $u$.
Then, 
\begin{align}\label{def:Gbis}
G(u)&:= (1-u_0)c_0^+(u)- u_0c_0^-(u)\\\notag
&=(1-u_0)\mathbb{P}_{\gu_u}\left[\rho_0(\eta)\geqslant 1-\frac{\kappa_N}{K_N}\right] -u_0\mathbb{P}_{\gu_u}\left[\rho_0(\eta)\leqslant\frac{\kappa_N}{K_N}\right]
\end{align}
and $G(i, u):=G(\tau_i u)$.
We start with some results on the non linearity $G$:
\begin{proposition}\label{prop:G1}
Let $u, v\in[0,1]^{\Tn}$ such that $u_i\geqslant v_i$ for all $i\in\cV_N$ and $u_0=v_0$. Then $G(0,u)\geqslant G(0,v)$.
\end{proposition}

\begin{proof}
We construct a coupling between $\gu_u$ and $\gu_v$: let $(U_i)$ be independent and identically distributed random variables uniform on $[0,1]$. Define $\eta_i=\1{U_i\leqslant u_i}$ and $\eta'_i=\1{U_i\leqslant v_i}$. We have $\eta_i\geqslant \eta'_i$ for all $i\in \cV_N$, therefore $\rho_0(\eta)\geqslant \rho_0(\eta')$. This proves that $c^+_0(u)\geqslant c^+_0(v)$ and $c^-_0(u)\leqslant c^-_0(v)$. The results follows since $u_0=v_0$.
\end{proof}

For $p\in[0,1]$, $T\in[0,1]$,  $K\in\mathbb{N}^*$, we let $\kappa(K,T)=\min\Big\{\lceil K T \rceil-1; \lfloor K (1-T)\rfloor \Big\}\leqslant K/2$, 
and we define $g_{K}(p)$ as 
\[
g_K(p):=(1-p)\mathbb{P}_p\big[X>K-\kappa(K,T)\big] -p\mathbb{P}_p\big[X\leqslant \kappa(K,T)\big]
\]
 where $X$ is a random variable with binomial distribution with parameter $(K,p)$ under $\mathbb{P}_p$. In particular, we have that $g_{K_N}(p)=G(i, u)$ for $u(t, i)=p$ for all $i\in\Tn$.  Note that $g_K$ is $C^{\infty}([0,1])$. We also have
\begin{equation}
\label{eq:gK_ext}
\begin{split}
g_K(p)
&=(1-p)\mathbb{P}_{1-p}\big[X<\kappa(K,T)\big] -p\mathbb{P}_p\big[X\leqslant \kappa(K,T)\big]\\
&=\sum_{k=0}^{\kappa(K,T)-1}\binom{K}{k}\left[(1-p)^{k+1}p^{K-k}-p^{k+1}(1-p)^{K-k}\right]\\&\qquad-\binom{K}{\kappa(K,T)}p^{\kappa(K,T)+1}(1-p)^{K-\kappa(K,T)}.
\end{split}
\end{equation}

We recall $g_{\infty}(p):=(1-p)\1{1-p<p_0(T)}-p\1{p\leqslant p_0(T)}$, where $p_0(T)=\min(T,1-T)$.

The following proposition estimates the convergence of $g_K$ to $g_\infty$, in particular we prove that close to $p=0$  (resp. $p=1$), $g_K$ is negative (resp. positive).
\begin{proposition}\label{prop:g_K1}
For all $K\in\mathbb{N}^*\cup \{\infty\}$, we have $g_K(0)=g_K(1)=0$.
For all $K\in\mathbb{N}^*$ and $p\in[0,1]$, we have, for $|p-p_0(T)|>\frac1K$ and $|(1-p)-p_0(T)|>\frac1K$
\begin{equation}\label{eq:g_K1}
|g_K(p)-g_{\infty}(p)|\leqslant 2e \exp(-2K(p_0(T)-p)^2)+ 2e \exp(-2K(p_0(T)-(1-p))^2)
\end{equation}
In particular, for any $0<\delta<\frac1{4e}$, $T\in[4\delta,1-4\delta]$, and $K>\frac{|\log(\delta)|}{4\delta^2}$, we have that $g_K(2\delta)<-\delta$ and $g_K(1-2\delta)>\delta$.
\end{proposition}

\begin{proof}
We consider $g_K$ written as in \eqref{eq:gK_ext}. Then, the values at $p=0$ and $p=1$ are obvious.

Note that, under $\mathbb{P}_p$, $\frac{X}{K}-p$ converges to $0$  (in $L^2(\Omega_N)$) and is sub-Gaussian with variance parameter $\frac1{4K}$, thus, for any $t>0$, 
\begin{align*}
\mathbb{P}_p\left(\left|\frac{X}{K}-p\right|>t\right)\leqslant 2 \exp(-2K t^2).
\end{align*}
We also have that 
\begin{align*}
0\leqslant p_0(T)-\frac{\kappa(K,T)}{K}\leqslant\frac1K.
\end{align*}
Then, for $p>p_0(T)$,
\begin{equation*}
\begin{split}
\mathbb{P}_p\big[X\leqslant\kappa(K,T)\big]
&=\mathbb{P}_p\left[p-\frac{X}{K}\geqslant p-p_0(T)+p_0(T)-\frac{\kappa(K,T)}K\right]\\
&\leqslant \mathbb{P}_p\left[p-\frac{X}{K}\geqslant p-p_0(T)\right]
\leqslant 2\exp(-2K(p-p_0(T))^2),
\end{split}
\end{equation*}
and for $p<p_0(T)-\frac1K$,
\begin{equation*}
\begin{split}
\mathbb{P}_p[X>\kappa(K,T)]
&=\mathbb{P}_p\left[\frac{X}{K}-p> \frac{\kappa(K,T)}K-p_0(T)+p_0(T)-p\right]\\
&\leqslant \mathbb{P}_p\left[\frac{X}{K}-p\geqslant p_0(T)-p-\frac1K\right]
\leqslant 2\exp(-2K(p_0(T)-p-1/K)^2)\\
&\leqslant 2e^{2(p_0(T)-p)}\exp(-2K(p_0(T)-p)^2)
\leqslant 2e \exp(-2K(p_0(T)-p)^2).
\end{split}
\end{equation*}
Thus, we have the following, for $|p-p_0(T)|>\frac1K$,
\begin{align*}
\left|\mathbb{P}_p[X\leqslant\kappa(K,T)]-\1{p\leqslant p_0(T)}\right|< 2e \exp(-2K(p_0(T)-p)^2).
\end{align*}
The results follows with the same estimates with $(1-p)$ instead of $p$. This prove \eqref{eq:g_K1}.

Let $\delta>0$ and set $p=2\delta$, and $T\in[4\delta,1-4\delta]$, then $|2\delta-p_0(T)|\leqslant 2\delta$ for $K>\frac1{2\delta}$ and we have the same for $1-p=1-2\delta$. Applying the result, we have 
\begin{align*}
|g_K(2\delta)-g_{\infty}(2\delta)|&=|g_K(2\delta)+2\delta|\leqslant 4e \exp(-8K\delta^2)
\end{align*}
Then, if $4e \exp(-8K\delta^2)<\delta$, we get the result. This happens if $\exp(-8K\delta^2)<\delta^2$ and $\delta<\frac1{4e}$, which gives the condition on $K$.
\end{proof}

\begin{proposition}\label{prop:conv_rho}
Let $v:\bbT^d\to [0,1]$ be a continuous fixed density on the torus. For any $i\in\Tn$ we let
$u_i:=v_{i/N}$. 
Then as $i/N\to x$, we have that $u_i$ converges to $v_x$. Let $\gu_u=\bigotimes_{i\in \Tn}\mathrm{B}\big(u_i)$. Then, 
$\rho_0(\eta)$ converges to $v_0$ in probability.
\end{proposition}
\begin{proof}
We use the coupling introduced in Lemma \ref{prop:G1}: we let $(U_i)$ be i.i.d uniform random variables on $[0,1]$ so that under $\gu_u$ we have that $\rho_0(\eta)$ and $\frac{1}{K_N}\sum_{i\in\cV_N}\1{U_i<u_i}$ are equal in law. By the Tchebychev inequality, 
we have that for any $\epsilon>0$
\[
\bbP\Bigg(\Big|\frac{\sum_{i\in\cV_N}\1{U_i<u_i}}{K_N}-\frac{\sum_{i\in\cV_N}u_i}{K_N}\Big|>\epsilon\Bigg)\le \frac{1}{4K_N\epsilon^2}.
\]
Moreover, the sequence $K_N^{-1}\sum_{i\in\cV_N}u_i$ converges to $v_0$ because for any $i\in \cV_N$, $i/N\to 0$.
\end{proof}
We have the following corollary.

\begin{corollary}\label{cor:conv_rho}
Let $u=(u_i)_{i\in \Tn}$ as in Proposition \ref{prop:conv_rho}. Then, $G(u)$ and $g_{K_N}(v_0)$ converge both to $g_\infty(v_0)$ as $N\to +\infty$.
\end{corollary}

Recall that, for $q\in[-1,1]$,
\begin{align*}
h_K(q)
&=2g_K\left(\frac12(q+1)\right)+q\\
&=(1-q)\mathbb{P}_{\frac{1-q}2}\big[X<\kappa(K,T)\big] -(1+q)\mathbb{P}_{\frac{1+q}2}\big[X\leqslant \kappa(K,T)\big]+q.
\end{align*}
 Recall the critical parameter $\rho:=\rho(T)=\left|T-\frac12\right|=\frac12-p_0(T)\in\left[0,\frac12\right]$. Note that $h_K$ converges pointwise to the function  $h_\infty(q)=2g_{\infty}\left(\frac{1+q}2\right)+q
=-\mathbbm{1}_{q\leqslant -2\rho}+q\mathbbm{1}_{-2\rho<q\leqslant 2\rho}+\mathbbm{1}_{q>2\rho}$. The points $q=\pm2\rho$ are the discontinuities of $h_\infty$, and compare the Lemma  to the fact that $2\rho=h_{\infty}(2\rho^-)$ and $1=h_{\infty}(2\rho^+)$. A similar estimate holds at $q=-2\rho$.

\begin{lemma}\label{lemma:h_Kbounded}
For all $q\in[-1,1]$, 
$|h_K(q)|\leqslant 1$.
Moreover, for all $\epsilon>0$ such that $2\rho-\epsilon>0$, there is $K_0>0$ such that for $K\geqslant K_0$, 
\begin{align*}
2\rho-2\epsilon=h_{\infty}(2\rho^-)-2\epsilon\leqslant &h_K(q)\leqslant h_{\infty}(2\rho^+)=1&&\text{for $q\in[2\rho-\epsilon,2\rho+\epsilon]$}\\
-1=h_{\infty}(-2\rho^-) \leqslant &h_K(q)\leqslant h_{\infty}(-2\rho^+)+2\epsilon= -2\rho+2\epsilon&&\text{for $q\in[-2\rho-\epsilon,-2\rho+\epsilon]$}
\end{align*}
\end{lemma}

\begin{proof}
We start by observing that for all $q\in[-1,1]$, $\mathbb{P}_{\frac{1-q}2}\big[X<\kappa(K,T)\big]=\mathbb{P}_{\frac{1+q}2}\big[X>K-\kappa(K,T)\big]$. We also have that
\begin{equation*}
1=\mathbb{P}_{\frac{1+q}2}\big[X\leqslant\kappa(K,T)\big] +\mathbb{P}_{\frac{1+q}2}\big[\kappa(K,T)<X\leqslant K-\kappa(K,T)\big]+\mathbb{P}_{\frac{1+q}2}\big[X>K- \kappa(K,T)\big].
\end{equation*}
Thus, using the definiton of $h_K$, we get
\begin{align*}
h_K(q)=&-\mathbb{P}_{\frac{1+q}2}\big[X\leqslant\kappa(K,T)\big] +q\mathbb{P}_{\frac{1+q}2}\big[\kappa(K,T)<X\leqslant K-\kappa(K,T)\big]\\&+\mathbb{P}_{\frac{1+q}2}\big[X>K- \kappa(K,T)\big].
\end{align*}
This gives us the result.
Indeed, for $q\in[0,1]$, we have
\begin{align*}
h_K(q)&\leqslant q\mathbb{P}_{\frac{1+q}2}\big[\kappa(K,T)<X\leqslant K-\kappa(K,T)\big]+\mathbb{P}_{\frac{1+q}2}\big[X>K- \kappa(K,T)\big]\\
&\leqslant \mathbb{P}_{\frac{1+q}2}\big[\kappa(K,T)<X\big]\leqslant 1
\end{align*}
and 
\begin{align*}
h_K(q)&\geqslant -\mathbb{P}_{\frac{1+q}2}\big[X\leqslant\kappa(K,T)\big] +\mathbb{P}_{\frac{1+q}2}\big[X>K- \kappa(K,T)\big]\\
&\geqslant \mathbb{P}_{\frac{1-q}2}\big[X<\kappa(K,T)\big]-\mathbb{P}_{\frac{1+q}2}\big[X\leqslant\kappa(K,T)\big]
\geqslant  -\mathbb{P}_{\frac{1+q}2}\big[X=\kappa(K,T)\big].
\end{align*}
The last inequality comes form the fact that, by a coupling argument, $p\mapsto  \mathbb{P}_{p}\big[X<\kappa(K,T)\big]$ is decreasing on $[0,1]$, and since $q\geqslant 0$, $\frac{1-q}2\leqslant\frac{1+q}2$.

In particular, for $q\in[2\rho-\epsilon,2\rho+\epsilon]$ such that $2\rho-\epsilon >0$, we have the same upper bound as before for $h_K(q)$, and for the lower bound: 
\begin{align*}
h_K(q)&\geqslant -\mathbb{P}_{\frac{1+q}2}\big[X\leqslant\kappa(K,T)\big] +q\mathbb{P}_{\frac{1+q}2}\big[X> \kappa(K,T)\big]\\
&\geqslant  -\mathbb{P}_{\frac{1+q}2}\big[X\leqslant\kappa(K,T)\big] + (2\rho-\epsilon)\left(1-\mathbb{P}_{\frac{1+q}2}\big[X\leqslant\kappa(K,T)\big]\right)\\
 &\geqslant  2\rho-\epsilon-2\mathbb{P}_{\frac{1+q}2}\big[X\leqslant\kappa(K,T)\big].
\end{align*}
From the proof of Proposition \ref{prop:g_K1}, we have that, for $K\geqslant K_0$:
\begin{align*}
\mathbb{P}_{\frac{1+q}2}\big[X\leqslant\kappa(K,T)\big]\leqslant 2\exp\left( -K(q+2\rho)^2/2\right)\leqslant 2e^{-2K_0\rho^2}
\end{align*}
Choosing $K_0$ large enough such that the right hand side is less than $\epsilon/2$, we get the result.

For $q\in[-1,0]$, the proof is completely similar.
%

\end{proof}

\begin{lemma}\label{lemma:bound:H}
Let $u=(u_i)_{i\in \Tn}$, with $u_i\in [0,1]$ and let $v=2u-1$. Then, $|H(i,v)|\le 1$, uniformly on $i\in \Tn$.
\end{lemma}
\begin{proof}
We recall \eqref{def:G}, in particular that $G(i,u)=G(\tau_i u)$. So that we only prove that $|H(v)|\le 1$, where $H(v)=G((v+1)/2)+v_0$. The proof is similar to Lemma \ref{lemma:h_Kbounded}. Indeed, again by \eqref{def:G} we have that
\begin{align*}
2G\left(\frac{v+1}{2}\right)+v_0 &=c_0^+\left(\frac{v+1}{2}\right)(1-v_0)-c_0^-\left(\frac{v+1}{2}\right)(1+v_0)+v_0\\
&=c_0^+\left(\frac{v+1}{2}\right)-c_0^-\left(\frac{v+1}{2}\right)+v_0\left(1-c_0^+\left(\frac{v+1}{2}\right)-c_0^-\left(\frac{v+1}{2}\right)\right).
\end{align*}
We observe that, by definition, $1-c_0^+\left(\frac{v+1}{2}\right)-c_0^-\left(\frac{v+1}{2}\right)>0$.
This implies that $ -1\le H(v)\le 1$.
\end{proof}

\providecommand{\bysame}{\leavevmode\hbox to3em{\hrulefill}\thinspace}
\providecommand{\MR}{\relax\ifhmode\unskip\space\fi MR }
\providecommand{\MRhref}[2]{%
  \href{http://www.ams.org/mathscinet-getitem?mr=#1}{#2}
}
\providecommand{\href}[2]{#2}

\end{document}